\documentclass[11pt,letterpaper]{amsart}

\usepackage[backend=biber,autocite=inline]{biblatex}
\DeclareFieldFormat{postnote}{#1}
\DeclareFieldFormat{multipostnote}{#1}

\DeclareSourcemap{
  \maps[datatype=bibtex]{
    \map{
      \step[fieldset=doi, null]      
      \step[fieldset=url, null]      
      \step[fieldset=urlyear, null]  
      \step[fieldset=issn, null]
      \step[fieldset=isbn, null]
    }
  }
}

\usepackage{accents}

\usepackage{geometry}

\usepackage{comment} 

\usepackage{hyphenat}

\usepackage{manfnt} 
\reversemarginpar 

\usepackage[svgnames]{xcolor}

\usepackage{url}

\usepackage{footmisc}

\usepackage{amsmath}
\usepackage{amsthm}
\usepackage{amssymb}

\allowdisplaybreaks 

\usepackage{amscd}

\usepackage{bbm}

\usepackage{mathrsfs} 

\usepackage[OMLmathsfit]{isomath}

\usepackage{mathabx}

\usepackage{leftindex}

\usepackage{enumitem}
\setlist{noitemsep}

\usepackage{stackengine}

\usepackage{hyperref}

\newtheoremstyle{exercise}
  {3pt} 
  {3pt} 
  {\small\rmfamily} 
  {
} 
  {\rmfamily\scshape} 
  {.} 
  {.5em} 
  {} 

\newtheoremstyle{newplain}
  {5pt}
  {5pt}
  {\itshape}
  {}
  {\rmfamily\scshape}
  { ---}
  {.5em}
  {}

\newtheoremstyle{newremark}
  {5pt}
  {5pt}
  {\rmfamily}
  {}
  {\rmfamily\scshape}
  { ---}
  {.5em}
  {}

\theoremstyle{newplain}
\newtheorem*{Theorem*}{Theorem} 

\theoremstyle{newplain}
\newtheorem{Theorem}{Theorem}

\newtheorem{Proposition}[Theorem]{Proposition}
\newtheorem{Conjecture}[Theorem]{Conjecture}
\newtheorem{Definition}[Theorem]{Definition}
\newtheorem{IntroTheorem}{Theorem}

\theoremstyle{newremark}
\newtheorem{Empty}[Theorem]{}
\newtheorem{Remark}[Theorem]{Remark}

\newtheorem{Claim}[Theorem]{Claim}

\theoremstyle{exercise}

\numberwithin{Theorem}{subsection}

\newcommand{\N}{\mathbb{N}}

\newcommand{\R}{\mathbb{R}}

\newcommand{\Rm}{\R^m}

\newcommand{\calB}{\mathscr{B}}
\newcommand{\calC}{\mathscr{C}}
\newcommand{\calD}{\mathscr{D}}

\newcommand{\calF}{\mathscr{F}}
\newcommand{\calG}{\mathscr{G}}

\newcommand{\calM}{\mathscr{M}}

\newcommand{\calP}{\mathscr{P}}

\newcommand{\calS}{\mathscr{S}}
\newcommand{\calT}{\mathscr{T}}

\newcommand{\calW}{\mathscr{W}}

\newcommand{\ssfH}{\mathsfit{H}}

\newcommand{\ssfL}{\mathsfit{L}}

\newcommand{\balpha}{\boldsymbol{\alpha}}

\newcommand{\bdelta}{\boldsymbol{\delta}}

\newcommand{\boldeta}{\boldsymbol{\eta}}
\newcommand{\bgamma}{\boldsymbol{\gamma}}
\newcommand{\bGamma}{\boldsymbol{\Gamma}}

\newcommand{\bLambda}{\boldsymbol{\Lambda}}

\newcommand{\bD}{\mathbf{D}}
\newcommand{\bE}{\mathbf{E}}

\newcommand{\bG}{\mathbf{G}}

\newcommand{\bI}{\mathbf{I}}

\newcommand{\bM}{\mathbf{M}}

\newcommand{\bc}{\pmb{c}}

\newcommand{\bq}{\pmb{q}}

\DeclareMathOperator{\bdiv}{\mathbf{div}}
\DeclareMathOperator{\bnabla}{\pmb{\nabla}}

\DeclareMathOperator{\rmbdry}{\mathrm{bdry}}          
\DeclareMathOperator{\rmcard}{\mathrm{card}}          
\DeclareMathOperator{\rmclos}{\mathrm{clos}}          
\newcommand{\rmcst}{\mathrm{cst}}
\DeclareMathOperator{\rmdiam}{\mathrm{diam}}          

\DeclareMathOperator{\rmdist}{\mathrm{dist}}          
\DeclareMathOperator{\rmdiv}{\mathrm{div}}            
\newcommand{\rmext}{\mathrm{ext}}
\DeclareMathOperator{\rmev}{\mathrm{ev}}              
\DeclareMathOperator{\rmFin}{\mathrm{Fin}}            
\newcommand{\rmid}{\mathrm{id}}                       
\DeclareMathOperator{\rmim}{\mathrm{im}}              
\DeclareMathOperator{\rmint}{\mathrm{int}}            
\DeclareMathOperator{\rmLip}{\mathrm{Lip}}            

\DeclareMathOperator{\rmloc}{\mathrm{loc}}

\DeclareMathOperator{\rmrestr}{\mathrm{restr}}
\DeclareMathOperator{\rmsign}{\mathrm{sign}}          
\DeclareMathOperator{\rmspan}{\mathrm{span}}          
\DeclareMathOperator{\rmspt}{\mathrm{spt}}            

\newcommand{\rmI}{\mathrm{I}}
\newcommand{\rmII}{\mathrm{II}}

\newcommand{\bvs}{BV_{1^*}}

\newcommand{\uTheta}{\bar{\Theta}}   
\newcommand{\ubar}[1]{\underaccent{\bar}{#1}}
\newcommand{\lTheta}{\ubar{\Theta}}

\newcommand{\ip}{\,\,\begin{picture}(-1,1)(-1,-2.5)\circle*{2}\end{picture}\;\,\,}

\newcommand{\ind}{\mathbbm{1}}

\newcommand{\lseg}{\boldsymbol{[}\!\boldsymbol{[}}
\newcommand{\rseg}{\boldsymbol{]}\!\boldsymbol{]}}

\newcommand{\hel} {
\hskip2.5pt{\vrule height7pt width.5pt depth0pt}
\hskip-.2pt\vbox{\hrule height.5pt width7pt depth0pt}
\, }

\newcommand{\lno}{\pmb{\thickvert}}
\newcommand{\rno}{\mkern+1mu\pmb{\thickvert}}

\def\Xint#1{\mathchoice
{\XXint\displaystyle\textstyle{#1}}%
{\XXint\textstyle\scriptstyle{#1}}%
{\XXint\scriptstyle\scriptscriptstyle{#1}}%
{\XXint\scriptscriptstyle
\scriptscriptstyle{#1}}%
\!\int}
\def\XXint#1#2#3{{%
\setbox0=\hbox{$#1{#2#3}{\int}$}
\vcenter{\hbox{$#2#3$}}\kern-.5\wd0}}

\def\dashint{\Xint-}

\newcommand{\MZ}{M\!Z}

\newcommand{\SCH}{SC\!H}

\newcommand{\cqfd} {
\renewcommand{\qedsymbol}{$\blacksquare$}
\qed
\renewcommand{\qedsymbol}{$\square$} }

\newcommand{\veps}{\varepsilon}
\newcommand{\bveps}{\boldsymbol{\veps}}

\newcommand{\vphi}{\varphi}

\newcommand{\la}{\langle}
\newcommand{\ra}{\rangle}

\newcommand{\wh}{\widehat}

\newcommand{\ts}{\mkern 3mu} 

\newcommand{\ie}{{\it i.e.}\ }
\newcommand{\eg}{{\it e.g.}\ }

\renewcommand{\em}{\bf}

\renewcommand{\leq}{\leqslant}
\renewcommand{\geq}{\geqslant}

\renewcommand{\subset}{\subseteq}

\makeatletter
\ifcsname phantomsection\endcsname
    \newcommand*{\qrr@gobblenexttocentry}[5]{}
\else
    \newcommand*{\qrr@gobblenexttocentry}[4]{}
\fi
\newcommand*{\addsubsection}{%
    \addtocontents{toc}{\protect\qrr@gobblenexttocentry}%
    \subsection}
\makeatother

\begin{document}



\title[On divergence operators]{On divergence operators:\\Free space and vanishing charges}

\author[Th. De Pauw]{Thierry De Pauw}

\address{Institute for Theoretical Sciences / School of Science, Westlake University\\
No. 600, Dunyu Road, Xihu District, Hangzhou, Zhejiang, 310030, China}

\email{thierry.depauw@westlake.edu.cn}

\keywords{Divergence equation}

\subjclass[2020]{Primary 35F15, 35Q35, 42B35; Secondary 35F05, 46F05, 35D30}


\begin{abstract}
We use localized topologies to prove existence and optimal regularity results for the divergence equation $\rmdiv v = F$ in critical cases $v \in L_1(\Omega;\Rm)$ or $v \in C_0(\Omega;\Rm)$, \ie we characterize those $F$ for which a solution $v$ exists whose norm is bounded by an appropriate norm of $F$.
We assume $\Omega$ satisfies a Poincar\'e inequality or an extension property.
We apply the general theory to give examples of admissible $F$ in each case.
\end{abstract}

\maketitle



\tableofcontents


\section{Introduction}

\addsubsection{Foreword}
Let $\Omega \subset \Rm$ be a non-empty open set.
Unless otherwise stated, $m \geq 2$.
Our aim is to study the ill-posed partial differential equation $\rmdiv v = F$ in $\Omega$ which we shall refer to as the divergence equation.
We seek solutions $v$ that belong to either $L_1(\Omega;\Rm)$ or to $C_0(\Omega;\Rm)$.
Our goal is to prove existence and optimal regularity theorems, \ie to identify exactly the spaces of {\it all} those $F$ so that the divergence equation admits a solution in the stipulated class and to give specific examples of such $F$.
We also seek to unravel the role played by the geometry of $\Omega$, whether it be a Poincar\'e inequality or an extension property.
Note that $v$ may be differentiable nowhere: solutions are understood either in the distributional sense or with respect to appropriate test functions to be determined. 
\par 
A natural approach consists in stipulating the regularity of $F$ and then studying that of some $v$ obtained via convolution with $K = \nabla E$, where $E$ is the Poisson kernel, or with a Bogovski\v{\i}-variant \cite{BOG.79} of $K$ to account for vanishing boundary conditions.
Works along these lines include, for instance, Borchers-Sohr \cite{BOR.SOH.90}, Dacorogna-Moser \cite{DAC.MOS.90}, Acosta-Dur\'an-Muschietti \cite{ACO.DUR.MUS.06}, Danchin-Mucha \cite{DAN.MUC.12}, and Berselli-Longo \cite{BER.LON.20}, where $F$ belongs to Lebesgue or H\"older spaces.
The core of these results is to apply Calder\'on-Zygmund estimates and Bogovski\v{\i} formul\ae.
These methods fail in the critical cases that we address here, see \eg the survey by Russ \cite{RUS.13}.
See Van Schaftingen \cite {VSJ.13} for endpoint inequalities in the more general setting of canceling operators.
The case of a rougher source $F$ -- namely $v \in L_\infty$ and $F$ a signed measure -- is studied, for instance, by Chen-Torres-Ziemer \cite{CHE.TOR.ZIE.09}.
Finally, rough solutions of the divergence equation have also been obtained by means of convex integration, see \eg De Lellis-Sz\'ekelyhidi \cite{DEL.SZE.09}.
\par 
Here we do not start by imposing a condition on $F$.
Instead, we stipulate the regularity sought for $v$ and infer the appropriate space where $F$ belongs, by means of {\it localized locally convex topologies} that we introduced elsewhere \cite{DEP.26c}, as explained below.
\addsubsection{The case \texorpdfstring{$m=1$}{m=1}}
\label{subsec.1.2}
To understand the characters  who are about to come on stage, it is useful to take a brief look at the situation in dimension $m=1$ and $\Omega=\R$.
We seek spaces of distributions $X_L(\R)$ and $X_C(\R)$, each equipped with a norm making it a Banach space, such that 
\begin{enumerate}
\item[(a)] $\bD(L_1(\R)) \subset X_L(\R)$, respectively, $\bD(C_0(\R)) \subset X_C(\R)$, and the former is dense in the latter, and
\item[(b)] the range of $\bD$ is closed, where $\bD$ is the derivative in the sense of distributions.
\end{enumerate}
A key observation\footnote{To this author's knowledge, reference to the closed range theorem is similar problems originates in \cite[\S 3 proof of proposition 1]{BOU.BRE.03}. It is one of the main steps in \cite{DEP.PFE.06b} and \cite{DEP.26c}.} follows from the closed range theorem: Condition (b) is equivalent to the range of $\bD^*$ being closed, where $\bD^* = - \bD$ (integration by parts).
This is illustrated below.
$$
\begin{array}{cc}
\begin{CD}
L_\infty(\R) @<{\bD^*}<< X_L(\R)^* \\
L_1(\R) @>{\bD}>> X_L(\R)
\end{CD}
\qquad\qquad\qquad\qquad
\begin{CD}
M(\R) @<{\bD^*}<< X_C(\R)^* \\
C_0(\R) @>{\bD}>> X_C(\R)
\end{CD}
\end{array}
$$
By means of anti-differentiation, members of $L_\infty(\R)$ and $M(\R)$ are identified with, respectively, Lipschitz functions and functions of bounded variation, up to a constant.
This suggests considering the Banach spaces $\rmLip_0(\R)[\|\cdot\|_L]$ for $X^*_L(\R)$ and $BV_0(\R)[\|\cdot\|_{TV}]$ for $X^*_C(\R)$.
The subscript $0$ indicates a normalization so that the only constant function in these spaces is identically zero.
Moreover, $\|\cdot\|_L$ is the Lipschitz constant and $\|\cdot\|_{TV}$ is the total variation.
Thus, we have achieved that the top line of both diagrams above is an isometric isomorphism -- in particular, has closed range -- and condition (b) is satisfied.
Admitting for the moment that $\rmLip_0(\R)$ and $BV_0(\R)$ have a predual it remains to check that this predual contains densely the derivatives of members of $L_1(\R)$, respectively, members of $C_0(\R)$.
Denseness is proved as follows: Let $u \in X_L(\R)^* = \rmLip_0(\R)$ be such that for all $v \in L_1(\R)$ we have $0 = \la \bD(v),u \ra = \la v , \bD^*(u) \ra = - \int_\R v \cdot u' \,d\ssfL^1$; then $u=0$ and it follows from Hahn-Banach's theorem that $\bD(L_1(\R))$ is dense in $X_L(\R)$.
The argument in case $v \in C_0(\R)$ is analogous.
Accordingly, it remains to identify a predual\footnote{Recall that, in general, a predual of a Banach space is not unique.} of $\rmLip_0(\R)$ and of $BV_0(\R)$ and check that these contain, respectively, $\bD(L_1(\R))$ and $\bD(C_0(\R))$.
By means of the evaluation map, these preduals are subspaces of, respectively, $\rmLip_0(\R)^*$ and $BV_0(\R)^*$.
\addsubsection{Free space and vanishing charges}
\label{subsec.1.3}
Returning to our primary concern we now assume that $m \geq 2$. 
We are about to describe the relevant predual $\calF(\Rm)$ of $\rmLip_0(\Rm)[\|\cdot\|_L]$ and the relevant predual $\SCH_0(\Rm)$ of $BV_{1^*}(\Rm)[\|\cdot\|_{TV}]$.
Here $1^*$ is the Sobolev conjugate exponent of $1$, and $BV_{1^*}(\Rm)$ consists of those members $u$ of $L_{1^*}(\Rm)$ each of whose distributional partial derivatives is a signed measure, and $\|u\|_{TV} = \|Du\|(\Rm)$ is the total variation of the distributional gradient.
\par 
The map $\bdelta : \Rm \to \rmLip_0(\Rm)^*$ that sends $x \in \Rm$ to the evaluation at $x$ is a (non-linear) isometry, \ref{3.1.1}(A).
The closure in $\rmLip_0(\Rm)^*$ of the linear space spanned by $\bdelta(\Rm)$ is denoted $\calF(\Rm)$ and called the Lipschitz-free space of $\Rm$ or Arens-Eells space of $\Rm$.
One classically shows that $\calF(\Rm)^* \cong \rmLip_0(\Rm)$, \ref{3.1.1}(B), and it is easy to realize that $\rmdiv v \in \calF(\Rm)$ when $v \in L_1(\Rm;\Rm)$, \ref{3.1.3}.
An analogous free space $\calF(X)$ can be associated with any pointed metric space $X$, and its study is an active branch of functional analysis; see \eg Godefroy and Weaver \cite{GOD.15,WEAVER.2nd} for an overview.
%
%
\par 
As $1^*$ is also the H\"older conjugate of $m$, to each $f \in L_m(\Rm)$ there corresponds an obvious member $\bLambda_m(f) \in BV_{1^*}(\Rm)^*$ and the norm of $\bLambda_m : L_m(\Rm) \to BV_{1^*}(\Rm)^*$ does not exceed $\kappa_m$, the isoperimetric constant, \ref{3.2.2}(A).
The closure in $BV_{1^*}(\Rm)^*$ of the range of $\bLambda_m$ is denoted $\SCH_0(\Rm)$.
It is not hard to show that $\SCH_0(\Rm)^* \cong \bvs(\Rm)$ and that $\rmdiv v \in \SCH_0(\Rm)$ when $v \in C_0(\Rm)$, \ref{3.2.2}(C) and \ref{3.2.4}.
Members of some related space\footnote{Specifically $\SCH0_{0,\calM,1}(\Omega)$; see \ref{5.1.3} for a definition.} are ``strong charges'' vanishing at the boundary.
These are particular cases of charges, a notion that initially appeared as an ersatz for measures in the study of multi-dimensional non-absolutely convergent integration theories developed by Pfeffer; see \eg \cite{PFE.91,PFE.05,PFEFFER}.
They first appeared in connection with the present context in work by Pfeffer and the author \cite{DEP.PFE.06b}.
\par 
Thus, Lipschitz-free spaces and vanishing charges originate from different mathematical horizons.
Yet, in subsections \ref{subsec.LSVF} and \ref{subsec.CVFVI} we prove the surjectivity of the following two divergence operators in ways that are so closely similar that it suggests they are particular instances in a general framework explained below.
The case of free space is originally due to C\'uth-Kalenda-Kaplick\'y \cite{CUT.KAL.KAP.17} and Godefroy-Lerner \cite{GOD.LER.18}, and the case of vanishing charges is originally due to Torres and the author \cite{DEP.TOR.09}, though the similarity is not apparent from the proofs in these sources.
Observe that in the diagrams below the adjoint operators $-\nabla$ are not isomorphisms (unlike the case $m=1$) but are still isometries (hence have closed range).
\begin{IntroTheorem}
Both divergence operators occurring in the diagrams below are surjective. 
\end{IntroTheorem}
$$
\begin{array}{cc}
\begin{CD}
L_\infty(\Rm;\Rm) @<{-\nabla}<< \rmLip_0(\Rm) \\
L_1(\Rm;\Rm) @>{\rmdiv}>> \calF(\Rm)
\end{CD}
\qquad\qquad\quad
\begin{CD}
M(\Rm;\Rm) @<{-\nabla}<< BV_{1^*}(\Rm) \\
C_0(\Rm;\Rm) @>{\rmdiv}>> \SCH_0(\Rm)
\end{CD}
\end{array}
$$
\vskip.3cm
\par 
We also interpret these results in the language of distributions; see theorem \ref{EX.1.3} for the case of $\calF(\Rm)$ and \ref{EX.2.2} for the case of $\SCH_0(\Rm)$.
\par 
On the side of free space, for instance:
\begin{itemize}
\item A distribution $F = \sum_{x \in E} \theta_x \bdelta_x$, where $E \subset \Rm$ is finite and $\theta_x$ are real numbers such that $\sum_{x \in E} \theta_x = 0$, can be expressed as $\rmdiv v$ for some $v \in L_1(\Rm;\Rm)$, \ref{EX.1.4}.
\item The same holds for Borel signed measures $F = \mu$ such that $\int_{\Rm} |x|_2 \,d|\mu|(x) < \infty$ and $\mu(\Rm) = 0$, \ref{EX.1.5}.
\item Still more generally, the same holds for 0-dimensional flat chains $F$ in $\Rm$ that are ``balanced'' (\ie formally $\la \ind_{\Rm} , F \ra = 0$), for instance when $F = \sum_{j \in \N} \theta_j \cdot (\bdelta_{b_j} - \bdelta_{a_j})$ is an infinite sum of dipoles such that $\sum_{j \in \N} |\theta_j| \cdot |b_j-a_j|_2 < \infty$; see \ref{EX.1.1} and \ref{EX.1.6} for this and other examples.
\end{itemize}
In each case, some $v$ can be chosen whose norm is bounded by that of $F$ in $\calF(\Rm)$; for instance, in the third case above, $\|v\|_\infty \leq (1 + \veps) \cdot \sum_j |\theta_j| \cdot [b_j-a_j|_2$. 
On the other side, it becomes obvious, by definition of $\SCH_0(\Rm)$, that each $f \in L_m(\Rm)$ can be expressed as $\rmdiv v$ for some $v \in C_0(\Rm)$ with $\|v\|_\infty \leq (1+ \veps) \cdot \|f\|_{L_m}$.
A similar result (with a periodic rather than vanishing condition) was initially obtained -- among other things -- by Bourgain-Brezis \cite{BOU.BRE.03}.
Note, however, that $v = \nabla u$, with $\triangle u = f$, may not be continuous, as this is a critical case of Sobolev embeddings. 
\addsubsection{Non-invertibility} 
It is important to note that some of the vector fields alluded to in the previous paragraph cannot be obtained by means of a convolution product with a fundamental solution of the divergence equation, \ie $K(x) = \frac{1}{m \cdot \balpha(m)}\cdot \frac{x}{|x|_2^m}$.
Indeed, in the case of a dipole $\mu = \bdelta_b - \bdelta_a$, $a \neq b$, $K * \mu \not \in L_1(\Rm;\Rm)$, and in case $f = \triangle u$, where $u(x) = \chi(x) \cdot x_1 \cdot |\log |x|_2|^t$, $0 < t < \frac{m-1}{m}$ and $\chi \in C^\infty_c(\Rm)$ equals 1 at the origin, then $f \in L_m(\Rm)$, yet $v = \nabla u \not \in C_0(\Rm;\Rm)$; in fact $v$ is not locally bounded near the origin.
\par 
Thus, the question occurs whether the divergence operators above admit a continuous linear right inverse.
The answer is negative in both cases.
For if $\rmdiv : L_1(\Rm;\Rm) \to \calF(\Rm)$ had a continuous linear right inverse $L : \calF(\Rm) \to L_1(\Rm;\Rm)$, then, since $L_1(\Rm;\Rm) \simeq L_1(\R)$, there would be an injective continuous linear map $\calF(\Rm) \to L_1(\R)$; in case $m=2$ this is in contradiction with Naor-Schechtman \cite[Theorem 1.1]{NAO.SCH.07} and in case $m \geq 3$ as well, since $\calF(\Rm)$ contains an isometric linear copy of $\calF(\R^2)$.
This is equivalent to saying that $\ker(\rmdiv)$ is not complemented in $L_1(\Rm;\Rm)$.
In fact, there does not even exist a (non-linear) uniformly continuous retraction $L_1(\Rm;\Rm) \to \ker{\rmdiv}$, as this would imply the existence of a continuous linear one, by means of the method of invariant means \cite[Theorem 7.2]{BENYAMINI.LINDENSTRAUSS}.
In particular, there does not exist a uniformly continuous $L : \calF(\Rm) \to L_1(\Rm;\Rm)$ such that $\rmdiv \circ L = \rmid_{\calF(\Rm)}$, for otherwise $\rmid_{L_1(\Rm;\Rm)} - L \circ \rmdiv$ would be a retraction whose existence is ruled out in the previous sentence.
\par
Similarly, $\rmdiv : C_0(\Rm;\Rm) \to \SCH_0(\Rm)$ does not admit a uniformly continuous right inverse.
This was established by Bouafia \cite{BOU.11} based on Bourgain-Brezis \cite{BOU.BRE.03}.
In other words, the classical Calder\'on-Zygmund theory asserting that $\rmdiv : \dot{W}_{1,p}(\Rm) \to L_p(\Rm)$, $1 < p < \infty$, admits a continuous linear right inverse does not extend to our setting; see Preiss \cite{PRE.97} and McMullen \cite{MCM.98} for the failure in case $p=\infty$.
\par 
Notwithstanding, based on our abstract theory \cite[8.3(M,N)]{DEP.26c}, we show that a solution $v$ of the divergence equation can be chosen such that $\|v\|_{L_1} \leq (1+\veps) \cdot \|F\|_{\calF}$, respectively, $\|v\|_{\infty} \leq (1+\veps) \cdot \|F\|_{\SCH_0}$ and, moreover, there exists such a choice of $v$ that depends continuously upon $F$; see theorem \ref{4.2.7}(D) and theorem \ref{5.2.3}.
\addsubsection{Localized locally convex topologies}
%
Since our main results rely on the general framework of {\it localized locally convex topologies}, that we developed in \cite{DEP.26c}, we now describe briefly what these are.
We are given a Banach space $E(\Omega;\Rm)$ of vector fields; in this paper $E = L_1$ or $E=C_0$.
As before, we seek $X(\Omega)$ that densely contains $\rmdiv(E(\Omega;\Rm))$ and so that we can apply the closed range theorem in the diagram below.
In order to do so, we shall obtain $X(\Omega)$ itself as a dual $X(\Omega) = Z(\Omega)[\calT_\calC]^*$, but {\it not the dual of a Banach space}, \ie the locally convex topologies $\calT_\calC$ are not normable.
\begin{equation*}
\begin{CD}
E(\Omega;\Rm)^* @<{-\nabla}<< X(\Omega)^* = Z(\Omega)[\calT_\calC]^{**} \\
E(\Omega;\Rm) @>{\rmdiv}>> X(\Omega) = Z(\Omega)[\calT_\calC]^*
\end{CD}
\end{equation*}
We note two things: 
\begin{itemize}
\item We should not expect in all cases that $-\nabla$ is an isometry, in fact, if its range is closed at all then the distortion of $-\nabla$ may be related to the geometry\footnote{This occurs in case $E=L_1$. For instance, if $m=2$, $C$ is a unit circle with a small arc of length $3\cdot \delta$ removed, and $\Omega$ is a tubular neighborhood of $C$ of thickness $\delta$, we consider $u(r,\theta) = \theta$. Then $\|\nabla u\|_{L_\infty} \sim 1$ is independent of $\delta$ and $\rmLip u \sim \frac{1}{\delta}$.} of $\Omega$.
\item We have candidates for $Z(\Omega)[\calT_\calC]^{**}$ in each case, namely $\rmLip_0(\Omega)$ and $BV_{\rmcst}(\Omega)$, respectively; recall subsection \ref{subsec.1.2}. Here and elsewhere, $BV_{\rmcst}(\Omega)$ is $BV(\Omega)$ modulo constant functions. More precisely, we have candidates for $Z(\Omega)$ but it is not clear now that these can be achieved as a bi-dual space.
\end{itemize}
\par 
The range of $-\nabla$ is closed if and only if there exists $\bc > 0$ such that $\bc \cdot \|u\|_{Z^{**}} \leq \|\nabla u\|_{E^*}$ for all $u \in Z(\Omega)$; see \ref{2.2}(A).
\begin{enumerate}
\item[(1)] When $E(\Omega;\Rm) = L_1(\Omega;\Rm)$, then $Z(\Omega) = \rmLip_0(\Omega)$ and the bound $\bc \cdot \|u\|_L \leq  \|\nabla u \|_{L_\infty}$ is equivalent to a geometric property of $\Omega$ that we call {\it connectedness by $\Lambda$-pencils}; see \ref{4.2.4} for a definition and theorem \ref{4.2.7} for a proof of the equivalence.
\item[(2)] When $E(\Omega;\Rm) = C_0(\Omega;\Rm)$, then $Z(\Omega) = BV_{\rmcst}(\Omega)$ and the bound $\|Du\|(\Omega) \leq \|Du\|_{M^m}$ is, in fact, an equality, by definition of the two terms, independently of $\Omega$.
\end{enumerate}
\par 
In the two situations above, we have a vector space $Z(\Omega)$ equipped with a seminorm $\lno \cdot \rno$, namely $\lno u \rno = \|u\|_L$ or $\lno u \rno = \|Du\|(\Omega)$.
It will be critical in the sequel that this seminorm be a norm and that the sets $C_k = Z(\Omega) \cap \{ u : \lno u \rno \leq k \}$, $k \in \N$, be $\calT$-compact.
This leads us to choosing the locally convex topologies $\calT$ as follows, according to extra conditions on $\Omega$ in the second case.
\begin{enumerate}
\item[(1)] When $E(\Omega;\Rm) = L_1(\Omega;\Rm)$ and $Z(\Omega) = \rmLip_0(\Omega)$, we let $\calT$ be the topology of pointwise convergence. According to Ascoli's theorem, $\|\cdot\|_L$-bounded subsets of $Z(\Omega)$ are compact with respect to $\calT$. 
\item[(2)] When $E(\Omega;\Rm) = C_0(\Omega;\Rm)$, we further distinguish between two cases.
\begin{enumerate}
\item[(2.1)] If $\Omega$ is a {\it bounded, connected $BV$-extension set} (see \ref{bv.ext}), then $\|D (\cdot)\|(\Omega)$-bounded subsets of $Z(\Omega) = BV_{\rmcst}(\Omega)$ are compact with respect to the quotient norm of $\|\cdot\|_{L_1}$. In this case, $\calT$ is the topology associated with this quotient norm.
\item[(2.2)] If $\Omega$ satisfies a {\it $(p,1)$-Poincar\'e inequality}, where $1 < p \leq 1^*$, see \ref{2.6}, then $\|D (\cdot)\|(\Omega)$-bounded subsets of $Z(\Omega) = BV_{p,\rmcst}(\Omega)$ are compact with respect to the weak* topology induced by the duality $L_{p,\rmcst}(\Omega) \times L_{q,\#}(\Omega)$, where $q$ is the H\"older conjugate of $p$. Here $BV_p(\Omega)$ consists of those functions in $L_p(\Omega)$ whose distributional gradient is a measure, and $L_{q,\#}(\Omega)$ consists of those members of $L_q(\Omega)$ whose average vanishes (in case $\ssfL^m(\Omega)=\infty$ we define $L_{q,\#}(\Omega)=L_q(\Omega)$). In this case, $\calT$ is the restriction to $BV_{p,\rmcst}(\Omega)$ of $\sigma(L_{p,\rmcst}(\Omega),L_{q,\#}(\Omega))$.
\end{enumerate}
\end{enumerate}
\par 
We associate with $\calT$ and $\calC = \{ C_k : k \in \N\}$ a new locally convex topology $\calT_\calC$ on $Z(\Omega)$, called {\it localized} for the following reason: The $\calT_\calC$-closed sets are exactly those whose intersection with each $C_k$ is $\calT$-closed.
Localized locally convex topologies are studied in \cite{DEP.26c} in a generality that encompasses our applications here.
\par
The key identification of $Z(\Omega)[\calT_\calC]^{**}$ with $Z(\Omega)$ via the evaluation map, called semireflexivity, is a consequence of the $\calT$-compactness of the $C_k$ and the Mackey-Arens theorem; see \cite[7.4]{DEP.26c} (moreover, $\calT_\calC$ is identified with the bounded-weak* topology of $Z(\Omega)[\calT_\calC]^{**}$).
This critical result will allow us to interpret members of $Z(\Omega)[\calT_\calC]^*$ as, indeed, acting on test functions from $Z(\Omega)$.
\par 
The topologies $\calT_\calC$ studied here present some oddities; for instance, they are not barrelled (Banach-Steinhaus' theorem fails) nor bornological (some bounded seminorms are not continuous).
They are sequential (sequentially $\calT_\calC$-continuous maps are $\calT_\calC$-continuous) but not Fr\'echet-Urysohn (closure and sequential closure do not always coincide). 
Details and references to \cite{DEP.26c} are given in relevant sections further down.
\par 
The thrust of introducing these topologies is the description of their convergent sequences and, thereby, the characterization of $\calT_\calC$-continuous linear forms.
This pertains to our goal, since $\rmdiv(E(\Omega;\Rm)) = Z(\Omega)[\calT_\calC]^{*}$.
A sequence $\la u_j \ra_j$ in $Z(\Omega)$ is $\calT_\calC$-convergent to $u$ if and only it is contained in some $C_k$ and $\calT$-converges to $u$.
For instance, in case (2.1) above, if $\ssfL^m(\Omega)=\infty$ this means that $\lim_j \|u_j-u\|_{L_1}=0$ and $\sup_j \|Du_j\|(\Omega) < \infty$.
Note carefully that this is different from {\it strict convergence} defined in \cite[3.14]{AMBROSIO.FUSCO.PALLARA}.
\par 
Next, we briefly describe the theorems arising in the three contexts mentioned above.
To this end, we introduce a specific notation in each case:
\begin{enumerate}
\item[(1)] The localized topology of $\|\cdot\|_L$-bounded pointwise convergence on $\rmLip_0(\Omega)$ will be denoted $\calP_{\rmLip}$.
\item[(2.1)] The localized topology of $\|D(\cdot)\|(\Omega)$-bounded $\|\cdot\|_{L_1}$-convergence on $BV_{\rmcst}(\Omega)$ will be denoted $\calM_{1,TV}$.
\item[(2.2)] The localized topology of $\|D(\cdot)\|(\Omega)$-bounded weak* convergence on $BV_{p,\rmcst}(\Omega)$ will be denoted $\calW_{p,TV}$.
\end{enumerate}
\addsubsection{The localized topology \texorpdfstring{$\calP_{\rmLip}$}{P-Lip}}
A linear form $F : \rmLip_0(\Omega) \to \R$ is continuous with respect to the localized topology $\calP_{\rmLip}$ if and only if for every $\veps > 0$ there are a finite set $E \subset \Omega$ and $\theta > 0$ such that $$| \la u , F \ra| \leq \theta \cdot \max_{x \in E} |u(x)| + \veps \cdot \|u\|_L$$ for all $u$, see \ref{4.1.2}(F).
Theorem \ref{4.1.3} states that $\calF(\Omega) = \rmLip_0(\Omega)[\calP_{\rmLip}]^*$, where the latter is equipped with its strong topology, \ie the topology of uniform convergence on $\calP_{\rmLip}$-bounded sets.
\par 
For our main theorem regarding Lipschitz-free spaces we need the following definition.
Given distinct points $a$ and $b$ in $\Omega$ and $\Lambda > 0$ we call {\it $\Lambda$-pencil of curves in $\Omega$ from $a$ to $b$} a probability measure $\mu$ on the space of oriented Lipschitz curves $\gamma$ from $a$ to $b$ satisfying the following property: If $\nu$ is the Borel measure on $\Rm$ defined by $\nu(A) = \int \ssfH^1(A \cap \rmim \gamma)\, d\mu(\gamma)$, then $\nu \ll \ssfL^m$, $\rmspt \nu \subset \rmclos \Omega$, and $\nu(\Rm) \leq \Lambda \cdot |b-a|_2$.
We say that $\Omega$ is connected by $\Lambda$-pencils of curves of if each distinct $a,b \in \Omega$ admit a $\Lambda$-pencil of curves in $\Omega$ connecting them.
\begin{IntroTheorem}
The operator $\rmdiv : L_1(\Omega;\Rm) \to \calF(\Omega)$ is surjective if and only if $\Omega$ is connected by $\Lambda$-pencils of curves for some $\Lambda > 0$.
In this case, the smallest such $\Lambda$ is the inverse of the largest $\lambda > 0$ such that $\lambda \cdot \|u\|_L \leq \|\nabla u\|_{L_\infty}$ for all $u$.
\end{IntroTheorem}
\par 
This is theorem \ref{4.2.7}, the proof of which we now briefly sketch.
We identify $\calF(\Omega) = \rmLip_0(\Omega)[\calP_{\rmLip}]^*$.
If $\Omega$ is connected by $\Lambda$-pencils of curves then $\|u\|_L \leq \Lambda \cdot \|\nabla u\|_{L_\infty}$ for all $u$; this is a simple calculation, \ref{4.2.5}(B).
Thus, the range of $-\nabla = \rmdiv^*$ is closed and $\rmdiv$ is surjective, since its range is dense (this is a consequence of Hahn-Banach and semireflexivity).
Conversely, if $\rmdiv$ is surjective then there exists $\lambda > 0$ such that $\lambda \cdot \|u\|_L \leq \|\nabla u\|_{L_\infty}$ for all $u$.
Thus, if $a, b \in \Omega$ are distinct there exists $v \in L_1(\Omega;\Rm)$ such that $\rmdiv v = \bdelta_b - \bdelta_a$ and $\|v\|_{L_1} \leq (1 + \veps) \cdot \lambda^{-1} \cdot \|\bdelta_b - \bdelta_a\|_{\calF}$, see \cite[8.1(M)]{DEP.26c}, and $\|\bdelta_b - \bdelta_a\|_{\calF} = |b-a|_2$.
Now we interpret the vector field $v$ as a 1-dimensional current $T$ in the obvious way and we note that $T$ is normal. 
It then follows from a theorem of Smirnov \cite{SMI.93} and its version by Paolini-Stepanov \cite{PAO.STE.12} that $T$ can be represented as an integral over a family of weighted elementary 1-dimensional currents $\lseg \gamma \rseg$ that correspond to integration over oriented Lipschitz curves $\gamma$.
The connection with a $(1+\veps)\cdot \lambda^{-1}$-pencil should now be clear.
\addsubsection{The localized topology \texorpdfstring{$\calW_{p,TV}$}{W-p-TV}}
A linear form $F : BV_{p,\rmcst}(\Omega) \to \R$ is continuous with respect to the localized topology $\calW_{p,TV}$ if and only if for every $\veps > 0$ there exists a finite set $E \subset L_{q,\#}(\Omega)$ such that 
\begin{equation}
\label{eq.intro.1}
| \la u , F \ra| \leq \max_{f \in E} \left| \int_\Omega u \cdot f \,d\ssfL^m \right| + \veps \cdot \|Du\|(\Omega)
\end{equation}
for all $u$, \ref{5.2.1}(F).
Theorem \ref{5.2.4} asserts that $\SCH_0(\Rm) = \bvs(\Rm)[\calW_{1^*,TV}]^*$, so that we are indeed generalizing to $\Omega \neq \Rm$ the situation described in subsection \ref{subsec.1.3}.
\begin{IntroTheorem}
If $1 < p \leq 1^*$ and $\Omega$ satisfies the $(p,1)$-Poincar\'e inequality, then the operator $\rmdiv : C_0(\Omega;\Rm) \to BV_{p,\rmcst}(\Omega)[\calW_{p,TV}]^*$ is surjective. 
In particular, if $f \in L_{q,\#}(\Omega)$ then there exists $v \in C_0(\Omega;\Rm)$ such that $\rmdiv v = f$ and $\|v\|_\infty \leq (1+\veps) \cdot \|f\|_{L_q}$.
\end{IntroTheorem}
\par 
This is theorem \ref{5.2.3} together with \ref{EX.2.5}.
From \eqref{eq.intro.1} the following should now be clear: If we know a subspace $\Xi \subset BV_{p,\rmcst}(\Omega)[\lno\cdot\rno]^{*}$ that contains $L_{q,\#}(\Omega)$, then the closure of $L_{q,\#}(\Omega)$ in $\Xi$ provides new examples of members of $BV_{p,\rmcst}(\Omega)[\calW_{p,TV}]^{*}$.
Finding such $\Xi$ is not necessarily an easy task, since even in the simplest case $p=1^*$ and $\Omega = \Rm$ the dual of $BV_{1^*}(\Rm)$ remains a mystery despite work by Meyers-Ziemer \cite{MEY.ZIE.77}, Phuc-Torres \cite{PHU.TOR.17}, Fusco-Spector \cite{FUS.SPE.18}, Bouafia and the author \cite{BOU.DEP.24}; see the survey by Torres \cite{TOR.18}.
Nevertheless, following an idea of Cohen-DeVore-Tadmore \cite{COH.DEV.TAD.24} we show that the weak Lebesgue space $\Xi = L_{q,\infty}(\Omega) \cap L_{q,\#}(\Omega)$ embeds into the dual of $BV_{p,\rmcst}(\Omega)[\lno\cdot\rno]$ in case $\Omega$ is a bounded, connected $BV$-extension set, \ref{EX.2.6}(D).
\par 
We are then led to determine the closure of $L_q(\Omega)$ in $L_{q,\infty}(\Omega)$, $1 < q < \infty$, which seems to be a new result.
In order to describe it, given a Borel-measurable function $f : \Omega \to \R$ and $y > 0$ we introduce 
\begin{equation*}
\bveps_q(f,y) = \ssfL^m(\Omega \cap \{ x : |f(x)| > y \}) \cdot y^q.
\end{equation*}
Thus, $f \in L_q(\Omega)$ if and only if $\int_0^\infty \frac{\bveps_q(f,y)}{y}\,dy < \infty$, whereas $f \in L_{q,\infty}(\Omega)$ if and only if $y \mapsto \bveps_q(f,y)$ is bounded.
We then introduce the space
\begin{equation*}
L_{q,0}(\Omega) = L_0(\Omega) \cap \left\{ f : \lim_{y \to 0^+} \bveps_q(f,y) = 0 = \lim_{y \to \infty} \bveps_q(f,y)\right\}. 
\end{equation*}
We prove that the closure of $L_q(\Omega)$ in $L_{q,\infty}(\Omega)$ is $L_{q,0}(\Omega)$, \ref{EX.2.6}(A).
Furthermore, $L_q(\Omega) \subsetneq L_{q,0}(\Omega) \subsetneq L_{q,\infty}(\Omega)$, \ref{EX.2.6}(B).
For instance $f(x) = \frac{1}{|x|_2^\frac{m}{q} \cdot \sqrt[q]{1 + |\log |x|_2|}}$ is a member of $L_{q,0}(\Rm)$ that does not belong to $L_q(\Rm)$.
In \ref{EX.2.6}(E) we obtain:
\begin{IntroTheorem}
If either $\Omega = \Rm$ or $\Omega$ is a bounded, connected, $BV$-extension set and $f \in L_{m,0}(\Omega) \cap L_{m,\#}(\Omega)$ then there exists $v \in C_0(\Omega)$ such that $\rmdiv v = f$ and $\|v\|_\infty \leq (1+ \veps) \cdot \|f\|_{L_{m,\infty}}$.
\end{IntroTheorem}
\addsubsection{The localized topology \texorpdfstring{$\calM_{1,TV}$}{M-1-TV}}
A linear form $F : BV_{\rmcst}(\Omega) \to \R$ is continuous with respect to the localized topology $\calM_{1,TV}$ if and only if for every $\veps >0$ there exists $\theta$ such that $$|\la u , F \ra | \leq \theta \cdot \|u - (u)_\Omega\|_{L_1} + \veps \cdot \|Du\|(\Omega)$$ for all $u$, \ref{5.1.3}(F).
\begin{IntroTheorem}
If the operator $\rmdiv : C_0(\Omega;\Rm) \to BV_{\rmcst}(\Omega)[\calM_{1,TV}]^*$ is surjective then $\Omega$ satisfies the $(1,1)$-Poincar\'e inequality.
If $\Omega$ is a bounded, connected $BV$-extension set then the operator $\rmdiv : C_0(\Omega;\Rm) \to BV_{\rmcst}(\Omega)[\calM_{1,TV}]^*$ is surjective.
\end{IntroTheorem}
\par 
These are theorems \ref{5.1.6} and \ref{5.1.7}.
In order to state our application of this, we introduce the following definitions.
For $\mu \in M(\Rm)$ we let
\begin{equation*}
\|\mu\|_{\MZ} = \sup \left\{ \frac{|\mu|(B(x,r))}{r^{m-1}} : x \in \Rm \text{ and } r > 0 \right\} \in [0,\infty].
\end{equation*}
We also define
$
\MZ(\Rm) = M(\Rm) \cap \{ \mu : \|\mu\|_{\MZ} < \infty \}
$.
Moreover, for $\mu \in M(\Omega)$ and $0 < \tau \leq 1$ we define
\begin{equation*}
\boldeta_\Omega(\mu,\tau) = \sup \left\{ \frac{1}{\delta(x)} \cdot \frac{|\mu|(B(x, \tau \cdot \delta(x))}{(\tau \cdot \delta(x))^{m-1}} : x \in \Omega \right\}\in [0,\infty],
\end{equation*}
where, $\delta(x) = \rmdist(x,\Rm \setminus \Omega)$ (we now assume that $\Omega \neq \Rm$).
Finally, we introduce the following space
\begin{equation*}
\MZ_{0,\#}(\Omega) = M(\Omega) \cap \left\{ \mu : \hat{\mu} \in \MZ(\Rm) \text{ and } \lim_{\tau \to 0^+} \boldeta_\Omega(\mu,\tau) = 0 \text{ and } \mu(\Omega)=0 \right\},
\end{equation*}
where $\hat{\mu} \in M(\Rm)$ is the extension of $\mu$ to $\Rm$ defined by the formula $\hat{\mu}(B) = \mu(\Omega \cap B)$.
Based on a variant of Whitney's partitions of unity \ref{WHITNEY}, we establish the following in \ref{EX.2.7}(E):
\begin{IntroTheorem}
If $\Omega$ is a bounded, connected $BV$-extension set and $\mu \in \MZ_{0,\#}(\Omega;\Rm)$ then there exists $v \in C_0(\Omega)$ such that $\rmdiv v = \mu$ and $\|v\|_\infty \leq \bc(m,\Omega) \cdot \|\mu\|_{\MZ}$.
\end{IntroTheorem}
\par 
To close this introduction we explain how to construct examples of measures $\mu$ to which Theorem F applies.
We start with a variant $K \subset \R^2$ of the Koch curve obtained in the following way.
Instead of using the same angle $\theta = \frac{\pi}{3}$ at each step of the iterative construction, we use a sequence of angles $\la \theta_j \ra_j$ decreasing to 0 but not too fast. 
Indeed, in case the decay of $\theta_j$ to 0 is too fast, $\ssfH^1(K) < \infty$ and there exists no non-trivial continuous $v$ whose divergence is supported in $K$, according to de Valeriola-Moonens \cite{DEV.MOO.10} (see also Ponce \cite{PON.13} in this respect and in connection with the theorem above). 
Thus, we choose $\la \theta_j \ra_j$ so that the Hausdorff dimension of $K$ is 1 and $\ssfH^1 \hel (K \cap U)$ is not $\sigma$-finite for any open set $U \subset \R^2$ such that $K \cap U \neq \emptyset$.
We then define $\mu = \gamma_*(\ssfL^1 \hel [0,1])$, where $\gamma : [0,1] \to K$ is the ensuing parametrization of $K$.
We choose $\Omega$ whose closure contains $K$ and we allow for $K$ to intersect $\rmbdry \Omega$.
We then choose an arbitrary Borel-measurable function $h : K \to \R$ such that $\int_K h \,d\mu = 0$ and $|h(x)| \leq \bc \cdot \delta(x)$, $x \in K$, for some $\bc > 0$.
Then there exists $v \in C_0(\Omega;\R^2)$ such that $\rmdiv v = \mu \hel h$.

\swapnumbers
\numberwithin{Theorem}{section}

\section{Notation and preliminaries}

\begin{Empty}[Euclidean space]
\label{2.4}
Throughout, $m$ is a positive integer and, unless explicitly stated we assume that $m \geq 2$.
$\Rm$ is equipped with its usual structure of inner product space.
Thus, $x \ip y$ denotes the inner product of $x, y \in \Rm$ and $|x|_2 = \sqrt{x \ip x}$.
In $\Rm$ and any other normed space $B(x,r)$ and $U(x,r)$ are, respectively, the closed and open balls of center $x$ and radius $r$.
\end{Empty}

\begin{Empty}[$\simeq$ and $\cong$]
\label{2.1}
If $X$ and $Y$ are Banach spaces then we write $X \simeq Y$ if they are isomorphic and we write $X \cong Y$ if they are isometrically isomorphic.
Recall that a linear isometry is injective but not necessarily surjective.
\end{Empty}

\begin{Empty}[Duality]
All vector spaces in this paper are real.
If $X$ and $Y$ are two vector spaces in duality then we most often denote the duality as $X \times Y \to \R : (x,y) \mapsto \la x, y \ra$, for instance when $X$ is a Banach space and $Y=X^*$ is its first dual.
The corresponding weak topologies on $X$ and $Y$ are denoted, respectively, $\sigma(X,Y)$ and $\sigma(Y,X)$.
\end{Empty}

\begin{Empty}[Duality and adjoint operators]
\label{2.2}
Let $X$ be a Banach space.
We denote by $X^*$ the dual of $X$, by $x^*$ a typical member of $X^*$, and by $\|x^*\|_X^*$ its norm in $X^*$.
\par 
Let $X$ and $Y$ be Banach spaces and $T : X \to Y$ a bounded linear operator.
We let $\ker (T)$ and $\rmim (T)$ be, respectively, the kernel and the range (the image) of $T$.
We consider $T^* : Y^* \to X^*$, the corresponding adjoint operator.
The following two simple observations will be instrumental in the proofs of theorems \ref{3.1.4}, \ref{3.2.5}, \ref{4.2.7}, and \ref{5.1.6}.
\begin{enumerate}
\item[(A)] {\it The following are equivalent.
\begin{enumerate}
\item[(a)] There exists $\bc > 0$ such that $\bc \cdot \|y^*\|_Y^* \leq \|T^*(y^*)\|_X^*$ for all $y^* \in Y^*$.
\item[(b)] $T$ is surjective.
\end{enumerate}
}
\end{enumerate}
\par 
{\it Proof.}
{\bf (i)} We prove that $(a) \Rightarrow (b)$.
Notice that $T^*$ is injective, hence, $\rmim (T)$ is dense, by \cite[4.12 corollary (b)]{RUDIN}.
Furthermore, $\rmim(T^*)$ is closed, thus, $\rmim (T)$ is closed, by the closed range theorem \cite[4.14]{RUDIN}.
Therefore, $T$ is surjective.
\par 
{\bf (ii)}
We prove that $(b) \Rightarrow (a)$.
Since $T$ is surjective, $T^*$ is injective, by \cite[4.12]{RUDIN}.
Furthermore, as $\rmim T$ is closed, we infer that $\rmim T^*$ is closed, by the closed range theorem.
Accordingly, $T^*$ is is a bijective continuous linear operator between the Banach spaces $X^*$ and $\rmim T^*$.
The conclusion follows from the open mapping theorem.\cqfd
\par 
We abbreviate $Z = \ker(T)$ and we let $\pi : X \to \frac{X}{Z}$ be the quotient map.
\par 
\begin{enumerate}
\item[(B)] {\it If $T^*$ is an isometry then the induced operator $S : \frac{X}{\ker(T)} \to Y : \pi(x) \mapsto T(x)$ is an isometric isomorphism. Thus, $\frac{X}{\ker T} \cong Y$.}
\end{enumerate}
\par 
{\it Proof.}
According to (A), it suffices to show that $S$ is an isometry.
It is easy to see that $\|S\| \leq \|T\|=\|T^*\| \leq 1$.
Conversely, given $x \in X$ there exists $\alpha \in \left( \frac{X}{Z}\right)^*$ such that $\la \pi(x) , \alpha \ra = \|\pi(x)\|_{X/Z}$ and $\|\alpha\|_{X/Z}^* = 1$.
With the notation of \cite[4.6]{RUDIN}, notice that $\alpha \circ \pi \in (\ker(T))^\perp = \left(\leftindex^\perp \rmim(T^*) \right)^\perp = \rmclos_{\sigma(X^*,X)} \rmim(T^*)$.
Furthermore, as in the proof of (A) we infer that $\rmim(T^*)$ is $\sigma(X^*,X)$-closed.
Therefore, $\alpha \circ \pi = T^*(y^*)$ for some $y^* \in Y^*$.
This $y^*$ is unique and $\|y^*\|_Y^* = \|\alpha \circ \pi\|_Y^* \leq 1$, since $T^*$ is an isometry.
Finally, $\|\pi(x)\|_{X/Z} = \la \pi(x) , \alpha \ra = \la x , T^*(y^*) \ra = \la T(x) , y^* \ra \leq \|T(x)\|_Y \leq \|S(\pi(x))\|_Y$.\cqfd
\end{Empty}

\begin{Empty}[$\Omega$]
\label{2.7}
Throughout, $\Omega \subset \Rm$ is a non-empty open set. 
In section \ref{sec.CIPCRT}, we will always assume $\Omega = \Rm$.
In section \ref{sec.LFS}, we will assume for convenience that $0 \in \Omega$ and our main result there, theorem \ref{4.2.7}, will involve the existence of pencils of curves in $\Omega$, as defined in \ref{4.2.4}.
In section \ref{sec.SCH}, we will either assume (theorem \ref{5.1.7}) that $\Omega$ is a bounded, connected, $BV$-extension set, as defined in \ref{bv.ext}, or (theorem \ref{5.2.3}) that $\Omega$  satisfies a $(p,1)$-Poincar\'e inequality with $p \neq 1$, see \ref{2.6}.
In section \ref{sec.EXAMPLES}, all results have their own specific assumptions about $\Omega$.
\end{Empty}

\begin{Empty}[Mollifying]
\label{2.8}
$\calD(\Omega)$ is the set of test functions, \ie those compactly supported $C^\infty$-functions $\Omega \to \R$.
Moreover, we let $\la \Phi_r \ra_{r > 0}$ be a family of mollifiers.
Specifically, $\Phi_r(x) = r^{-m} \cdot \Phi(r^{-1} \cdot x)$, where $\Phi \in \calD(\Rm)$, $\Phi \geq 0$, $\rmspt \Phi \subset B(0,1)$, $\Phi$ is even, and $\int_{\Rm} \Phi \,d\ssfL^m = 1$.
\end{Empty}

\begin{Empty}[Measures]
\label{2.9}
We let $\ssfL^m$ be the Lebesgue measure in $\Rm$ and we abbreviate $\balpha(m) = \ssfL^m(B(0,1))$.
Moreover, $\ssfH^1$ and $\ssfH^{m-1}$ are, respectively, the $1$-dimensional and $(m-1)$-dimensional Hausdorff measures in $\Rm$.
\par 
We let $\calB(\Omega)$ denote the set of Borel-measurable subsets of $\Omega$.
The indicator function of a set $A \subset \Rm$ is denoted $\ind_A$.
If $\mu$ is a signed Borel measure on $\Omega$ and $f : \Omega \to \R$ is Borel-measurable and $|\mu|$-summable then we let $\mu \hel f$ be the signed Borel measure on $\Omega$ defined by $(\mu \hel f)(B) = \int_B f \, d\mu$, $B \in \calB(\Omega)$.
If $f = \ind_A$ we write $\mu \hel A$ instead of $\mu \hel \ind_A$.
\par 
When $\mu$ is a Borel measure on $\Omega$ we let $L_p(\Omega,\mu)$ be the usual Lebesgue space of equivalence classes of Borel-measurable functions $u : \Omega \to \R$ so that $|u|^p$ is $\mu$-summable.
Predominantly we do not need to distinguish between a function and its equivalence class in $L_p(\Omega,\mu)$ except for in \ref{EX.2.7} where the appropriate representative will be defined.
If $\mu = \ssfL^m|_{\calB(\Omega)}$ we abbreviate $L_p(\Omega)$ for $L_p(\Omega,\mu)$.
The norm of $L_p(\Omega)$ is denoted indifferently $\|\cdot\|_{L_p}$ or $\|\cdot\|_{L_p(\Omega)}$. 
\begin{enumerate}
\item[(A)] {\it Assume that $1 \leq r < \infty$, $1 \leq s < \infty$, $u \in L_r(\Omega)$, $y \in \R$, $u - y \cdot \ind_\Omega \in L_s(\Omega)$, and $\ssfL^m(\Omega) = \infty$. Then $y=0$.}
\end{enumerate}
\par 
{\it Proof.}
This follows from Chebyshev's inequality applied twice.
If $y \neq 0$ and $A = \Omega \cap \{ x : |u(x) - y| > \frac{|y|}{2} \}$ then $\ssfL^m(A) < \infty$, since $|u-y \cdot \ind_\Omega|^s \in L_1(\Omega)$.
Moreover, $\ssfL^m(\Omega \setminus A) < \infty$, since $\Omega \setminus A \subset \Omega \cap \{ x : |u(x)| \geq \frac{|y|}{2} \}$ and $|u|^r \in L_1(\Omega)$.
Thus, $\ssfL^m(\Omega) < \infty$, a contradiction.\cqfd
\end{Empty}

\begin{Empty}[$L_{p,\rmcst}(\Omega)$ and $L_{q,\#}(\Omega)$]
\label{2.5}
In this number and throughout the paper, $1 \leq p < \infty$, and $q$ is the H\"older conjugate exponent of $p$.
We define $L_{p,\rmcst}(\Omega)$ to be the quotient of $L_p(\Omega)$ by its subspace $L_p(\Omega) \cap \{ u : u \text{ is constant a.e.}\}$.
Notice that this subspace is non-trivial if and only if $\ssfL^m(\Omega) < \infty$.
Therefore, $L_{p,\rmcst}(\Omega) = L_p(\Omega)$ if $\ssfL^m(\Omega) = \infty$ and $L_{p,\rmcst}(\Omega) = L_p(\Omega)/\R$ if $\ssfL^m(\Omega) < \infty$.
Members of $L_{p,\rmcst}(\Omega)$ are denoted $[u]$, $u \in L_p(\Omega)$.
Note that $L_{p,\rmcst}(\Omega)$ is a separable Banach space under the norm
\begin{equation*}
\|[u]\|_p = \inf \left\{ \| u - y \cdot \ind_\Omega \|_{L_p(\Omega)} : y \in \R \right\}.
\end{equation*}
\par 
If $u \in L_p(\Omega)$ we define
\begin{equation*}
(u)_\Omega = \begin{cases}
\frac{1}{\ssfL^m(\Omega)} \int_\Omega u \, d\ssfL^m & \text{if } \ssfL^m(\Omega) < \infty \\
0 & \text{if } \ssfL^m(\Omega) = \infty.
\end{cases}
\end{equation*}
\begin{enumerate}
\item[(A)] {\it For all $[u] \in L_{p,\rmcst}(\Omega)$ we have $\|[u]\|_p \leq \|u - (u)_\Omega \cdot \ind_\Omega\|_{L_p(\Omega)} \leq 2 \cdot \|[u]\|_p$.}
\end{enumerate}
\par 
{\it Proof.}
The first inequality is trivial. 
We ought to prove the second one only in case $\ssfL^m(\Omega) < \infty$.
Choose $y \in \R$ so that $\|u - y \cdot \ind_\Omega| \|_{L_p(\Omega)} = \|[u]\|_p$ and note that $|y - (u)_\Omega| \leq \ssfL^m(\Omega)^{\frac{1}{q}-1} \cdot  \|[u]\|_p$, thus, $\|(y - (u)_\Omega) \cdot \ind_\Omega \|_{L_p(\Omega)} \leq \|[u]\|_p$,
hence, $\| u - (u)_\Omega \cdot \ind_\Omega \|_{L_p(\Omega)} \leq 2 \cdot \|[u]\|_p$.\cqfd
\par 
As $\|u-(u)_\Omega\|_{L_p(\Omega)}$ does not depend on the choice of a representative of $[u] \in L_{p,\rmcst}(\Omega)$, this defines an equivalent norm on $L_{p,\rmcst}(\Omega)$, according to (A).
\par 
Next, we define $L_{q,\#}(\Omega) = L_q(\Omega)$ in case $\ssfL^m(\Omega) = \infty$ and $L_{q,\#}(\Omega) = L_q(\Omega) \cap \left\{ f : \int_\Omega f\, d\ssfL^m = 0 \right\}$ in case $\ssfL^m(\Omega) < \infty$.
This is readily a closed subspace of $L_q(\Omega)$.
\begin{enumerate}
\item[(B)] {\it $L_{p,\rmcst}(\Omega)^* \cong L_{q,\#}(\Omega)$.}
\end{enumerate} 
\par 
{\it Proof.}
We may assume that $\ssfL^m(\Omega) < \infty$.
If $f \in L_{q,\#}(\Omega)$ then $\int_\Omega u \cdot f \, d\ssfL^m$ is defined and does not depend upon the choice of a representative of $[u] \in L_{p,\rmcst}(\Omega)$, since $\int_\Omega f \, d\ssfL^m = 0$.
This defines a linear map $T : L_{q,\#}(\Omega) \to L_{p,\rmcst}(\Omega)^*$ by means of the formula $\la [u] ,T(f) \ra = \int_\Omega u \cdot f \, d\ssfL^m$.
Routine arguments based on the Riesz representation theorem for $L_p$ show that $T$ is an isometry and surjective.\cqfd
\begin{enumerate}
\item[(C)] {\it If $1 < p < \infty$ then $L_{p,\rmcst}(\Omega)$ is reflexive.} 
\end{enumerate}
\par 
{\it Proof.}
$L_{p,\rmcst}(\Omega)$ is a closed subspace of the reflexive space $L_p(\Omega)$.\cqfd
\end{Empty}

\begin{Empty}[$C_0(\Omega;\Rm)$ and $M(\Omega;\Rm)$]
\label{3.2.1}
We let $C_0(\Omega)$ (resp. $C_0(\Omega;\Rm)$) be the set of continuous functions $\vphi : \Omega \to \R$ (resp. of continuous vector fields $v : \Omega \to \Rm$) that vanish at infinity, \ie for every $\veps > 0$ there exists a compact set $K(\veps) \subset \Omega$ such that $|\vphi(x)| < \veps$ (resp. $|v(x)|_2 < \veps$) for all $x \in \Omega \setminus K(\veps)$.
These are, in particular, bounded; and $C_0(\Omega)[\|\cdot\|_\infty]$ (resp. $C_0(\Omega;\Rm)[\|\cdot\|_\infty]$) is a Banach space, where, when the target is $\Rm$, we let $\|v\|_\infty = \sup \{ |v(x)|_2 : x \in \Omega \}$.
Moreover, $\calD(\Omega)$ (resp. $C^\infty_c(\Omega;\Rm)$) is dense in $C_0(\Omega)$ (resp. $C_0(\Omega;\Rm)$).
\par 
According to the Riesz-Markov representation theorem, $C_0(\Omega)^*$ is isometrically isomorphic to the space $M(\Omega)$ of signed Borel measures $\mu$ on $\Omega$.
The identification sends $\mu \in M(\Omega)$ to $T_\mu \in C_0(\Omega)^*$ defined by $\la \vphi , T_\mu \ra = \int_{\Omega} \vphi \, d\mu$.
In particular, $\|\mu\|_M = \|T_\mu\|_{C_0}^* = \sup \left\{ \int_{\Omega} \vphi \, d\mu : \vphi \in C_0(\Omega) \text{ and } \|\vphi\|_\infty \leq 1 \right\} = \sup \left\{ \int_{\Omega} \vphi \, d\mu : \vphi \in \calD(\Omega) \text{ and } \|\vphi\|_\infty \leq 1 \right\}$.
\par 
If $\la T_{\mu_j} \ra_j$ is a sequence in $C_0(\Omega)^*$ that converges to a linear functional $T : C_0(\Omega) \to \R$ in the sense of distributions then $\|T\|_{C_0}^* \leq \liminf_j \|T_{\mu_j}\|_{C_0}^*$.
In particular, if $\liminf_j \|\mu_j\|_M < \infty$ then $T=T_\mu$ for some $\mu \in M(\Omega)$.
\par 
We let $M(\Omega;\Rm) = M(\Omega)^m$ and we note that it is isomorphic to $C_0(\Omega;\Rm)^*$, the identification sending $(\mu_1,\ldots,\mu_m)$ to $C_0(\Omega;\Rm) \to \R : v \mapsto \sum_{i=1}^m \int_{\Omega} v_i \,d\mu_i$.
Furthermore, this isomorphism is an isometry when $M(\Omega;\Rm)$ is equipped with the norm
\begin{multline*}
\|(\mu_1,\ldots,\mu_m)\|_{M^m} = \sup \left\{ \sum_{i=1}^m \int_{\Omega} v_i \, d\mu_i : v \in C_0(\Omega;\Rm) \text{ and } \|v\|_\infty \leq 1 \right\} \\ = \sup \left\{ \sum_{i=1}^m \int_{\Omega} v_i \, d\mu_i : v \in C_c^\infty(\Omega;\Rm) \text{ and } \|v\|_\infty \leq 1 \right\}.
\end{multline*}
Thus, $\max \{ \|\mu_i\|_M : i =1,\ldots,m \} \leq \|(\mu_1,\ldots,\mu_m)\|_{M^m} \leq \sum_{i=1}^m \|\mu_i\|_M$.
\end{Empty}

\begin{Empty}[$BV(\Omega)$]
\label{2.10}
Let $u \in L_{1,\rmloc}(\Omega)$ and assume that all its partial derivatives $\partial_i u$ in the sense of distributions are signed Borel measures.
In other words, for each $i=1,\ldots,m$, there exists a signed Borel measure $D_i u \in M(\Omega)$ such that $- \int_{\Omega} u \cdot \partial_i \vphi \, d\ssfL^m = \int_{\Omega} \vphi \, d(D_i u)$ for all $\vphi \in \calD(\Omega)$.
This implies in particular that if $v \in C_c^\infty(\Omega;\Rm)$ then
\begin{equation}
\label{eq.5.1}
- \int_{\Omega} u \cdot \rmdiv v \, d\ssfL^m = \sum_{i=1}^m \int_{\Omega} v_i \cdot d(D_i u).
\end{equation}
Recall from \ref{3.2.1} that
\begin{equation*}
\|(D_1u,\ldots,D_mu)\|_{M^m} = \sup \left\{ \int_{\Omega} u \cdot \rmdiv v \, d\ssfL^m : v \in C^\infty_c(\Omega;\Rm) \text{ and } \|v\|_\infty \leq 1 \right\}.
\end{equation*}
Most often we will write $\|Du\|(\Omega)$ instead of $\|(D_1u,\ldots,D_mu)\|_{M^m}$.
Notice that if a sequence $\la u_j \ra_j$ in $L_{1,\rmloc}(\Omega)$ converges in the sense of distributions to $u \in L_{1,\rmloc}(\Omega)$ then $\|Du\|(\Omega) \leq \liminf_j \|Du_j\|(\Omega)$.
\par 
We let $BV(\Omega)$ consist of those $u \in L_{1}(\Omega)$ of which all partial derivatives $\partial_i u$ in the sense of distributions are signed Borel measures.
If $u \in \calD(\Omega)$ then $u \in BV(\Omega)$ and $\|Du\|(\Omega) = \int_\Omega |\nabla u|_2\,d\ssfL^m$.
Finally, we define $\|u\|_{BV} = \|u\|_{L_1(\Omega)} + \|Du\|(\Omega)$ and we recall that $BV(\Omega)[\|\cdot\|_{BV}]$ is a Banach space.
\end{Empty}

\begin{Empty}[$1^*$, $\kappa_m$, and $BV_p(\Omega)$]
\label{2.11}
We assume that $m \geq 2$ and we define $1^* = \frac{m}{m-1} > 1$, the Sobolev conjugate exponent of $m$.
It is useful on many occasions to notice that $1^*$ is also the H\"older conjugate of $m$.
We recall the Gagliardo-Nirenberg-Sobolev inequality: 
\begin{enumerate}
\item[(A)] {\it For every $u \in \calD(\Rm)$ we have $\|u\|_{L_{1^*}(\Rm)} \leq \kappa_m \cdot \|Du\|(\Rm)$.} 
\end{enumerate}
\par 
For a proof, see \eg \cite[2.7.4]{ZIEMER} or \cite[4.5.9(31)]{GMT}, where the best constant is determined, namely $\kappa_m = m^{-1} \cdot \balpha(m)^\frac{-1}{m}$.
In particular, $BV(\Rm) \subset L_{1^*}(\Rm)$.
It is important, however, to note that if $u \in L_{1^*}(\Rm)$ and the distributional derivatives $\partial_i u$ of $u$, $i=1,\ldots,m$, are signed Borel measures on $\Rm$, we do not necessarily have $u \in L_1(\Rm)$.
A folkloric example being $u(x) = \frac{1}{|x|_2^{m-1} \cdot (\log |x|_2)^t}$ for $|x|_2 \geq 2$, where $t > 1$.
This justifies the introduction of the space $BV_p(\Omega)$, for $1 \leq p \leq 1^*$, that consists of those member of $L_p(\Omega)$ whose distributional derivatives $\partial_i u$, $i=1,\ldots,m$, are signed Borel measures on $\Omega$.
Thus, $BV_1(\Omega) = BV(\Omega)$ and $BV_p(\Omega) \subset BV(\Omega)$ if $\ssfL^m(\Omega) < \infty$ but this inclusion fails \eg if $\Omega = \Rm$ and $p=1^*$, as the above example shows.
\begin{enumerate}
\item[(B)] {\it If $u \in BV(\Rm)$ or $u \in BV_{1^*}(\Rm)$ then $\|u\|_{L_{1^*}(\Rm)} \leq \kappa_m \cdot \|Du\|(\Rm)$.}
\end{enumerate}
\par 
{\it Proof.}
In case $u \in BV(\Rm)$, see \eg \cite[theorem 5.10(i)]{EVANS.GARIEPY.2}.
If $u \in BV_{1^*}(\Rm)$, apply \cite[3.47]{AMBROSIO.FUSCO.PALLARA} and \ref{2.9}(A) (with $r=s=1^*$).\cqfd
\par 
It is clear that $BV_{1^*}(\Rm)$ is a Banach space when equipped with the norm $\|u\|_{L_{1^*}(\Rm)} + \|Du\|(\Rm)$.
However, we shall always consider the equivalent norm $\|u\|_{BV_{1^*}} = \|Du\|(\Rm)$.
\end{Empty}

\begin{Empty}[Divergence]
\label{2.3}
If $v : \Omega \to \Rm$ is differentiable at $x \in \Omega$ we let $(\rmdiv v)(x) = \sum_{i=1}^m (\partial_i v_i)(x)$, where $v = (v_1,\ldots,v_m)$.
If $u : \Omega \to \R$ is also differentiable at $x$ then $\rmdiv (u \cdot v)(x) = (\nabla u)(x) \ip v(x) + u(x) \cdot (\rmdiv v)(x)$.
Thus, if $v \in L_{1,\rmloc}(\Omega;\Rm)$, $v$ is differentiable everywhere, $\rmdiv v \in L_{1,\rmloc}(\Omega)$, and $\vphi \in \calD(\Omega)$ then
\begin{equation}
\label{eq.2.1}
\int_{\Omega} \vphi \cdot \rmdiv v \, d\ssfL^m = - \int_{\Omega} (\nabla \vphi) \ip v \, d\ssfL^m.
\end{equation}
Indeed, in that case $\rmdiv (\vphi \cdot v) \in L_1(\Omega)$ and one checks that its integral vanishes, since $\vphi$ has compact support.
\par 
The assumptions we made about $v$ for \eqref{eq.2.1} to hold may be weakened.
For instance, if $E_0, E_\sigma \subset \Rm$, $\ssfH^{m-1}(E_0)=0$, $\ssfH^{m-1} \hel E_\sigma$ is $\sigma$-finite; if $v$ is locally bounded, continuous in $\Rm \setminus E_0$, and pointwise Lipschitz\footnote{This guarantees that $v$ is differentiable a.e., according to Stepanoff's theorem.} in $\Rm \setminus E_\sigma$; and if $\rmdiv v \in L_{1,\rmloc}(\Rm)$ then \eqref{eq.2.1} holds, see \cite[theorem 2.9 and corollary 2.10]{PFE.05}.
In fact, the hypothesis that $\rmdiv v$ be locally Lebesgue-summable can be removed altogether and in that case the integral in the left member of \eqref{eq.2.1} ought to be interpreted in a different averaging sense than a Lebesgue integral, see \eg \cite[theorems 5.19 and 10.9]{PFE.91}, \cite[theorem 3.10]{DEP.PFE.03}, and the monographs \cite{PFEFFER,PFEFFER.12}.
\par 
In any case, if $v$ itself is merely locally Lebesgue-summable then the right member of \eqref{eq.2.1} defines a distribution of order 1, called the distributional divergence of $v$, 
\begin{equation*}
\bdiv v : \calD(\Omega) \to \R : \vphi \mapsto - \int_{\Omega} (\nabla \vphi) \ip v \, d\ssfL^m
\end{equation*}
and the previous paragraph states a set of conditions implying that the distributional divergence of $v$ can be recovered, through an appropriate averaging process, from the pointwise a.e. divergence of $v$.
\par 
In the sequel -- specifically \ref{3.1.3}, \ref{3.2.4}, \ref{4.2.6}, \ref{5.1.4}, and \ref{5.2.2} -- we shall define operators $\bdiv : X(\Omega) \to Y(\Omega)$ for a selection of different Banach spaces $X(\Omega)$ and $Y(\Omega)$, thus, these operators are strictly speaking different mathematical objects, yet the same notation $\bdiv$ is used for each of them.
Maintaining focus should prevent any confusion.
\par 
We say that $v \in L_{1,\rmloc}(\Omega;\Rm)$ is solenoidal if its distributional divergence vanishes identically, \ie $- \int_{\Omega} (\nabla \vphi) \ip v \, d\ssfL^m = 0$ for all test functions $\vphi \in \calD(\Omega)$.
We abbreviate
\begin{equation*}
Z_m(\Omega) = L_{1,\rmloc}(\Omega;\Rm) \cap \{ v : \text{$v$ is solenoidal}\}.
\end{equation*}
\end{Empty}

\begin{Empty}[$(p,1)$-Poincar\'e inequality]
\label{2.6}
Let $1 \leq p \leq 1^*$.
We say that $\Omega \subset \Rm$ satisfies the {\em $(p,1)$-Poincar\'e inequality} if the following holds.
\begin{equation}
\label{eq.pi}
(\exists \bc_p(\Omega) > 0)(\forall u \in C^\infty(\Omega) \cap L_p(\Omega))(\exists y \in \R): \|u - y \cdot \ind_\Omega\|_{L_p(\Omega)} \leq \bc_p(\Omega) \cdot \|\nabla u\|_{L_1(\Omega)}.
\end{equation}
Notice that ``$(\exists y \in \R) \cdots$'' is equivalent to ``$\|[u]\|_p \leq \bc_p(\Omega) \cdot \|\nabla u\|_{L_1(\Omega)}$''.
\begin{enumerate}
\item[(A)] {\it If $\Omega$ satisfies the $(p,1)$-Poincar\'e inequality and $\ssfL^m(\Omega) = \infty$ then necessarily $y=0$ for all $u \in C^\infty(\Omega) \cap L_p(\Omega)$ such that $\|\nabla u\|_{L_1(\Omega)} < \infty$.}
\end{enumerate}
\par 
{\it Proof.}
Apply \ref{2.9}(A) (with $r=s=p$).\cqfd
\begin{enumerate}
\item[(B)] {\it In the definition of the $(p,1)$-Poincar\'e inequality the occurrence of $(\forall u \in C^\infty(\Omega) \cap L_p(\Omega))$ may be replaced by $(\forall u \in BV_p(\Omega))$ if we also replace $\|\nabla u \|_{L_1(\Omega)}$ by $\|Du\|(\Omega)$.}
\end{enumerate}
\par 
{\it Proof.}
Let $u \in BV_p(\Omega)$ and choose a sequence $\la u_j \ra_j$ in $C^\infty(\Omega) \cap L_p(\Omega)$ such that\footnote{To see that is possible, it suffices to inspect the proof of \cite[5.3.3]{ZIEMER} and see that inequality (5.3.3) can be replaced by $\|(f_i \cdot u)_{\veps_i}-f_i \cdot u\|_{L_p(\Omega)} < \veps \cdot 2^{-(i+1)}$ and that (5.3.5) becomes $\|v_\veps-u\|_{L_p(\Omega)} < \veps$.} $\lim_j \|u-u_j\|_{L_p(\Omega)} = 0$ and $\lim_j \|\nabla u_j\|_{L_1(\Omega)} = \|Du\|(\Omega)$.
Then $\|[u]\|_p = \lim_j \|[u_j]\|_p \leq \lim_j \bc_p(\Omega) \cdot \|\nabla u_j\|_{L_1(\Omega)} = \bc_p(\Omega) \cdot \|Du\|(\Omega)$.\cqfd
\begin{enumerate}
\item[(C)] {\it If $1 \leq p_1 < p_2 \leq 1^*$, $\ssfL^m(\Omega) < \infty$, and $\Omega$ satisfies the $(p_2,1)$-Poincar\'e property then it also satisfies the $(p_1,1)$-Poincar\'e inequality.}
\end{enumerate}
\par 
{\it Proof.}
This immediately follows from H\"older's inequality.\cqfd
\par 
The following goes without saying and even better when said.
\begin{enumerate}
\item[(D)] {\it If $\Omega$ satisfies the $(p,1)$-Poincar\'e inequality and $u \in BV_p(\Omega)$ is so that $\|Du\|(\Omega) = 0$ then $u=y \cdot \ind_\Omega$ a.e. for some $y \in \R$.}
\end{enumerate}
\par
We now proceed to giving some examples of non-empty open subsets of $\Rm$ that satisfy the $(p,1)$-Poincar\'e property for some $1 < p \leq 1^*$.
\begin{enumerate}
\item[(1)] $\Omega = \Rm$ satisfies the $(1^*,1)$-Poincar\'e inequality, recall \ref{2.11}(B).
\item[(2)] If $\Omega \subset \R^2$ is the familiar von Koch snowflake (whose boundary is the von Koch curve) then $\Omega$ is a John domain and, consequently, satisfies the $(2,1)$-Poincar\'e property, see \eg \cite{BOJ.88}. In fact, if $m=2$ and $\Omega$ is simply connected and of finite area then it satisfies the $(2,1)$-Poincar\'e property if and only if it is a John domain, see \cite[corollary 1.2]{BUC.KOS.95}. For a comprehensive discussion of these topics, consult \cite{HAJ.KOS.00}.
\item[(3)] Even though the definition of John domain implies these are bounded, unbounded versions have been studied as well and yield other examples of good domains for our purpose, see \eg \cite{HUR.92}.
\item[(4)] In relation with (C) above we note that if $1 \leq p_1 < p_2 \leq 1^*$, $\ssfL^m(\Omega) < \infty$, and $\Omega$ satisfies the $(p_1,1)$-Poincar\'e inequality then it does not necessarily satisfy the $(p_2,1)$ Poincar\'e inequality, see \eg the discussion \cite[p.437]{HAJ.KOS.98} of Nikod\'ym's domains. Subsection \ref{subsec.EX.SCH.RM} illustrates the link with the solvability of the divergence equation.
\end{enumerate}
\par 
The next important class of examples is treated in the number below.
\end{Empty}

\begin{Empty}[$BV$-extension sets]
\label{bv.ext}
We say that $\Omega \subset \Rm$ is a {\em $BV$-extension set} if the following holds.
\begin{equation*}
(\exists \bc_\rmext(\Omega)  > 0) (\forall u \in BV(\Omega))(\exists \hat{u} \in BV(\Omega)) : \hat{u}|_\Omega = u \text{ and } \|D\hat{u}\|(\Rm) \leq \bc_\rmext(\Omega) \cdot \|u\|_{BV(\Omega)}.
\end{equation*}
Note that this is weaker than the notion of extension domain in \cite[3.20]{AMBROSIO.FUSCO.PALLARA}, for instance, we do not insist that $\hat{u}$ depend linearly on $u$.
If $\Omega$ has compact Lipschitz boundary then it is a $BV$-extension set, according to \cite[3.21]{AMBROSIO.FUSCO.PALLARA}.
If $m=2$ and $\Omega$ is bounded and simply connected then it is a $BV$-extension set if and only if its complement is quasi-convex, per \cite{KOS.MIR.SHA.10}.
Below we argue classically as \eg in \cite[3.50, 3.44, and 3.21]{AMBROSIO.FUSCO.PALLARA}.
\begin{enumerate}
\item[(A)] {\it If $\Omega$ is bounded and a $BV$-extension set then
\begin{enumerate}
\item[(a)] For every $k \in \N$ the set $BV(\Omega) \cap \left\{ u : \|u\|_{BV(\Omega)} \leq k \right\}$ is $\|\cdot\|_{L_1(\Omega)}$-compact.
\item[(b)] If one further assumes that $\Omega$ is connected then it satisfies the $(p,1)$-Poincar\'e inequality for all $1 \leq p \leq 1^*$.
\item[(c)] $BV(\Omega) = BV_p(\Omega)$ for all $1 \leq p \leq 1^*$.
\end{enumerate}
}
\end{enumerate}
\par 
{\it Proof.}
{\bf (i)}
Let $\la u_j \ra_j$ be a sequence in $BV(\Omega)$ such that $\sup_j \|u_j\|_{BV(\Omega)} \leq k$ and for each $j$ let $\hat{u}_j \in BV(\Rm)$ be an extension of $u_j$ satisfying $\|D\hat{u}_j\|(\Rm) \leq \bc(\Omega) \cdot k$.
If $U \subset \Rm$ is an arbitrary bounded open set then we infer from (1) above and H\"older's inequality that $\|\hat{u}_j|_U\|_{BV(U)} = \|\hat{u}_j|_U\|_{L_1(U)} + \|D\hat{u}_j|_U\|(U) \leq \|\ind_U\|_{L_m(U)} \cdot \| \hat{u}_j|_U\|_{L_{1^*}(U)} + \|D\hat{u}_j\|(\Rm) \leq \bc(\Omega) \cdot k \cdot \left( \ssfL^m(U)^\frac{1}{m} \cdot \kappa_m + 1 \right)$.
Therefore, $\la \hat{u}_j \ra_j$ is relatively compact in $L_{1,\rmloc}(\Rm)$, according to \cite[3.23]{AMBROSIO.FUSCO.PALLARA}.
In particular, $\la u_j \ra_j$ admits a subsequence that converges to $u \in L_1(\Omega)$ with respect to $\|\cdot\|_{L_1(\Omega)}$.
By lower semicontinuity, $\|Du\|(\Omega) \leq k$.
This proves (a).
\par 
{\bf (ii)}
We prove (b) first in case $p=1$, by contradiction.
Assume if possible that there is a sequence $\la u_j \ra_j$ in $BV(\Omega)$ such that $\|u_j - y_j \cdot \ind_\Omega\|_{L_1(\Omega)} > j \cdot \|Du_j\|(\Omega)$, where $y_j = \dashint_\Omega u_j \, d\ssfL^m$.
Letting $\tilde{u}_j = (u_j - y_j \cdot \ind_\Omega) \cdot \| u_j - y_j \cdot \ind_\Omega \|_{L_1(\Omega)}^{-1} \in BV(\Omega)$, we notice that $1 = \|\tilde{u}_j \|_{L_1(\Omega)} > j \cdot \|D\tilde{u}_j\|(\Omega)$ and $\dashint_\Omega \tilde{u}_j \, d\ssfL^m = 0$ for all $j$.
By {\bf (i)}, a subsequence of $\la \tilde{u}_j \ra_j$ converges to some $\tilde{u} \in BV(\Omega)$ with respect to $\|\cdot\|_{L_1(\Omega)}$.
In particular, $\|\tilde{u}\|_{L_1(\Omega)} = 1$ and $\dashint_\Omega \tilde{u} \, d\ssfL^m = 0$.
Moreover, $\|D\tilde{u}\|(\Omega) = 0$, by lower semicontinuity.
As $\Omega$ is connected, it follows that $\tilde{u}$ is constant a.e., a contradiction.
\par 
{\bf (iii)}
Here, we prove (b) for $p=1^*$ and, in fact, a little bit more, namely that \eqref{eq.pi} holds for all $u \in BV(\Omega)$ instead of merely all $u \in BV_{1^*}(\Omega)$.
Let $u \in  BV(\Omega)$, and let $y \in \R$ be associated with $u$ in the $(1,1)$-Poincar\'e inequality.
Let also $\hat{u} \in BV(\Rm)$ be an extension of $u - y \cdot \ind_\Omega$ as in the definition of $BV$-extension set.
Then it follows from \ref{2.11}(A) that $\|u - y \cdot \ind_\Omega\|_{L_{1^*}(\Omega)} \leq  \|\hat{u}\|_{L_{1^*}(\Rm)} \leq  \kappa_m \cdot \|D\hat{u}\|(\Rm) \leq  \kappa_m \cdot \bc_\rmext(\Omega) \cdot [ \|u-y\cdot \ind_\Omega\|_{L_1(\Omega)} + \|Du\|(\Omega)] \leq \kappa_m \cdot \bc_\rmext(\Omega) \cdot (1 + \bc_1(\Omega)) \cdot \|Du\|(\Omega)$.
\par 
{\bf (iv)}
It ensues from \ref{2.6}(C) and {\bf (iii)} that $\Omega$ satisfies the $(p,1)$-Poincar\'e inequality also for all $1 < p < 1^*$.
This complete the proof of (b).
\par 
{\bf (v)}
Finally, we prove (c).
$\bvs(\Omega) \subset BV_p(\Omega) \subset BV(\Omega) \subset \bvs(\Omega)$, where the first two inclusions follow from H\"older's inequality and the last one follows from {\bf (iii)}.\cqfd
\par 
Assume as in (A)(b) that $\Omega$ is a bounded, connected $BV$-extension set.
It is useful to observe in that case that
\begin{equation*}
\|D\hat{u}\|(\Rm) \leq \bc_\rmext(\Omega) \cdot ( 1 + \bc_{1}(\Omega)) \cdot \|Du\|(\Omega)
\end{equation*}
for all $u \in BV(\Omega)$ such that $\int_\Omega u \,d\ssfL^m = 0$.
This will be used without reference.
\end{Empty}

\numberwithin{Theorem}{subsection}

\section{Consequences of integration by parts and the closed range theorem}
\label{sec.CIPCRT}

\subsection{Lebesgue-summable vector fields and the space \texorpdfstring{$\calF(\Rm)$}{F(Rm)}}
\label{subsec.LSVF}

\begin{Empty}[Lipschitz-free space]
\label{3.1.1}
We let $\rmLip_0(\Rm)$ consist of all Lipschitz functions $u : \Rm \to \R$ vanishing at the origin.
Upon defining $\|u\|_L = \rmLip u$, we notice that $|u(x)| \leq \|u\|_L \cdot |x|_2$ for all $(u,x) \in \rmLip_0(\Rm) \times \Rm$.
Building on this, one checks that $\rmLip_0(\Rm)[\|\cdot\|_L]$ is a Banach space.
Moreover, its unit ball is compact under the topology of local uniform convergence, according to Ascoli's theorem, therefore, also $\sigma(\rmLip_0(\Rm),\rmspan X)$-compact, where $X = \rmLip_0(\Rm)^* \cap \{ \rmev_x : x \in \Rm \}$.
Here, $\rmev_x$ is the evaluation at $x$ defined by $\la u , \rmev_x \ra = u(x)$.
It ensues from Kaijser's theorem \cite{KAI.77} that $\rmLip_0(\Rm)$ is a dual space.
Next, we proceed to describing a predual of $\rmLip_0(\Rm)$.
From now on, we use the alternative symbol $\lseg x \rseg$ to denote the evaluation $\rmev_x$, $x \in \Rm$.
\begin{enumerate}
\item[(A)] {\it The map $\bdelta : \Rm \to \rmLip_0(\Rm)^* : x \mapsto \lseg x \rseg$ is a (non-linear) isometry.}
\end{enumerate}
\par 
{\it Proof.}
Indeed, one trivially checks that $\|\lseg x \rseg  - \lseg y \rseg\|_{\rmLip_0}^* \leq |x-y|_2$ and the reverse inequality is a consequence of $\la u_y , \lseg x \rseg - \lseg y \rseg \ra = 1$, where $u_y(x) = |x-y|_2 - |y|_2$, since $\|u_y\|_L = 1$.\cqfd
\par 
We define $\calF(\Rm) = \rmclos_{\rmLip_0^*}(\rmspan \rmim \bdelta)$ and call it the {\em Lipschitz-free space of $\Rm$}.
It is readily a separable Banach space.
For brevity we shall let $\|\cdot\|_\calF$ be its norm, \ie $\|\cdot\|_\calF = \|\cdot\|_{\rmLip_0}^*$.
\begin{enumerate}
\item[(B)] {\it $\calF(\Rm)^* \cong \rmLip_0(\Rm)$.}
\end{enumerate}
\par 
{\it Proof.}
We associate with $u \in \rmLip_0(\Rm)$ a functional $\beta(u) \in \calF(\Rm)^*$ by declaring that $\la F , \beta(u) \ra = \la u , F \ra$, $F \in \calF(\Rm)$.
Note that $\beta : \rmLip_0(\Rm) \to \calF(\Rm)^*$ is a bounded linear operator and $\|\beta(u)\|_{\calF}^* \leq \|u\|_L$ for all $u$.
Furthermore, if $x,y \in \Rm$ are distinct we observe that $\|F\|_{\calF} = 1$, where $F = |x-y|_2^{-1} \cdot (\lseg x \rseg - \lseg y \rseg)$, since $\bdelta$ is an isometry, therefore, $\|\beta(u)\|_{\calF}^* \geq \la u , F \ra = |x-y|^{-1}_2 \cdot |u(x)-u(y)|$.
Since $x$ and $y$ are arbitrary, it ensues that $\|\beta(u)\|_{\calF}^* \geq \|u\|_L$, whence, $\beta$ is a linear isometry.
It remains to show that show that $\beta$ is surjective.
If $F^* \in \calF(\Rm)^*$ then $u := F^* \circ \bdelta \in \rmLip_0(\Rm)$ and we claim that $\beta(u) = F^*$.
Indeed, for all $F \in \calF(\Rm)$ we have $\la F , \beta(u) \ra = \la F^* \circ \bdelta , F \ra = \la F , F^* \ra$.
The latter equation holds for each $F = \lseg x \rseg$, $x \in \Rm$, by definition of $\bdelta$, therefore, also for each linear combination of such $\lseg x \rseg$, since both sides of the equation are linear in $F$ and, therefore, also for each member $F$ of $\calF(\Rm)$, since both sides of the equation are continuous in $F$.\cqfd
\par 
The following easy lemma serves the purpose to interpret the results of this subsection in terms of distributions but does not extend to the case $\Omega \neq \Rm$.
\begin{enumerate}
\item[(C)] {\it If $u \in \rmLip_0(\Rm)$, $r > 0$, and $u_r = (\Phi_r * u) - (\Phi_r * u)(0)$ then:
\begin{enumerate}
\item[(a)] $u_r \in \rmLip_0(\Rm) \cap C^\infty(\Rm)$ and $\|\nabla u_r\|_\infty = \|u_r\|_L \leq \|u\|_L$;
\item[(b)] $u_r \to u$ uniformly
as $r \to 0^+$;
\item[(c)] if $\sigma > 0$, $\chi \in \calD(\Rm)$, $\ind_{B(0,\sigma)} \leq \chi \leq \ind_{B(0,2 \cdot \sigma)}$, and $\|\nabla \chi \|_\infty \leq \frac{2}{\sigma}$ then $\chi \cdot u_r \in \rmLip_0(\Rm) \cap \calD(\Rm)$ and $\|\chi \cdot u_r\|_L \leq 5 \cdot \|u\|_L$.
\end{enumerate}
}
\end{enumerate}
\par 
{\it Proof.}
{\bf (i)} 
Conclusion (a) is obvious.
\par 
{\bf (ii)}
We have $|u(x) - (\Phi_r*u)(x)| \leq \|u\|_L \cdot \int_{\Rm} |y|_2 \cdot \Phi_r(y) \, d\ssfL^m(y) \leq r \cdot \|u\|_L$, thus, $|u(x)-u_r(x)| \leq 2 \cdot r \cdot \|u\|_L$ for all $x \in \Rm$.
This proves (b).
\par 
{\bf (iii)}
It is obvious that $\chi \cdot u_r \in \rmLip_0(\Rm) \cap \calD(\Rm)$.
For all $x \in \Rm$ we have $| \nabla (\chi \cdot u_r)(x) |_2 \leq |\chi(x)| \cdot |(\nabla u_r)(x)|_2 + |(\nabla \chi)(x)|_2 \cdot |u_r(x)| \leq \|\nabla u_r\|_\infty + \frac{2}{\sigma} \cdot \|u_r\|_L \cdot |x|_2 \cdot \ind_{B(0,2 \cdot \sigma)}(x) \leq 5 \cdot \|u_r\|_L \leq 5 \cdot \|u\|_L$, by (a).\cqfd
\begin{enumerate}
\item[(D)] {\it If $F \in \calF(\Rm)$ and $\la \vphi , F \ra = 0$ for all $\vphi \in  \calD(\Rm) \cap \rmLip_0(\Rm)$ then $F = 0$.}
\end{enumerate}
\par 
{\it Proof.}
Let $\veps > 0$ and choose $\hat{F} \in \rmspan \rmim \bdelta$ such that $\| F - \hat{F} \|_{\rmLip_0}^* < \veps$.
Write $\hat{F} = \sum_{x \in S} \theta_x \cdot \lseg x \rseg$, where $S \subset \Rm$ is finite and $\{ \theta_x : x \in S \} \subset \R$.
Let $u \in \rmLip_0(\Rm)$ be such that $\|u\|_L \leq 1$ and define $u_r$ as in (C) above.
It follows from (C)(b) and the special form of $\hat{F}$ that there exists $r$ such that $|\la u -u_r , \hat{F} \ra| < \veps$.
Let $\sigma > 0$ be such that $S \subset B(0,\sigma)$ and choose $\chi$ as in (C)(c).
Thus, $| \la \chi \cdot u_r , \hat{F} - F \ra| \leq \| \chi \cdot u_r\|_L \cdot \|\hat{F} - F \|_{\rmLip_0}^* < 5 \cdot \veps$.
As $\la \chi \cdot u_r , F \ra = 0$, by assumption, we infer that $| \la u , \hat{F} \ra| < \veps + |\la u_r , \hat{F} \ra| = \veps + | \la \chi \cdot u_f , \hat{F} \ra| = \veps + | \la \chi \cdot u_r , \hat{F} - F \ra| < 6 \cdot \veps$.
Since $u$ is arbitrary, it follows that $\| \hat{F} \|_{\rmLip_0}^* \leq 6 \cdot \veps$ and, in turn, $\|F\|_{\rmLip_0}^* \leq \|F-\hat{F}\|_{\rmLip_0}^* + \|\hat{F}\|_{\rmLip_0}^* < 7 \cdot \veps$.\cqfd
\begin{enumerate}
\item[(E)] {\it Let $\mu$ be a compactly supported signed Borel measure on $\Rm$. Then the linear functional
$
\rmLip_0(\Rm) \to \R : u \mapsto \int_{\Rm} u  \, d\mu
$
is a member of $\calF(\Rm)$ and 
$
\left| \int_{\Rm} u  \, d\mu \right| \leq \|u\|_L \cdot \|\mu\|_{\calF}
$,
where we have identified $\mu$ with the corresponding member of $\calF(\Rm)$.
}
\end{enumerate}
\par 
{\it Proof.}
{\bf (i)}
Notice that the integral makes sense for every pair $(\mu,u)$, since $\mu$ is compactly supported and $u$ is locally bounded.
Let $F_\mu$ be the corresponding linear functional defined on $\rmLip_0(\Rm)$.
If $u \in \rmLip_0(\Rm)$ then 
\begin{equation*}
| \la u , F_\mu \ra | = \left| \int_{\Rm} (u(x) - u(0))  \,d\mu(x) \right| \leq \|u\|_L \cdot \int_{\Rm} |x|_2  \,d|\mu|(x).
\end{equation*}
Thus, $F_\mu \in \rmLip_0(\Rm)^*$.
\par 
{\bf (ii)}
Let $\veps > 0$.
Referring to the compactness of $\rmspt \mu$, we note there exists a finite partition $\calB$ of $\rmspt \mu$ consisting of Borel-measurable sets whose diameter is less than $\veps$.
For each $B \in \calB$ pick $x_B \in B$.
Define $\theta_B = \mu(B)$ and $F = \sum_{B \in \calB} \theta_B \cdot \lseg x_B \rseg$.
For each $u \in \rmLip_0(\Rm)$ note that
\begin{equation*}
\left| \la u , F \ra - \la u , F_\mu \ra \right| = \left| \sum_{B \in \calB} \int_B \big( u(x_B)  - u(x) \big) d\mu(x) \right| 
 < \veps \cdot \|u\|_L \cdot |\mu|(\Rm).
\end{equation*}
Taking the supremum over all $u$ in the unit ball of $\rmLip_0(\Rm)$, we infer that $\| F - F_\mu\|_{\rmLip_0}^* < \veps \cdot |\mu|(\Rm)$.
Since $F \in \calF(\Rm)$ and $\veps$ and arbitrary, we conclude that $F_\mu \in \calF(\Rm)$.
\par 
{\bf (iii)}
The claimed inequality is simply stating that $| \la u , F_\mu \ra| \leq \|u\|_L \cdot \|F_\mu\|_\calF$, which holds by definition of the norm $\|\cdot\|_\calF$.\cqfd
\end{Empty}

\begin{Empty}[A gradient operator]
\label{3.1.2}
Each $u \in \rmLip_0(\Rm)$ is differentiable a.e., according to Rade\-ma\-cher's theorem.
We let $\nabla u = (\partial_1 u ,\ldots, \partial_m u)$.
Furthermore, the partial derivatives a.e. of $u$ represent its distributional partial derivatives, \ie $\int_{\Rm} (\nabla u) \ip v = - \int_{\Rm} u \cdot \rmdiv v \, d\ssfL^m$ for all $v \in C^\infty_c(\Rm;\Rm)$.
\begin{enumerate}
\item[(A)] {\it $\|\nabla u \|_{L_\infty} = \rmLip u$.}
\end{enumerate}
\par 
{\it Proof.}
The inequality $\leq$ follows from the definition of gradient and Rademacher's theorem.
The reverse inequality ensues from smoothing and the lower semicontinuity of $\rmLip(\cdot)$ under pointwise convergence.\cqfd
\begin{enumerate}
\item[(B)] {\it $\bnabla : \rmLip_0(\Rm) \to L_\infty(\Rm;\Rm) : u \mapsto \nabla u$ is a linear isometry.}
\end{enumerate}
\par 
{\it Proof.}
Immediate consequence of (A).\cqfd
\par 
If $m=1$ this is an isometric isomorphism: with $g \in L_\infty(\R;\R)$ we associate $u(x) = \int_{0}^x g \, d\ssfL^1$.
If $m \geq 2$ this operator is not surjective, since members $v$ of its image satisfy $\partial_i v_j = \partial_j v_i$ distributionally for all $i,j=1,\ldots,m$ but not all members of $L_\infty(\Rm;\Rm)$ do.
\end{Empty}

\begin{Empty}[A divergence operator]
\label{3.1.3}
Let $v \in L_1(\Rm;\Rm)$.
For each $u \in \rmLip_0(\Rm)$ note that $(\nabla u) \ip v$ is Lebesgue-summable.
Inspired by \ref{2.3} we define
\begin{equation*}
\bdiv v : \rmLip_0(\Rm) \to \R : u \mapsto - \int_{\Rm} (\nabla u) \ip v \, d\ssfL^m. 
\end{equation*}
\begin{enumerate}
\item[(A)] {\it $\bdiv v \in \rmLip_0(\Rm)^*$ and $\|\bdiv v\|_{\rmLip_0}^* \leq \|v\|_{L_1}$.}
\end{enumerate}
\par 
{\it Proof.}
It suffices to observe that $| \la u , \bdiv v \ra | \leq \| \nabla u \|_{L_\infty} \cdot \|v\|_{L_1} = \|u\|_L \cdot \|v\|_{L_1}$ for all $u$.\cqfd
\begin{enumerate}
\item[(B)] {\it $\bdiv v \in \calF(\Rm)$.}
\end{enumerate}
\par 
{\it Proof.}
If $v \in C^\infty_c(\Rm;\Rm)$ then this follows from \ref{3.1.1}(E) applied with $\mu = \ssfL^m \hel \rmdiv v$ upon recalling that $\la u , \bdiv v \ra = \int_{\Rm} u \cdot \rmdiv v \, d\ssfL^m$ in this case.
In the general case let $\veps > 0$ and choose $w \in C^\infty_c(\Rm;\Rm)$ such that $\|v-w\|_{L_1} < \veps$.
Thus, $\| \bdiv v - \bdiv w \|_{\rmLip_0}^* < \veps$, by (A).
Since $\bdiv w \in \calF(\Rm)$ and $\veps$ is arbitrary, we conclude that $\bdiv v \in \calF(\Rm)$.\cqfd
\par 
It follows that the linear map
\begin{equation*}
\bdiv : L_1(\Rm;\Rm) \to \calF(\Rm) : v \mapsto \bdiv v
\end{equation*}
is well-defined.
Moreover, it is a bounded operator and $\|\bdiv v\|_{\calF} \leq \|v\|_{L_1}$.
\end{Empty}

Conclusion (A) of the following theorem states that in the diagram below the map at the top is adjoint to that at the bottom:
\begin{equation*}
\begin{CD}
L_\infty(\Rm;\Rm) @<{-\bnabla = \bdiv^*}<< \rmLip_0(\Rm) \\
L_1(\Rm;\Rm) @>{\bdiv}>> \calF(\Rm).
\end{CD}
\end{equation*}
Conclusions (B), (C), and particularly conclusion (D) are due to C\'uth-Kalenda-Kaplick\'y \cite{CUT.KAL.KAP.17} and to Godefroy-Lerner \cite{GOD.LER.18} (the former proved it with $\Rm$ replaced by an open convex subset of $\Rm$), albeit with different proofs, for instance, \cite{CUT.KAL.KAP.17} consider the operator $\bnabla^* : L_1^{**}(\Rm;\Rm) \to \rmLip_0^*(\Rm)$.
In subsection \ref{subsec.CHARLF} we will replace the domain $\Rm$ by open sets $\Omega$ and characterize those such that conclusion (B) of following theorem still holds.

\begin{Theorem}
\label{3.1.4}
The following hold.
\begin{enumerate}
\item[(A)] $\bdiv^* = - \bnabla$.
\item[(B)] $\bdiv : L_1(\Rm;\Rm) \to \calF(\Rm)$ is surjective.
\item[(C)] $\ker (\bdiv) = L_1(\Rm;\Rm) \cap Z_m(\Rm)$.
\item[(D)] $\calF(\Rm) \cong \frac{L_1(\Rm;\Rm)}{L_1 \cap Z_m(\Rm)}$.
\end{enumerate}
\end{Theorem}

\begin{proof}
{\bf (i)}
Given $u \in \rmLip_0(\Rm)$ we ought to determine $\bdiv^*(u)$
which is a $w \in L_\infty(\Rm;\Rm)$ such that for all $v \in L_1(\Rm;\Rm)$ we have
\begin{equation*}
\int_{\Rm} w \ip v \,d\ssfL^m = \la v , \bdiv^*(u) \ra = \la u , \bdiv(v) \ra = - \int_{\Rm} (\nabla u) \ip v \,d\ssfL^m.
\end{equation*}
This proves (A).
\par 
{\bf (ii)}
It follows from \ref{3.1.2}(B) and conclusion (A) that $\bdiv^*$ is an isometry. 
Therefore, $\rmim(\bdiv) = \calF(\Rm)$, by \ref{2.2}(A).
This proves (B).
\par 
{\bf (iii)}
If $v \in \ker(\bdiv)$ and $\vphi \in \calD(\Rm)$ then $u := \vphi - \vphi(0) \in \rmLip_0(\Rm)$ and $\nabla u = \nabla \vphi$ so that $- \int_{\Rm} (\nabla \vphi) \ip v \,d\ssfL^m = \la u , \bdiv v \ra = 0$.
Since $\vphi$ is arbitrary, $v$ is solenoidal.
Conversely, assuming that $v$ is Lebesgue-summable and solenoidal we ought to show that $\int_{\Rm} (\nabla u) \ip v \,d\ssfL^m = 0$ for all $u \in \rmLip_0(\Rm)$.
This immediately follows from \ref{3.1.1}(D).
\par 
{\bf (iv)}
Recall that $\bdiv^*$ is an isometry, as a consequence of (A) and \ref{3.1.2}(B).
Therefore, conclusion (D) follows from \ref{2.2}(B) and the present conclusion (C).
\end{proof}

\subsection{Continuous vector fields vanishing at infinity and the space \texorpdfstring{$\SCH_0(\Rm)$}{SCH-0(Rm)}}
\label{subsec.CVFVI}

\begin{Empty}[Vanishing strong charges]
\label{3.2.2}
Recall \ref{2.11} that $BV_{1^*}(\Rm)[\|\cdot\|_{BV_{1^*}}]$ is a Banach space, where $\|u\|_{BV_{1^*}} = \|Du\|(\Rm)$.
\begin{enumerate}
\item[(A)] {\it If $f \in L_m(\Rm)$ and $\bLambda_m(f) : \bvs(\Rm) \to \R : u \mapsto \int_{\Rm} u \cdot f \, d\ssfL^m$ then $\bLambda_m(f) \in \bvs(\Rm)^*$ and $\|\bLambda_m(f)\|_{\bvs}^* \leq \kappa_m \cdot \|f\|_{L_m}$. This defines a bounded linear operator $$\bLambda_m : L_m(\Rm) \to \bvs(\Rm)^*.$$}
\end{enumerate}
\par 
{\it Proof.}
Since $m$ and $1^*$ are H\"older conjugate exponents, $\bLambda_m(f)$ is well-defined and $| \la u , \bLambda_m(f) \ra | \leq \|u\|_{L_{1^*}} \cdot \|f\|_{L_m} \leq \kappa_m \cdot \|f\|_{L_m} \cdot \|u\|_{\bvs}$.\cqfd
\par 
Abbreviating $X = \rmim \bLambda_m$ we infer from the reflexivity of $L_{1^*}(\Rm)$, the Gagliardo-Nirenberg-Sobolev inequality, Banach-Alaoglu's theorem, and the lower semi-continuity of $\|\cdot\|_{\bvs}$ with respect to distributional convergence that the unit ball of $\bvs(\Rm)$ is $\sigma(\bvs(\Rm),X)$-compact.
It ensues from a result of Kaijser \cite{KAI.77} that $\bvs(\Rm)$ is a dual space and we now proceed to describing a predual of $\bvs(\Rm)$.
We define $\SCH_0(\Rm) = \rmclos_{\bvs^*} (\rmim \bLambda_m)$ and we call its members the {\em vanishing strong charges}.
For brevity we shall let $\|\cdot\|_{\SCH_0}$ be its norm, \ie $\|\cdot\|_{\SCH_0} = \|\cdot\|_{\bvs}^*$.
Since $L_m(\Rm)[\|\cdot\|_{L_m}]$ is separable, it follows from (A) that $\SCH_0(\Rm)[\|\cdot\|_{\SCH_0}]$ is separable as well.
\begin{enumerate}
\item[(B)] {\it If $v \in C^\infty_c(\Rm;\Rm)$ then $\|\bLambda_m(\rmdiv v)\|_{\SCH_0} \leq \|v\|_\infty$.}
\end{enumerate}
\par 
{\it Proof.}
Notice that, indeed, $\rmdiv v \in L_m(\Rm)$.
The inequality is a consequence of the definition of $\|u\|_{\bvs}$.\cqfd
\begin{enumerate}
\item[(C)] {\it $\SCH_0(\Rm)^* \cong \bvs(\Rm)$.}
\end{enumerate}
\par 
{\it Proof.}
{\bf (i)}
We consider $\beta : \bvs(\Rm) \to \SCH_0(\Rm)^*$ defined by $\la F , \beta(u) \ra = \la u , F \ra$ for all $(u,F) \in \bvs(\Rm) \times \SCH_0(\Rm)$.
It is readily well-defined, linear, and $\|\beta(u)\|_{\SCH_0}^* \leq \|u\|_{\bvs}$ for all $u$.
\par 
{\bf (ii)}
Here, we establish that $\beta$ is injective.
Indeed, if $\beta(u)=0$ then, in particular, $0 = \la \bLambda_m(f) , \beta(u) \ra = \la u , \bLambda_m(f) \ra = \int_{\Rm} u \cdot f \,d\ssfL^m$ for all $f \in L_m(\Rm)$ and we conclude that $u=0$.
\par 
{\bf (iii)}
Here, we show that $\beta$ is surjective.
If $F^* \in \SCH_0(\Rm)^*$ then $F^* \circ \bLambda_m \in L_m(\Rm)^*$, thus, there exists $u \in L_{1^*}(\Rm)$ such that $\la f , F^* \circ \bLambda_m \ra = \int_{\Rm} u \cdot f \,d\ssfL^m$ for all $f \in L_m(\Rm)$.
Since $\int_{\Rm} u \cdot \rmdiv v \, d\ssfL^m = \la \rmdiv v , F^* \circ \bLambda_m \ra = \la \bLambda( \rmdiv v) , F^* \ra \leq \|\bLambda(\rmdiv v)\|_{\SCH_0} \cdot \|F^*\|_{\SCH_0}^* \leq \|v\|_\infty \cdot \|F^*\|_{\SCH_0}^*$, by (B), for all $v \in C^\infty_c(\Rm;\Rm)$, we infer that $\|u\|_{\bvs} \leq  \|F^*\|_{\SCH_0}^*$, in particular, $u \in \bvs(\Rm)$.
It remains to show that $\beta(u) = F^*$, \ie $\la F , \beta(u) \ra = \la F , F^* \ra$ for all $F \in \SCH_0(\Rm)$.
By density of $\rmim \bLambda_m$ in $\SCH_0(\Rm)$, it is enough to check this holds for $F = \bLambda_m(f)$, $f \in L_m(\Rm)$, namely $\la \bLambda_m(f),u \ra = \la u , \bLambda_m(f) \ra = \int_{\Rm} u \cdot f \, d\ssfL^m = \la f , F^* \circ \bLambda_m \ra = \la \bLambda_m(f) , F^* \ra$.
\par 
{\bf (iv)}
It remains to show that $\beta$ is an isometry.
If $u \in \bvs(\Rm)$ then letting $F^* = \beta(u)$ we infer from {\bf (ii)} and {\bf (iii)} that $\| u \|_{\bvs} \leq \| \beta(u) \|_{\SCH_0}^*$.
The reverse inequality was noted in {\bf (i)}. \cqfd
\begin{enumerate}
\item[(D)] {\it If $u \in \bvs(\Rm)$, $r > 0$, and $u_r = \Phi_r * u$ then:
\begin{enumerate}
\item[(a)] $u_r \in \bvs(\Rm)$ and $\int_{\Rm} |\nabla u_r|_2\,d\ssfL^m = \|u_r\|_{\bvs} \leq \|u\|_{\bvs}$;
\item[(b)] $\|u_r - u\|_{L_{1^*}} \to 0$ as $r \to 0^+$;
\item[(c)] if $\sigma > 0$, $\chi \in \calD(\Rm)$, $\ind_{B(0,\sigma)} \leq \chi \leq \ind_{B(0,2 \cdot \sigma)}$, and $\|\nabla \chi \|_\infty \leq \frac{2}{\sigma}$ then $\chi \cdot u_r \in \bvs(\Rm) \cap \calD(\Rm)$ and $\|\chi \cdot u_r\|_{\bvs} \leq \left( 1 + \frac{4}{m}\right) \cdot \|u\|_{\bvs}$.
\end{enumerate}
}
\end{enumerate}
\par 
{\it Proof.}
{\bf (i)} 
It classically follows from Young's inequality that $u_r \in L_{1^*}(\Rm)$.
Moreover, since $u_r \in C^\infty(\Rm)$, it easily follows from the definition of $\|u_r\|_{\bvs}$ and proper choices of test vector fields $v$  that $\|u_r\|_{\bvs} = \int_{\Rm} |\nabla u_r|_2\, d\ssfL^m$.
If $v \in C^\infty_c(\Rm;\Rm)$ and $x \in \Rm$ then $|(\Phi_r * v)(x)|_2^2 \leq \|v\|_\infty^2$, as follows from Jensen's inequality applied to $\ssfL^m \hel \Phi_r$, hence, $\sum_{i=1}^m \int_{\Rm} v_i \, d(D_iu_r) = \sum_{i=1}^m \int_{\Rm} v_i \cdot \partial_i(\Phi_r * u) \,d\ssfL^m = - \sum_{i=1}^m \int_{\Rm} \partial_i(\Phi_r * v_i) \cdot u \, d\ssfL^m \leq \|\Phi_r * v\|_\infty \cdot \|u\|_{\bvs} \leq \|v\|_\infty \cdot \|u\|_{\bvs}$.
Since $v$ is arbitrary, we conclude that $\|u_r\|_{\bvs} \leq \|u\|_{\bvs}$.
This proves (a).
\par 
{\bf (ii)}
Conclusion (b) is standard.
\par 
{\bf (iii)}
It is obvious that $\chi \cdot u_r \in \calD(\Rm) \subset BV_{1^*}(\Rm)$.
Next, we observe that $\|\nabla \chi\|_{L_m}^m = \int_{\Rm} |\nabla \chi|_2^m \, d\ssfL^m \leq \|\nabla \chi\|_\infty^m \cdot \ssfL^m[B(0,2\cdot \sigma) \setminus B(0,\sigma)] \leq \left( \frac{2}{\sigma} \right)^m \cdot \balpha(m) \cdot \sigma^m \cdot (2^m-1) \leq 4^m \cdot \balpha(m)$.
Therefore, $\|\nabla \chi\|_{L_m} \leq 4 \cdot \balpha(m)^\frac{1}{m}$.
Since $|\nabla (\chi \cdot u_r)|_2 \leq |\chi| \cdot |\nabla u_r|_2 + |u_r| \cdot |\nabla \chi|_2$ everywhere, we have $\int_{\Rm} |\nabla (\chi \cdot u_r)|_2 \, d\ssfL^m \leq \int_{\Rm} |\nabla u_r|_2 \,d\ssfL^m + \|u_r\|_{L_{1^*}} \cdot \|\nabla \chi\|_{L_m} \leq (1 + 4 \cdot \balpha(m)^\frac{1}{m} \cdot \kappa_m) \cdot \|u_r\|_{BV_{1^*}} \leq \left(1 + \frac{4}{m} \right) \cdot \|u\|_{BV_{1^*}}$, by (a).
Conclusion (c) now easily follows from the definition of $\|\chi \cdot u_r\|_{\bvs}$ and the fact that $\chi \cdot u_r$ is smooth.\cqfd
\begin{enumerate}
\item[(E)] {\it If $F \in \SCH_0(\Rm)$ and $\la \vphi , F \ra = 0$ for all $\vphi \in \calD(\Rm)$ then $F = 0$.} 
\end{enumerate}
\par 
{\it Proof.}
Let $\veps > 0$ and choose $f \in L_m(\Rm)$ such that $\|F - \bLambda_m(f)\|_{\bvs}^* < \veps$.
Let $u \in \bvs(\Rm)$ with $\|u\|_{\bvs} \leq 1$.
Letting $u_r$ be as in (D) above, we have $|\la u_r -u , \bLambda_m(f) \ra| \leq \|u_r-u\|_{L_{1^*}} \cdot \|f\|_{L_m}$ and, therefore, we may choose $r$ small enough for $|\la u_r -u , \bLambda_m(f) \ra| < \veps$.
Moreover, there is $\sigma$ large enough for $\|f\|_{L_m}^{1^*} \cdot \int_{\Rm \setminus B(0,\sigma)} |u_r|^{1^*} \, d\ssfL^m < \veps^{1^*}$, hence, if $\chi$ is associated with $\sigma$ as in (D)(c) then $\|f\|_{L_m} \cdot \|\chi \cdot u_r - u_r\|_{L_{1^*}} < \veps$.
Accordingly, $|\la \chi \cdot u_r - u , \bLambda_m(f) \ra| < 2 \cdot \veps$.
Furthermore, $|\la \chi \cdot u_r , F - \bLambda_m(f) \ra| \leq \| \chi \cdot u_r\|_{\bvs} \cdot \|F - \bLambda_m(f)\|_{\bvs}^* < 5 \cdot \veps$, by (D)(c).
As $\la \chi \cdot u_r , F \ra = 0$, by assumption, we infer that $| \la u , \bLambda_m(f) \ra| < 2 \cdot \veps + |\la \chi \cdot u_r , \bLambda_m(f) \ra| = 2 \cdot \veps + | \la \chi \cdot u_r , \bLambda_m(f) - F \ra| < 7 \cdot \veps$.
Since $u$ is arbitrary, it follows that $\|\bLambda_m(f)\|_{\bvs}^* \leq 7 \cdot \veps$ and, in turn, $\|F\|_{\bvs}^* \leq \|F - \bLambda_m(f)\|_{\bvs}^* + \|\bLambda_m(f)\|_{\bvs}^* < 8 \cdot \veps$.\cqfd
\end{Empty}

\begin{Empty}[A gradient operator]
\label{3.2.3}
Let $u \in \bvs(\Rm)$ and define $Du = (D_1u,\ldots,D_mu) \in M(\Rm;\Rm)$. Observe that 
\begin{equation*}
\bnabla : \bvs(\Rm) \to M(\Rm;\Rm) : u \mapsto Du
\end{equation*}
is a well-defined and linear.
\begin{enumerate}
\item[(A)] {\it $\bnabla$ is an isometry.}
\end{enumerate}
\par 
{\it Proof.}
That is to say that $\|u\|_{\bvs} = \|(D_1u,\ldots,D_mu)\|_M$, which immediately follows from the definition of the first norm.\cqfd
\end{Empty}

\begin{Empty}[A divergence operator]
\label{3.2.4}
Let $v \in C_0(\Rm;\Rm)$.
For each $u \in \bvs(\Rm)$ note that each $v_i \in L_1(\Rm,D_iu)$, since $v_i$ is Borel-measurable and bounded.
Thus, we may define
\begin{equation*}
\bdiv v : \bvs(\Rm) \to \R : u \mapsto - \sum_{i=1}^m \int_{\Rm} v_i \, d(D_iu). 
\end{equation*}
This is clearly a linear functional.
\begin{enumerate}
\item[(A)] {\it $\bdiv v \in \bvs(\Rm)^*$ and $\|\bdiv v\|_{\bvs}^* \leq \|v\|_\infty$.}
\end{enumerate}
\par 
{\it Proof.}
It simply follows from the chosen definition of norm of $M(\Rm;\Rm)$ (recall \ref{3.2.1}) that $|\la u , \bdiv v \ra| \leq \|v\|_\infty \cdot \|Du\|_M$.\cqfd
\begin{enumerate}
\item[(B)] {\it $\bdiv v \in \SCH_0(\Rm)$.}
\end{enumerate}
\par 
{\it Proof.}
Let $\veps > 0$.
Choose $w \in C^\infty_c(\Rm;\Rm)$ such that $\|v-w\|_\infty < \veps$.
Thus, $\|\bdiv v - \bdiv w\|_{\bvs}^* < \veps$, according to (A).
Since $\bdiv w = \bLambda_m(\rmdiv w) \in \SCH_0(\Rm)$, by \ref{3.2.2}(B), and $\veps$ is arbitrary, we conclude that $\bdiv v \in \SCH_0(\Rm)$.\cqfd
\par 
It follows that the linear map
\begin{equation*}
\bdiv : C_0(\Rm;\Rm) \to \SCH_0(\Rm) : v \mapsto \bdiv v
\end{equation*}
is well-defined.
Furthermore, it is a bounded operator and $\|\bdiv v\|_{\SCH_0} \leq \|v\|_\infty$.
\end{Empty}

Conclusion (A) of the following theorem states that in the diagram below the map at the top is adjoint to that at the bottom:
\begin{equation*}
\begin{CD}
M(\Rm;\Rm) @<{-\bnabla = \bdiv^*}<< \bvs(\Rm) \\
C_0(\Rm;\Rm) @>{\bdiv}>> \SCH_0(\Rm).
\end{CD}
\end{equation*}
Conclusion (B) is due to Torres and the present author \cite{DEP.TOR.09}, albeit with a very different proof.

\begin{Theorem}
\label{3.2.5}
The following hold.
\begin{enumerate}
\item[(A)] $\bdiv^* = - \bnabla$.
\item[(B)] $\bdiv : C_0(\Rm;\Rm) \to \SCH_0(\Rm)$ is surjective.
\item[(C)] $\ker (\bdiv) = C_0(\Rm;\Rm) \cap Z_m(\Rm)$.
\item[(D)] $\SCH_0(\Rm) \cong \frac{C_0(\Rm;\Rm)}{C_0 \cap Z_m(\Rm)}$.
\end{enumerate}
\end{Theorem}

\begin{proof}
{\bf (i)}
Given $u \in \bvs(\Rm)$ we ought to determine $\bdiv^*(u)$.
This is some $(\mu_1,\ldots,\mu_m) \in M(\Rm;\Rm)$ such that for all $v \in C_0(\Rm;\Rm)$ we have
\begin{equation*}
\sum_{i=1}^m \int_{\Rm} v_i \, d\mu_i = \la v , \bdiv^*(u) \ra = \la u , \bdiv(v) \ra = - \sum_{i=1}^m \int_{\Rm} v_i \, d(D_iu).
\end{equation*}
From the arbitrariness of $v$ ensues that $(\mu_1,\ldots,\mu_l) = -Du$.
This proves (A).
\par 
{\bf (ii)}
It follows from \ref{3.2.3}(A) and conclusion (A) that $\bdiv^*$ is an isometry. 
Therefore, $\rmim(\bdiv) = \SCH_0(\Rm)$, by \ref{2.2}(A).
This proves (B).
\par 
{\bf (iii)}
If $v \in \ker(\bdiv)$ then $v$ is solenoidal, since $\calD(\Rm) \subset \bvs(\Rm)$.
Conversely, if $v \in C_0(\Rm;\Rm)$ is solenoidal we ought to show that $\la u , \bdiv v \ra = 0$ for all $u \in \bvs(\Rm)$.
This immediately follows from \ref{3.2.2}(E).
\par 
{\bf (iv)}
Recall that $\bdiv^*$ is an isometry, as a consequence of (A) and \ref{3.2.3}(A).
Therefore, conclusion (D) follows from \ref{2.2}(B) and (C).
\end{proof}

\section{The Lipschitz-free space \texorpdfstring{$\calF(\Omega)$}{F(Omega)}}
\label{sec.LFS}

Throughout this section we assume that $\Omega \subset \Rm$ is open and $0 \in \Omega$.
Sometimes we shall also assume that $\Omega$ is admissible (see \ref{4.2.1} for a definition) and/or that $\Omega$ is connected by $\Lambda$-pencils or uniformly connected by pencils (see \ref{4.2.4} for definitions).
These assumptions will be explicitly stated whenever they are in effect.

\subsection{The localized topology \texorpdfstring{$\calP_{\rmLip}$}{P-Lip}}
\label{subsec.LF}

\begin{Empty}[$\rmLip_0(\Omega)$ and $\calF(\Omega)$]
\label{4.1.1}
The Lipschitz-free space as defined in \ref{3.1.1} makes sense when replacing $(\Rm,0)$ by any pointed metric space, \ie $(X,o_X)$ where $X$ is a metric space and $o_X \in X$, see \eg \cite{GOD.15}.
Here, we restrict to a couple $(\Omega,0)$, where $\Omega \subset \Rm$ is open and $0 \in \Omega$.
$\rmLip_0(\Omega)$ is the vector space consisting of Lipschitz functions $\Omega \to \R$ that vanish at $0$.
We let $\|u\|_L = \rmLip u$, $u \in \rmLip_0(\Omega)$, so that $\rmLip_0(\Omega)[\|\cdot\|_L]$ is a Banach space.
Reasoning as in \ref{3.1.1}(A), we observe that $\bdelta : \Omega \to \rmLip_0(\Omega)^* : x \mapsto \lseg x \rseg$ is an isometry.
As before, we define the Lipschitz-free space $\calF(\Omega)$ to be the closure in $\rmLip_0(\Omega)^*$ of the span of the range of $\bdelta$ and we abbreviate $\|\cdot\|_\calF = \|\cdot\|_{\rmLip_0}^*$.
As in \ref{3.1.1}(B), $\calF(\Omega)^* \cong \rmLip_0(\Omega)$.
\end{Empty}

\begin{Empty}[The localized topology $\calP_{\rmLip}$] 
\label{4.1.2}
We refer to \cite{DEP.26c} for the terminology.
We consider on $\rmLip_0(\Omega)$ the locally convex topology $\calS$ of pointwise convergence.
It is generated by the (separating) family of seminorms $\la |\lseg x \rseg| \ra_{x \in \Omega}$.
As we will, in fact, need a {\it filtering} family of seminorms generating $\calS$, we introduce the following notation.
We let $\rmFin(\Omega) = \calP(\Omega) \cap \{ E : E \text{ is finite}\}$.
For $E \in \rmFin(\Omega)$ and $u \in \rmLip_0(\Omega)$ we also define $p_E(u) = \max\{ |u(x)| : x \in E \}$.
Thus, $\calS$ is generated by the filtering family $\la p_E \ra_{E \in \rmFin(\Omega)}$ of seminorms.
\par 
Next, we observe that the norm $\|\cdot\|_L$ is lower-semicontinuous with respect to $\calS$.
We abbreviate $C_k = \rmLip_0(\Omega) \cap \{ u : \|u\|_L \leq k\}$, $k \in \N$, and $\calC = \{ C_k : k \in \N \}$.
One readily checks that $\calC$ is a localizing family on $\rmLip_0(\Omega)$ (see \cite[1.2]{DEP.26c} for the definition). 
We denote by $\calP_{\rmLip}$ the localized locally convex topology on $\rmLip_0(\Omega)$ associated with $\calS$ and $\calC$ (\ie $X = \rmLip_0(\Omega)$ and $\calT = \calS$).
We now gather the properties of $\calP_{\rmLip}$ that ensue from \cite{DEP.26c} to which we refer for the notation.
\begin{enumerate}
\item[(A)] {\it Let $\la u_j \ra_j$ be a sequence in, and $u$ a member of, $\rmLip_0(\Omega)$. The following are equivalent.
\begin{enumerate}
\item[(a)] $\lim_j u_j = u$ with respect to $\calP_{\rmLip}$.
\item[(b)] $\lim_j u_j = u$ with respect to $\calS$ (\ie pointwise) and $\sup_j \|u_j\|_L < \infty$.
\end{enumerate}
}
\end{enumerate}
\par 
{\it Proof.}
This is direct consequence of \cite[3.1(A)]{DEP.26c}.\cqfd
\begin{enumerate}
\item[(B)] {\it A set $B \subset \rmLip_0(\Omega)$ is $\calP_{\rmLip}$-bounded if and only if $\sup \{ \|u\|_L : u \in B \} < \infty$. In particular, $\la C_k \ra_k$ is a fundamental system of $\calP_{\rmLip}$-bounded sets in $\rmLip_0(\Omega)$.}
\end{enumerate}
\par 
{\it Proof.}
The first conclusion follows from \cite[3.1(B)]{DEP.26c} upon noticing that if $\Gamma = \sup \{ \|u\|_L : u \in B \}$ and $u \in B$ then $|u(x)| \leq \Gamma \cdot |x|_2$ for all $u \in B$, since $u(0)=0$.
The second conclusion trivially follows from the first.\cqfd
\begin{enumerate}
\item[(C)] {\it The topological spaces $C_k[\calS \hel C_k]$ are metrizable and compact, $k \in \N$.}
\end{enumerate}
\par 
{\it Proof.}
{\bf (i)}
We abbreviate $\calS_k = \calS_k \hel C_k$.
Let $\la x_i \ra_i$ be a dense sequence in $\Omega$ and define $d : C_k \times C_k \to \R^+$ by the formula $d(u_1,u_2) = \sum_i \frac{|u_1(x_i)-u_2(x_i)|}{2^i \cdot (1 + |u_1(x_i)-u_2(x_i)|)}$.
Each term of this normally convergent series is readily $\calS_k \times \calS_k$-continuous, therefore, so is $d$.
Accordingly, the identity of $C_k$ is $(\calS_k,\calT_{k,d})$-continuous, where $\calT_{k,d}$ is the topology of $C_k$ associated with $d$.
It remains to show that it is also $(\calT_{k,d},\calS_k)$-continuous. 
If $u,u_1,u_2,\ldots$ are members of $C_k$ such that $d(u_j,u) \to 0$ as $j \to \infty$ then $|u_j(x_i)-u(x_i)| \to 0$ as $j \to \infty$ for all $i$, thus, given $x \in \Omega$ and $\eta > 0$ there exists $i$ such that $|x-x_i|_2 < \eta$, hence, $\limsup_j |u_j(x)-u(x)| \leq \eta \cdot (\rmLip u_j + \rmLip u) \leq 2 \cdot \eta \cdot k$. 
As $\eta$ is arbitrary, we conclude that $\la u_j \ra_j$ converges pointwise to $u$. 
It follows from (A) that $\la u_j \ra_j$ converges to $u$ with respect to $\calP_{\rmLip}$, whence, also with respect to $\calS_k$.
This completes the proof that $\calS_k$ is metrizable.
\par 
{\bf (ii)}
It remains to show that $\calS_k$ is compact.
In view of {\bf (i)}, it suffices to show that it is sequentially compact.
Let $\la K_i \ra_i$ be an increasing sequence of compact sets whose union is $\Omega$.
Given a sequence $\la u_j \ra_j$ in $C_k$, Ascoli's theorem applies to any subsequence of $\la u_j|_{K_i} \ra_j$ to yield a subsequence converging uniformly on $K_i$ to a function $K_i \to \R$ whose Lipschitz constant does not exceed $k$.
By induction on $i$ and a diagonal argument, we obtain a subsequence of $\la u_j \ra_j$ that converges uniformly on each $K_i$ (hence, also pointwise) to $u : \Omega \to \R$ such that $\rmLip u \leq k$ and, clearly, $u(0)=0$.
By (B), this subsequence converges to $u$ with respect to $\calS_k$.\cqfd
\begin{enumerate}
\item[(D)] {\it $\rmLip_0(\Omega)[\calP_{\rmLip}]$ is sequential.}
\end{enumerate}
\par 
{\it Proof.}
This is an immediate consequence of \cite[4.4(A)]{DEP.26c} and (C) above.\cqfd
\begin{enumerate}
\item[(E)] {\it $\rmLip_0(\Omega)[\calP_{\rmLip}]$ is none of Fr\'echet-Urysohn, barrelled, and bornological.}
\end{enumerate}
\par 
{\it Proof.}
Note that \cite[3.8]{DEP.26c} applies, according to (C) above.
Observe that $\|\cdot\|_L$ is not sequentially $\calS$-continuous (choose a null-sequence $\la r_j \ra_j$ of positive real numbers and define and $u_j(x) = \max \{ r_j - \rmdist(x,\Rm \setminus B(0,r_j),0\}$).
It then follows from \cite[3.8(A)]{DEP.26c} that $\rmLip_0(\Omega)[\calP_{\rmLip}]$ is neither Fr\'echet-Urysohn nor bornological.
In turn, we infer from \cite[3.8(B)]{DEP.26c} that $\rmLip_0(\Omega)[\calP_{\rmLip}]$ is not barrelled.\cqfd
\begin{enumerate}
\item[(F)] {\it Let $F : \rmLip_0(\Omega) \to \R$ be linear. The following are equivalent.
\begin{enumerate}
\item[(a)] $F$ is $\calP_{\rmLip}$-continuous.
\item[(b)] $(\forall \veps > 0)(\exists E \in \rmFin(\Omega))(\exists \theta > 0)(\forall u \in \rmLip_0(\Omega)): | \la u,F \ra| \leq \theta \cdot p_E(u) + \veps \cdot \|u\|_L$.
\end{enumerate}
}
\end{enumerate}
\par 
{\it Proof.}
It suffices to observe that \cite[3.6]{DEP.26c} applies.\cqfd
\begin{enumerate}
\item[(G)] {\it $\rmLip_0(\Omega)[\calP_{\rmLip}]^*$ equipped with its strong topology is a Banach space whose norm is $\vvvert F \vvvert = \sup \{ |\la u,F \ra| : \|u\|_L \leq 1 \}$.}
\end{enumerate}
\par 
{\it Proof.}
The strong topology on $\rmLip_0(\Omega)[\calP_{\rmLip}]^*$ being that of uniform convergence on $\calP_{\rmLip}$-bounded sets it follows from (B) above and the inclusions $C_k \subset k \cdot C_1$ for all positive integers $k$ that the strong topology is normed as stated.
The completeness follows as in \cite[5.4, 5.3]{DEP.26c}.\cqfd
\begin{enumerate}
\item[(H)] {\it $\rmLip_0(\Omega)[\calP_{\rmLip}]$ is semireflexive.}
\end{enumerate}
\par 
{\it Proof.}
According to (B) and (C) above, this follows from \cite[7.4]{DEP.26c}.\cqfd
\end{Empty}

\begin{Theorem}
\label{4.1.3}
$\calF(\Omega) = \rmLip_0(\Omega)[\calP_{\rmLip}]^*$.
\end{Theorem}

\begin{proof}
{\bf (i)}
Recall that members of $\calF(\Omega)$ are linear forms $F : \rmLip_0(\Omega) \to \R$ that are $\|\cdot\|_L$-continuous and lie in the $\|\cdot\|_L^*$-closure of the span of $\bdelta$.
Here, we show that such $F$ is $\calP_{\rmLip}$-continuous, by means of \ref{4.1.2}(F).
Let $\veps > 0$.
There exists a member $G$ of the span of $\bdelta$ such that $\|F-G\|_L^* < \veps$. 
There is $E \in \rmFin(\Omega)$ and for each $x \in E$ there is $\theta_x \in \R$ such that $G = \sum_{x \in E} \theta_x \cdot \lseg x \rseg$.
Thus, for all $u \in \rmLip_0(\Omega)$ we have
\begin{equation*}
| \la u , F \ra| \leq |\la u , G \ra| + \|F-G\|_L^* \cdot \|u\|_L \leq \sum_{x \in E} |\theta_x| \cdot |u(x)| + \veps \cdot \|u\|_L \leq \theta \cdot p_E(u) + \veps \cdot \|u\|_L,
\end{equation*}
where $\theta = 1 + \sum_{x \in E} |\theta_x|$.
\par 
{\bf (ii)}
It ensues from {\bf (i)} that we may consider the identity $\calF(\Omega) \to \rmLip_0(\Omega)[\calP_{\rmLip}]^*$.
Moreover, the identity is an isometry, as follows from the definition of the norms $\|\cdot\|_L^*$ in $\calF(\Omega)$ and $\vvvert\cdot\vvvert$ on $\rmLip_0(\Omega)[\calP_{\rmLip}]^*$ (recall \ref{4.1.2}(G)).
In particular, the range $W$ of the identity in $\rmLip_0(\Omega)[\calP_{\rmLip}]^*$ is closed and in order to complete the proof it remains only to show that $W$ is dense.
To this end, we apply Hahn-Banach's theorem as in \cite[3.5]{RUDIN}: we ought to show that if $\alpha \in \rmLip_0(\Omega)[\calP_{\rmLip}]^*[\vvvert\cdot\vvvert]^*$ vanishes on $W$ then it vanishes identically.
By \ref{4.1.2}(H), there exists $u \in \rmLip_0(\Omega)$ such that $\alpha = \rmev_u$.
For each $x \in \Omega$, $\lseg x \rseg \in \calF(\Omega) = W$, thus, $u(x) = \la \lseg x \rseg, \rmev_u \ra = 0$.
Therefore, $\alpha = 0$ and the proof is complete.

\end{proof}

\begin{Remark}
Note that $\calF(\Omega)^* \cong \rmLip_0(\Omega) \cong \rmLip_0(\Omega)[\calP_{\rmLip}]^{**}$.
The former holds, by construction of the Lipschitz-free space (recall \ref{4.1.1}), and the latter holds, by \ref{4.1.2}(H).
However, it is unknown in general whether a predual of $\rmLip_0(\Omega)$ is unique and the special cases considered in \cite{WEA.18} do not apply to proving theorem \ref{4.1.3}.
\end{Remark}

\subsection{Characterization of those \texorpdfstring{$\Omega$}{Omega} such that \texorpdfstring{$\calF(\Omega) \simeq L_1(\Omega;\Rm)/(\ker \bdiv)$}{F(Omega) = L-1(Omega;Rm)/(ker div)}}
\label{subsec.CHARLF}

\begin{Empty}[Assumptions on \texorpdfstring{$\Omega$}{Omega}]
\label{4.2.1}
Our goal in this subsection is to eventually relate $\calF(\Omega)$ to $L_1(\Omega;\Rm)/(\ker \bdiv)$, similarly to what is done in subsection \ref{subsec.LSVF}.
We note that for an arbitrary subset $\Omega$ of $\Rm$ the definition of $L_1(\Omega;\Rm)$ depends only on $\Omega$ up to a Lebesgue-null set.
In view of this remark, it is natural to require in this subsection, as we do, that $\ssfL^m(\rmbdry \Omega) = 0$.
We call such $\Omega$ an {\em admissible} open set.
\end{Empty}

\begin{Empty}[Curves]
\label{4.2.2}
A {\em curve} in $\Rm$ is a Lipschitz mapping $\gamma : [0,1] \to \Rm$.
A curve is an {\em arc} if it is injective.
If $\gamma$ is an arc then $\int_{\rmim \gamma} f(x) \, d\ssfH^1(x) = \int_0^1 f(\gamma(t)) \cdot |\gamma'(t)|_2 \, d\ssfL^1(t)$ for all Borel-measurable $f : \Rm \to \R$, where $\ssfH^1$ is the 1-dimensional Hausdorff measure in $\Rm$, see \eg \cite[3.2.6]{GMT}.
In particular, $\ssfH^1(\rmim \gamma) = \int_0^1 |\gamma'(t)|_2 \, d\ssfL^1(t)$.
\end{Empty}

\begin{Empty}[Quasi-convexity]
\label{4.2.3}
Let $\Lambda > 0$.
We say that a set $C \subset \Rm$ is {\em $\Lambda$-quasi-convex} if the following holds.
For all $a,b \in C$ there exists an arc $\gamma :[0,1] \to \Rm$ such that $\rmim \gamma \subset C$, $\gamma(0)=a$, $\gamma(1)=b$, and $\int_0^1 |\gamma'(t)|_2 d\ssfL^1(t) \leq \Lambda \cdot |a-b|_2$.
\par 
If $m=1$ then $C$ is quasi-convex if and only if $C$ is an interval if and only if $C$ is convex.
Assume that $m \geq 2$, let $S_j = \rmbdry B(0,2^j)$, $j \in \N$, let $H \subset \Rm$ be a hyperplane (thus, containing $0$), and define $E = H \cup \left( \cup_{j \in \N} S_j \right)$.
One checks that if $C = \Rm \cap \{x : \rmdist(x,E) \leq \eta \}$, $\eta \geq 0$, then $C$ is quasi-convex.
\end{Empty}

\begin{Empty}[Pencils of curves]
\label{4.2.4}
This paragraph sets forth notions from \cite{PAO.STE.12}.
We let $\Gamma$ consist of all curves in $\Rm$ and we equip $\Gamma$ with a semi-metric $d$ defined as follows: $$d(\gamma_1,\gamma_2) = \inf\{ \| \gamma_1 - \gamma_2 \circ \phi\|_\infty : \phi : [0,1] \to [0,1] \text{ is an increasing bijection}\}.$$
Two curves $\gamma_1$ and $\gamma_2$ are called {\em equivalent} if $d(\gamma_1,\gamma_2) = 0$; this is equivalent to saying that $\gamma_1 \circ \phi_1 = \gamma_2 \circ \phi_2$ for some surjective, continuous, non-decreasing functions $\phi_i : [0,1] \to [0,1]$, $i=1,2$.
The set of equivalence classes of curves is denoted $\bGamma$ and equipped with the distance induced by $d$ in the obvious way.
This defines, in particular, a Borel structure on $\bGamma$.
It seems relevant to the results of \cite{SMI.93} and \cite{PAO.STE.12} to record the following useful fact: If $\mu$ is a finite Borel measure on $\bGamma$ and $A \subset \Rm$ is Borel-measurable then the function $\bGamma \to \R : \bgamma \mapsto \int_{\Rm} \rmcard(\gamma^{-1}(x) \cap A) d\ssfH^1(x)$ is $\mu$-measurable.
\par 
Let $\gamma \in \bgamma \in \bGamma$.
It is easy to see that the endpoints $\gamma(0)$, $\gamma(1)$ 
depend only on $\bgamma$ and not on the choice of a representative $\gamma$.
Therefore, $\bGamma_{a,b} = \bGamma \cap \{ \bgamma : \gamma(0)=a \text{ and } \gamma(1)=b \}$ is well-defined.
\par 
Let $\Lambda > 0$ and $a, b \in \Omega$.
A {\em $\Lambda$-pencil of curves in $\Omega$ from $a$ to $b$} is a Borel probability measure $\mu$ on $\bGamma$ satisfying the following properties:
\begin{enumerate}
\item[(1)] $\mu$-almost every member of $\bGamma$ belongs to $\bGamma_{a,b}$ and admits an arc as a representative;
\item[(2)] The Borel measure $\nu : \calB(\Rm) \to [0,\infty]$ defined by $\nu(A) = \int_{\bGamma} \ssfH^1(A \cap \rmim \gamma)d\mu(\bgamma)$ satisfies:
\begin{enumerate}
\item[(i)] $\nu \ll \ssfL^m$;
\item[(ii)] $\rmspt \nu \subset \rmclos \Omega$;
\item[(iii)] $\nu(\Rm) \leq \Lambda \cdot |b-a|_2$. 
\end{enumerate}
\end{enumerate}
\par 
We say that $\Omega$ is {\em connected by $\Lambda$-pencils} if for all distinct $a,b \in \Omega$ there exists a $\Lambda$-pencil of curves in $\Omega$ from $a$ to $b$.
We also say that $\Omega$ is {\em uniformly connected by pencils} if there exists $\Lambda > 0$ such that $\Omega$ is connected by $\Lambda$-pencils.
\par
A similar notion was introduced by Semmes in \cite{SEM.96}.
The following holds in case the measure $\ssfL^m \hel \Omega$ is doubling and $\Omega$ is admissible: $\Omega$ is uniformly connected by pencils if and only if $\rmclos \Omega$ supports a weak 1-Poincar\'e inequality, see \cite{FAS.ORP.19} and \cite{DUR-CAR.ERI-BIQ.KOR.SHA.21}.
\par 
It is not hard to infer from Chebyshev's inequality and Ascoli's theorem that if $\Omega$ is connected by $\Lambda$-pencils then $\Omega$ is $\Lambda$-quasi-convex.

\end{Empty}

\begin{Empty}[A gradient operator]
\label{4.2.5}
Given $u \in \rmLip_0(\Omega)$ we recall that $\nabla u = (\partial_1u,\ldots,\partial_mu)$ is defined a.e. and belongs to $L_\infty(\Omega;\Rm)$ whose norm $\|\cdot\|_{L_\infty(\Omega)}$ is abbreviated as $\|\cdot\|_{L_\infty}$.
\begin{enumerate}
\item[(A)] {\it $\|\nabla u \|_{L_\infty} \leq \rmLip u$ for all $u \in \rmLip_0(\Omega)$.}
\end{enumerate}
\par 
{\it Proof.}
Notice that if $u$ is differentiable at $x \in \Omega$ then $|\nabla u(x)|_2 \leq \rmLip u$ is a consequence of the fact that $\Omega$ is open and of the definition of derivative applied in the direction of $(\nabla u)(x)$.
Thus, $\|\nabla u\|_{L_\infty} \leq \rmLip u$, by Rademacher's theorem.\cqfd
\begin{enumerate}
\item[(B)] {\it $\bnabla : \rmLip_0(\Omega) \to L_\infty(\Omega;\Rm) : u \mapsto \nabla u$ is linear.
If $\Omega$ is admissible and if there is $\Lambda > 0$ such that $\Omega$ is connected by $\Lambda$-pencils then $\rmLip u \leq \Lambda \cdot \|\nabla u\|_{L_\infty}$ for all $u \in \rmLip_0(\Omega)$.}
\end{enumerate}
\par 
{\it Proof.}
{\bf (i)}
That $\bnabla$ be well-defined follows from (A) and its linearity is obvious.
\par 
{\bf (ii)} 
Fix $u \in \rmLip_0(\Omega)$.
We ought to show the following: For all distinct $a,b \in \Omega$ we have $|u(b)-u(a)| \leq \Lambda \cdot \|\nabla u\|_{L_\infty(\Omega)} \cdot |b-a|_2$.
To this end, we fix $a,b \in \Omega$, we abbreviate $C = \rmclos \Omega$, and we let $\mu$ be a $\Lambda$-pencil of curves in $C$ from $a$ to $b$ whose existence ensues from the hypothesis.
\par 
{\bf (iii)}
Here, we show that $\rmim \gamma \subset C$ for $\mu$-a.e. $\bgamma \in \bGamma$ and $\gamma \in \bgamma$.
By \ref{4.2.4}(2)(ii) we have $0 = \nu(\Rm \setminus C) = \int_{\bGamma} \ssfH^1[\rmim(\gamma) \setminus C] \,d\mu(\gamma)$.
Hence, $\ssfH^1[\rmim(\gamma) \setminus C] = 0$ for $\mu$-a.e. $\gamma$ and, by \ref{4.2.4}(1), we may also assume that $\gamma$ is an arc.
For such $\gamma$ the set $O = \gamma^{-1}(E)$ is open in $[0,1]$, where $E = \rmim(\gamma) \setminus C$.
If it were non-empty then $\gamma'(t) \neq 0$ for a.e. $t \in O$, since $\gamma$ is an arc, whence $0 < \int_O |\gamma'(t)|_2\,d\ssfL^1(t) = \ssfH^1[\rmim(\gamma) \setminus C]$, a contradiction.
Accordingly, $O = \emptyset$, hence, $E = \emptyset$, \ie $\rmim \gamma \subset C$.
\par 
{\bf (iv)} 
Here, we let $\hat{u}$ be a Lipschitz extension of $u$ to $\Rm$, we define a Borel-measurable set 
\begin{equation*}
Z = C \cap \{ x : \hat{u} \text{ is not differentiable at } x\},
\end{equation*}
and we claim that $\ssfH^1(Z \cap \rmim \gamma) = 0$ for $\mu$-a.e. $\bgamma \in \bGamma$ and $\gamma \in \bgamma$.
Indeed, it ensues from Rademacher's theorem that $\ssfL^m(Z)=0$ and, in turn, from \ref{4.2.4}(2)(i) that $0 = \nu(Z) = \int_{\bGamma} \ssfH^1(Z \cap \rmim \gamma) \,d\mu(\gamma)$.
\par 
{\bf (v)} 
We infer from \ref{4.2.4}(1), {\bf (iii)}, and {\bf (iv)} that $\mu$-a.e. $\bgamma \in \bGamma$ admits a representative $\gamma \in \bgamma$ which is an arc such that $\gamma(0)=a$, $\gamma(1)=b$, $\rmim \gamma \subset C$, and $\ssfH^1(Z \cap \rmim \gamma) = 0$.
We notice that if $\gamma$ is differentiable at $t \in [0,1]$ and $\hat{u}$ is differentiable at $\gamma(t)$ then $(\hat{u} \circ \gamma)'(t) = (\nabla \hat{u})(\gamma(t)) \ip \gamma'(t)$.
Observe that this condition occurs at $\ssfL^1$-a.e. $t$, since $\ssfH^1(Z \cap \rmim \gamma ) = 0$ and $\gamma$ is an arc.
Accordingly,
\begin{multline*}
|u(b)-u(a)|  = | \hat{u}(b) - \hat{u}(a)| 
 = \left| \int_{\bGamma} \hat{u}(\gamma(1)) - \hat{u}(\gamma(0))\, d\mu(\gamma)  \right| \\
 = \left| \int_{\bGamma} d\mu(\gamma) \int_0^1 \frac{d}{dt}\, \hat{u}(\gamma(t)) \,d\ssfL^1(t)\right|
 \leq \int_{\bGamma} d\mu(\gamma) \int_0^1 | (\nabla \hat{u})(\gamma(t)) \ip \gamma'(t) | \,d\ssfL^1(t) \\
 \leq \int_{\bGamma} d\mu(\gamma) \int_{\Rm} | \nabla \hat{u} |_2 \, d\ssfH^1 \hel \rmim \gamma 
 = \int_{C} | \nabla \hat{u} |_2 \,d\nu 
 \leq \| \nabla \hat{u} \|_{L_\infty(C)} \cdot \nu(C) \\
 \leq \| \nabla \hat{u} \|_{L_\infty(C)} \cdot \Lambda \cdot |b-a|_2.
\end{multline*}
Recalling that $\Omega$ is an admissible domain \ref{4.2.1}, we see that $\|\nabla \hat{u}\|_{L_\infty(C)} = \|\nabla \hat{u}\|_{L_\infty(\Omega)} = \|\nabla u\|_{L_\infty(\Omega)}$.
Since $a$ and $b$ are arbitrary, the proof is complete.\cqfd
\end{Empty}

\begin{Empty}[A {\it partial} divergence operator]
\label{4.2.6}
Let $v \in L_1(\Omega;\Rm)$.
For each $u \in \rmLip_0(\Omega)$ the function $(\nabla u) \ip v$ is Lebesgue-summable in $\Omega$.
Thus, the following is a well-defined linear form:
\begin{equation*}
\bdiv v : \rmLip_0(\Omega) \to \R : u \mapsto - \int_{\Omega} (\nabla u) \ip v \,d\ssfL^m.
\end{equation*}
We call this a {\it partial} divergence of $v$ because it does not account for the boundary values of $v$.
If $\Omega$ is, say, a bounded open set of finite perimeter, $u \in \rmLip(\Omega)$, and $v : \Omega \to \Rm$ is Lipschitz then $- \int_{\Omega} (\nabla u) \ip v \,d\ssfL^m = \int_\Omega u \cdot \rmdiv v \, d\ssfL^m - \int_{\partial_* \Omega} u \cdot ( v \ip \vec{n}_\Omega) \,d\ssfH^{m-1}$.
However, as we merely assume that $v \in L_1(\Omega;\Rm)$ in general, boundary values of $v$ do not make sense at first sight and we absorb the corresponding term in the definition of $\bdiv$.
\begin{enumerate}
\item[(A)] {\it $| \la u , \bdiv v \ra| \leq \|u\|_L \cdot \|v\|_{L_1(\Omega)}$ for all $u$.}
\end{enumerate}
\par 
{\it Proof.}
Trivial, in view of \ref{4.2.5}(A).\cqfd
\begin{enumerate}
\item[(B)] {\it $\bdiv v$ is $\calP_{\rmLip}$-continuous and $\vvvert \bdiv v \vvvert \leq \|v\|_{L_1(\Omega)}$.}
\end{enumerate}
\par 
{\it Proof.}
{\bf (i)}
There are at least two ways to show that $\bdiv v$ is $\calP_{\rmLip}$-continuous.
One can argue as in the proof of \ref{3.1.3}(B) to show that $\bdiv v \in \calF(\Omega)$ and then refer to theorem \ref{4.1.3}.
Instead, we shall check that $\bdiv$ satisfies condition \ref{4.1.2}(F)(b)\footnote{Of course, the two methods are essentially the same.}.
Let $\veps > 0$.
Choose a compact $K \subset \Omega$ such that $\int_{\Omega \setminus K} |v|_2\,d\ssfL^m < \frac{\veps}{4}$.
Next, choose $r > 0$ sufficiently small for $\{ x : \rmdist(x,K) \leq r \} \subset \Omega$ and $\|w - v \cdot \ind_K \|_{L_1(\Omega)} < \frac{\veps}{4}$, where $w = \phi_r * (v \cdot \ind_K)$.
Note that $w \in C_c^\infty(\Omega;\Rm)$ and $\|w-v\|_{L_1(\Omega)} < \frac{\veps}{2}$.
Accordingly,
\begin{equation*}
| \la u , \bdiv v \ra| \leq \left| \int_\Omega (\nabla u) \ip w \, d\ssfL^m \right| + \left| \int_\Omega (\nabla u) \ip (w-v) \, d\ssfL^m \right| = \rmI + \rmII
\end{equation*}
for all $u \in \rmLip_0(\Omega)$.
We treat the terms $\rmI$ and $\rmII$ separately.
\par 
{\bf (ii)}
We estimate $\rmI$ from above. 
Choose $\eta > 0$ small enough for $\eta \cdot \|\rmdiv w\|_{L_1(\Omega)} < \frac{\veps}{2}$.
Next, referring to the compactness of $\rmspt w$, we choose a partition $\calB$ of $\rmspt (\rmdiv w)$ consisting of Borel-measurable sets whose diameter does not exceed $\eta$.
For each $B \in \calB$ we pick arbitrarily $x_B \in B$ and we abbreviate $E = \{ x_B : B \in \calB\} \in \rmFin(\Omega)$ and $\theta = 1 + \int_\Omega |\rmdiv w |\,d\ssfL^m$. 
We then observe that 
\begin{multline*}
\rmI  = \left| \int_\Omega u \cdot \rmdiv w \, d\ssfL^m \right| 
= \left| \sum_{B \in \calB} \int_B  u \cdot \rmdiv w \, d\ssfL^m \right| \\
\leq \sum_{B \in \calB} \int_B |u(x) - u(x_B)| \cdot |(\rmdiv w)(x)| \, d\ssfL^m(x) + \sum_{B \in \calB} |u(x_B)| \cdot \int_B |(\rmdiv w)(x)| \, d\ssfL^m(x) \\
\leq \|u\|_L \cdot \eta \cdot \|\rmdiv w\|_{L_1(\Omega)} + p_E(u) \cdot \|\rmdiv w \|_{L_1(\Omega)} \leq \frac{\veps}{2} \cdot \|u\|_L + \theta \cdot p_E(u).
\end{multline*}
\par 
{\bf (iii)}
Next, we estimate $\rmII$ from above.
Recalling \ref{4.2.5}(A), we have
\begin{equation*}
\rmII \leq \|\nabla u\|_{L_\infty} \cdot \|w-v\|_{L_1(\Omega)} \leq \frac{\veps}{2} \cdot \|u\|_L.
\end{equation*}
\par 
{\bf (iv)}
Finally, we infer from {\bf (i)}, {\bf (ii)}, and {\bf (iii)} that
\begin{equation*}
| \la u , \bdiv v \ra| \leq \rmI + \rmII \leq \theta \cdot p_E(u) + \veps \cdot \|u\|_L.
\end{equation*}
The proof that $\bdiv v$ is $\calP_{\rmLip}$-continuous is complete. 
\par 
{\bf (v)}
The inequality $\vvvert \bdiv v\vvvert \leq \|v \|_{L_1(\Omega)}$ is an immediate consequence of the definition of $\vvvert \cdot \vvvert$ and of (A) above.\cqfd
\begin{enumerate}
\item[(C)] {\it $\bdiv^* = - \bnabla$.}
\end{enumerate}
\par 
{\it Proof.}
Here, we consider $\bdiv : L_1(\Omega;\Rm) \to \rmLip_0(\Omega)[\calP_{\rmLip}]^*$, thus $\bdiv^* : \rmLip_0(\Omega) \to L_\infty(\Omega;\Rm)$, by \ref{4.1.2}(H).
Hence, given $u \in \rmLip_0(\Omega)$, $\bdiv^* u = w \in L_\infty(\Omega;\Rm)$ is characterized by $\la v , w \ra = \la v , \bdiv u \ra = - \int_\Omega (\nabla u) \ip v \, d\ssfL^m$ for all $v \in L_1(\Omega;\Rm)$.
Therefore, $\bdiv^* u = - \nabla u$.\cqfd
\end{Empty}

\begin{Theorem}
\label{4.2.7}
Let $\Omega \subset \Rm$ be an admissible open set containing 0 and consider the bounded linear operator $\bdiv : L_1(\Omega;\Rm) \to \calF(\Omega)$.
The following are equivalent.
\begin{enumerate}
\item[(A)] $\bdiv$ is surjective, \ie $\calF(\Omega) \simeq \frac{L_1(\Omega;\Rm)}{\ker \bdiv}$.
\item[(B)] $\Omega$ is uniformly connected by pencils.
\end{enumerate}
Moreover, assuming that (A) and (B) are satisfied the following hold:
\begin{enumerate}
\item[(C)] If we define $\lambda^*$ and $\Lambda_*$ by the formul\ae
\begin{equation*}
\lambda^* = \sup \{ \lambda > 0 : \lambda \cdot \|u\|_L \leq \|\nabla u \|_{L_\infty} \text{ for all } u \in \rmLip_0(\Omega) \}
\end{equation*}
and 
\begin{equation*}
\Lambda_* = \inf \{ \Lambda \geq 1 : \Omega \text{ is connected by $\Lambda$-pencils}\}
\end{equation*}
then $\lambda^* = \Lambda_*^{-1}$.
\item[(D)] For every $\veps > 0$ there exists $\bI : \calF(\Omega) \to L_1(\Omega;\Rm)$ (non-linear) satisfying the following properties.
\begin{enumerate}
\item[(a)] $(\bdiv \circ\, \bI)(F) = F$ for all $F \in \calF(\Omega)$.
\item[(b)] $\bI$ is positively homogeneous of degree 1, \ie $\bI(t \cdot F) = t \cdot \bI(F)$ for all $F \in \calF(\Omega)$ and $t > 0$.
\item[(c)] $\bI$ is continuous.
\item[(d)] For all $F \in \calF(\Omega)$ we have $\|\bI(F)\|_{L_1(\Omega)} \leq (1 + \veps) \cdot \Lambda_* \cdot \|F\|_{\calF}$.
\end{enumerate}
\end{enumerate}
\end{Theorem}

\begin{Remark}
\label{4.2.8}
Several comments are in order.
\begin{enumerate}
\item[(1)] By theorem \ref{4.1.3} we may (and from now on, we will) replace the occurrences of $\calF(\Omega)$ by $\rmLip_0(\Omega)[\calP_{\rmLip}]^*$.
\item[(2)] Hypothesis (A) is equivalent to the existence of $\lambda > 0$ such that $\lambda \cdot \|u\|_L \leq \|\nabla u\|_{L_\infty(\Omega)}$ for all $u \in \rmLip_0(\Omega)$, according to \ref{2.2}(A) and \ref{4.2.6}(C).
\item[(3)] The proof we shall present is an application of the abstract existence theorem \cite[8.1]{DEP.26c}. 
Most of the setup for this theorem (\ie its hypotheses) are independent of whether $\Omega$ is uniformly connected by pencils or not. 
In fact, only hypothesis (H) (there) is related to the geometry of $\Omega$. 
Here, we check that all other hypotheses are satisfied for any open set $\Omega \subset \Rm$ containing $0$.
Below, the letters (A),...,(I) refer to the hypotheses of \cite[8.1]{DEP.26c}.
\begin{enumerate}
\item[(A)] $\bE$ (there) is the Banach space $L_1(\Omega;\Rm)$ and all seminorms $\bq_k$ (there) a the single norm $\|\cdot\|_{L_1(\Omega)}$.
\item[(B)] $X$ (there) is $\rmLip_0(\Omega)$ and its topology $\calT$ (there) is the topology $\calS$ of pointwise convergence. It is generated by the filtering family of seminorms $\la p_E \ra_{E \in \rmFin(\Omega)}$ (recall \ref{4.1.2}) corresponding to the notation $\la \|\cdot\|_i \ra_{i \in I}$ (there). Furthermore, all $X_k$ (there) coincide with $X$.
\item[(C)] The norm $\lno \cdot \rno$ (there) is $\|\cdot\|_L$.
\item[(D)] The sets $C_k$ (there) are those defined in \ref{4.1.2} and they are, indeed, $\calS \hel C_k$ compact, \ref{4.1.2}(C).
\item[(E)] $\bD$ (there) is $\bdiv$.
\item[(F)] This is the $\calP_{\rmLip}$-continuity of $\bdiv v$, see \ref{4.2.6}(B) and \ref{4.1.2}(F).
\item[(G)] This holds with $\bc_{(G)}(k)=1$ for all $k$, by \ref{4.2.6}(A).
\item[(H)] This is the only hypothesis that depends on the geometry of $\Omega$, see the proof below.
\item[(I)] This is an empty hypothesis since $X=X_k$ for all $k$.
\end{enumerate}
\item[(4)] Contrary to theorem \ref{3.1.4}(C) we do not claim in the conclusion of theorem \ref{4.2.7} that $\ker \bdiv = L_1 \cap Z_m(\Omega)$. This is related to the failure of \ref{3.1.1}(C) when $\Omega \neq \Rm$.
\end{enumerate}
\end{Remark}

\begin{proof}[Proof of $(B) \Rightarrow (A)$, $(D)$, and $\lambda^* \geq \Lambda_*^{-1}$.]
{\bf (i)}
Let $\Lambda > 0$ and assume that $\Omega$ is connected by $\Lambda$-pencils.
It follows from \ref{4.2.5}(B) that $\|u\|_L \leq \Lambda \cdot \|\nabla u\|_{L_\infty(\Omega)}$ for all $u \in \rmLip_0(\Omega)$.
Since $\|\nabla u\|_{L_\infty(\Omega)} = \sup \{ |\la u , \bdiv u \ra| : v \in L_1(\Omega;\Rm) \text{ and } \|v\|_{L_1} \leq 1 \}$, this corresponds exactly to hypothesis (H) of \cite[8.1]{DEP.26c} with $\bc_{(H)}(k)=\Lambda$ for all $k$.
The surjectivity of $\bdiv$ is conclusion (L) of \cite[8.1]{DEP.26c}.
This completes the proof that $(B) \Rightarrow (A)$.
\par 
{\bf (ii)}
Conclusion (D) is an application of Michael's selection theorem and conclusion (M) of \cite[8.1]{DEP.26c} as in the proof of \cite[8.3(N)]{DEP.26c}.
\par 
{\bf (iii)}
Let $\eta > 0$ and $\Lambda > 0$ such that $\Lambda < \eta + \Lambda_*$ and $\Omega$ is connected by $\Lambda$-pencils.
Then $\Lambda^{-1} \cdot \|u\|_L \leq \|\nabla u \|_{L_\infty(\Omega)}$, by \ref{4.2.5}(B).
Thus, $(\eta + \Lambda_*)^{-1} < \Lambda^{-1} \leq \lambda^*$.
We conclude that $\Lambda_*^{-1} \leq \lambda^*$, by the arbitrariness of $\eta$.
\end{proof}

\begin{proof}[Proof of $(A) \Rightarrow (B)$ and $\lambda^* \leq \Lambda_*^{-1}$]
{\bf (i)}
Let $a, b \in \Omega$ be distinct.
We ought to prove that there exists a $\Lambda$-pencil of curves in $\Omega$ from $a$ to $b$, for some $\Lambda > 0$ independent of $a$ and $b$.
As we assume that $\bdiv$ is surjective, there exists $\lambda > 0$ such that $\lambda \cdot \|u\|_L \leq \|\nabla u\|_{L_\infty(\Omega)}$ for all $u \in \rmLip_0(\Omega)$, recall remark \ref{4.2.8}(2).
Let $\veps > 0$.
We choose $\lambda$ such that $\lambda^* - \veps < \lambda$ and $\lambda \cdot \|u\|_L \leq \|\nabla u\|_{L_\infty(\Omega)}$ for all $u \in \rmLip_0(\Omega)$.
In particular, hypothesis (H) of \cite[8.1]{DEP.26c} is satisfied with $\bc_{(H)}(k) = \lambda^{-1}$ for all $k$ (same observation as in {\bf (i)} of the proof that $(B) \Rightarrow (A)$).
By remark \ref{4.2.8}(3), we conclude that \cite[8.1]{DEP.26c} applies.
\par 
{\bf (ii)}
Given distinct $a,b \in \Omega$, define $F = \lseg b \rseg - \lseg a \rseg \in \calF(\Omega)$.
Recall that $\|F\|_\calF = |b-a|_2$ (\eg proof of \ref{3.1.1}(A)).
By conclusion (M) of \cite[8.1]{DEP.26c} there exists $v \in L_1(\Omega;\Rm)$ such that $F = \bdiv v$ and $\|v\|_{L_1(\Omega)} \leq (1+\veps) \cdot \lambda^{-1} \cdot \|F\|_{\calF}$
\par 
{\bf (iii)}
Here, we will associate with $v$ a $1$-dimensional metric current $T_v$ in the complete metric space $E = \rmclos \Omega$, in the sense of Ambrosio-Kirchheim \cite{AMB.KIR.00}.
We refer to \cite[B1]{PAO.STE.12} for the definition and notation.
With a pair $(f,\pi) \in \rmLip_b(E) \times \rmLip(E)$ we associate $T_v(f,\pi) = - \int_\Omega f \cdot (\nabla \pi) \ip v \, d\ssfL^m$.
We now check that $T_v$ satisfies the four axioms of $1$-dimensional metric currents. 
\par 
It is most obvious that $T_v : \rmLip_b(E) \times \rmLip(E) \to \R$ is {\it bilinear}.
\par 
The {\it continuity} axiom asks that $\lim_j T_v(f,\pi_j) = T_v(f,\pi)$ for every $f \in \rmLip_b(E)$ and every sequence $\la \pi_j \ra_j$ that converges pointwise to $\pi$ and has $\sup_j \rmLip \pi_j < \infty$.
To see that this is, indeed, the case we notice that $f \cdot v \in L_1(\Omega;\Rm)$, since $f$ is bounded, and we associate $u_\pi \in \rmLip_0(\Omega)$ with $\pi \in \rmLip(E)$ by letting $u_\pi(x) = \pi(x)-\pi(0)$, $x \in \Omega$, so that $T_v(f,\pi) = \la u_\pi , \bdiv(f \cdot v) \ra$ for all $(f,\pi) \in \rmLip_b(E) \times \rmLip(E)$, since $\nabla \pi = \nabla u_\pi$.
Observe that the way of convergence of $\la \pi_j \ra_j$ to $\pi$ in the next to last sentence implies that $\lim_j u_{\pi_j} = u_\pi$ in the sense of $\calP_{\rmLip}$ (recall \ref{4.1.2}(A)).
Since $\bdiv f \cdot v$ is $\calP_{\rmLip}$-continuous, \ref{4.2.6}(B), we conclude that $\lim_j T_v(f,\pi_j) = \lim_j \la u_{\pi_j} , \bdiv f\cdot v \ra = \la u_\pi , \bdiv f \cdot v \ra = T_v(f,\pi)$.
\par 
The axiom of {\it locality} is trivially satisfied.
It asks that $T_v(f,\pi)=0$ if $f$ is a constant in a neighborhood of the $\rmclos \{ f \neq 0 \}$.
Under this assumption, it easily follows from Rademacher's theorem that $f \cdot \nabla \pi$ vanishes $\ssfL^m$-a.e. in $\Omega$, thus, $T_v(f,\pi) = 0$.
\par Finally, the axiom of {\it finite mass} is also trivially satisfied.
Indeed, we have $|T_v(f,\pi)| \leq (\rmLip \pi) \int_\Omega |f|d(\ssfL^m \hel |v|_2)$ for all $(f,\pi)$.
Therefore, by definition of the measure $\|T_v\|$ (denoted $\mu_{T_v}$ in \cite{PAO.STE.12}) we have $\|T_v\| \leq \ssfL^m \hel |v|_2$.
\par 
{\bf (iv)}
Here, we notice that $\partial T_v = \lseg b \rseg - \lseg a \rseg$, where the 0-dimensional metric currents $\lseg a \rseg , \lseg b \rseg : \rmLip_b(E) \to \R$ are defined to be the evaluation at, respectively, $a$ and $b$.
Indeed, $(\partial T_v)(f) = T_v(1,f) = \la u_f , \bdiv v \ra = \la u_f , F \ra = \la u_f , \lseg b \rseg - \lseg a \rseg \ra = u_f(b) - u_f(a) = f(b) - f(a)$ for all $f \in \rmLip_b(E)$.
It follows from \cite[proposition 3.8]{PAO.STE.12} that $T_v$ admits an acyclic $1$-dimensional metric subcurrent $S$ in $E$.
In particular, $\|S\|(E) \leq \|T_v\|(E) \leq (\ssfL^m \hel |v|_2)(E) = \|v\|_{L_1}$ (recall \cite[definition 3.1]{PAO.STE.12}) and $\partial S = \partial T_v = \lseg b \rseg - \lseg a \rseg$.
Thus, $S$ is a $1$-dimensional normal metric current.
Moreover, $\|\partial S \| = \delta_b + \delta_a$, since $a \neq b$\footnote{If $a=b$ then $\|\partial S\|=0$; this is the place where we use that $a \neq b$.}.
\par 
{\bf (v)}
We come to the main step of this proof.
To this end, we let $\bGamma(E)$ consist of those members of $\bGamma$ whose range is contained in $E$ and given $\gamma \in \bgamma \in \bGamma(E)$ we define the $1$-dimensional normal metric current $\lseg \bgamma \rseg$ in $E$ by the formula $\lseg \bgamma \rseg(f,\pi) = \int_0^1 f(\gamma(t)) \cdot (\pi \circ \gamma)'(t) \, d\ssfL^1(t)$.
We will use the fact that if $\gamma$ is an arc then $\bM(\lseg \bgamma \rseg) = \| \lseg \bgamma \rseg \|(E) = \ssfH^1(\rmim \gamma)$.
\par
We are in a position to apply \cite[theorem 5.1]{PAO.STE.12} to $S$ and we conclude that there exists a finite Borel measure\footnote{In \cite{PAO.STE.12} this is called a {\it transport}.} $\eta$ on $\bGamma(E)$ such that $\eta$-a.e. $\bgamma$ is represented by an arc $\gamma$,
\begin{equation}
\label{eq.20.1}
S(f,\pi) = \int_{\bGamma(E)} \lseg \bgamma \rseg (f,\pi) \, d\eta(\bgamma)
\end{equation}
for all $(f,\pi) \in \rmLip_b(E) \times \rmLip(E)$ (this is expressing that $S$ equals $T_\eta$ in the notation of \cite{PAO.STE.12} and (4.1) there),
\begin{equation}
\label{eq.20.2}
\|S\|(A) = \int_{\bGamma(E)} \| \lseg \bgamma \rseg \|(A) \, d\eta(\bgamma)
\end{equation}
for all Borel-measurable set $A \subset E$ (this is (4.4) and lemma 4.17 in \cite{PAO.STE.12}), and 
\begin{equation}
\label{eq.20.3}
\|\partial S \| = \int_{\bGamma(E)} (\delta_{\gamma(0)} + \delta_{\gamma(1)}) \, d\eta(\bgamma).
\end{equation}
(this is (4.4) and the last line of p.3366 in \cite{PAO.STE.12}).
\par 
Upon noticing that $\bGamma(E)$ is closed in $\bGamma$, we note that $\eta$ may be extended to $\bGamma$ to a finite Borel measure $\mu$ such that $\rmspt \mu \subset \bGamma(E)$ and $\mu(A)=\eta(A)$ for all Borel-measurable $A \subset \bGamma(E)$.
We ought to show that $\mu$ is a pencil of curves in $\Omega$ from $a$ to $b$.
The proof of this is in {\bf (vi)} and {\bf (vii)} below.
\par 
{\bf (vi)}
Here, we will prove that $\mu$ is a probability measure and that it satisfies condition (1) of \ref{4.2.4}.
First, we observe that \eqref{eq.20.3} implies that
\begin{equation*}
\int_{\bGamma} \big( \delta_{\gamma(0)}(A) + \delta_{\gamma(1)}(A) \big)\, d\mu(\bgamma) = \|\partial S \|(A) = \delta_a(A) + \delta_b(A)
\end{equation*}
for every Borel-measurable set $A \subset E$.
In particular, letting $A = E$ we infer that $\mu(\bGamma)=\mu(\bGamma(E))=1$.
\par 
Next, for $i \in \{0,1,2\}$ we define $\bGamma_i = \bGamma \cap \big\{ \bgamma : \rmcard [ \{ \gamma(0),\gamma(1) \} \cap \{a,b\} ] = i \big\}$.
Notice that each $\bGamma_i$ is well-defined and Borel-measurable.
Moreover, if $A = \{a,b\}$ then $\delta_{\gamma(0)}(A) + \delta_{\gamma(1)}(A) = i$ for all $\bgamma \in \bGamma_i$.
Since $\{\bGamma_0,\bGamma_1,\bGamma_2\}$ is a Borel partition of $\bGamma$, we infer that
\begin{multline*}
2 \cdot \mu(\bGamma_0) + 2 \cdot \mu(\bGamma_1) + 2 \cdot \mu(\bGamma_2) = 2 =  \delta_a(A) + \delta_b(A) = \int_{\bGamma} \big( \delta_{\gamma(0)}(A) + \delta_{\gamma(1)}(A) \big)\, d\mu(\bgamma) \\
= \sum_{i=0}^2 \int_{\bGamma_i} \big( \delta_{\gamma(0)}(A) + \delta_{\gamma(1)}(A) \big)\, d\mu(\bgamma) = 0 \cdot \mu(\bGamma_0) + 1 \cdot \mu(\bGamma_1) + 2 \cdot \mu(\bGamma_2).
\end{multline*}
Thus, $2 \cdot \mu(\bGamma_0) + 1 \cdot \mu(\bGamma_1) = 0$ and we conclude that $\mu$-a.e. $\bgamma$ belongs to $\bGamma_2$.
Observe that $\bGamma_2 = \bGamma_{a,b} \cup \bGamma_{b,a}$ and that $\bGamma_{a,b} \cap \bGamma_{b,a} = \emptyset$.
Recalling that $(\partial \lseg \bgamma \rseg)(f) = f(\gamma(1)) - f(\gamma(0))$ for every $f \in \rmLip_b(E)$ and $\gamma \in \bgamma \in \bGamma(E)$, we infer from \eqref{eq.20.1} that
\begin{multline*}
f(b) - f(a) = (\partial T_v)(f) = (\partial S)(f) = S(1,f) = \int_{\bGamma} \lseg \bgamma \rseg(1,f) \,d\mu(\bgamma)  \\ = \int_{\bGamma_{a,b}} \lseg \bgamma \rseg(1,f) \,d\mu(\bgamma) + \int_{\bGamma_{b,a}} \lseg \bgamma \rseg(1,f) \,d\mu(\bgamma) \\ = (f(b)-f(a)) \cdot \mu(\bGamma_{a,b}) - (f(b)-f(a)) \cdot \mu(\bGamma_{b,a}).
\end{multline*}
Choosing $f$ such that $f(b) \neq f(a)$, we have $1 = \mu(\bGamma_{a,b}) - \mu(\bGamma_{b,a})$.
Hence, $\mu(\bGamma_{a,b})=1$.
Since $\mu$-a.e. $\bgamma$ is represented by an arc, by definition of {\it transport} in \cite{PAO.STE.12}, condition (1) of \ref{4.2.4} is satisfied.
\par 
{\bf (vii)}
We turn to showing that condition (2) of \ref{4.2.4} is satisfied as well.
Let $\nu$ be defined by $\nu(A) = \int_{\Rm} \ssfH^1(A \cap \rmim \gamma)\,d\mu(\bgamma) = \int_{\bGamma(E)} \ssfH^1(A \cap \rmim \gamma)\, d\eta(\bgamma)$, $A \in \calB(\Rm)$.
Since $\rmim \gamma \subset E$ for all $\bgamma \in \bGamma(E)$, it follows that $\nu(\Rm \setminus E)=0$.
In particular, $\rmspt \nu \subset \rmclos \Omega$, which is \ref{4.2.4}(2)(ii).
We now turn to proving \ref{4.2.4}(2)(i), \ie $\nu \ll \ssfL^m$.
It suffices to show that $\nu(A)=0$ if $\ssfL^m(A)=0$ and $A$ is a Borel-measurable subset of $E$.
In that case, we infer from \eqref{eq.20.1} that
\begin{multline*}
\nu(A) = \int_{\bGamma(E)} \ssfH^1(A \cap \rmim \gamma)\, d\eta(\bgamma) = \int_{\bGamma(E)} \|\lseg \bgamma \rseg \|(A) \,d\eta(\bgamma) = \|S\|(A) \leq \|T_v\|(A) \\ \leq (\ssfL^m \hel |v|_2)(A) = 0,
\end{multline*}
where the inequality $\|S\| \leq \|T_v\|$ is a consequence of the fact that $S$ is a subcurrent of $T_v$ and (3.1) in \cite{PAO.STE.12}.
Finally, applying the above inequalities to $A=E$ we have
\begin{equation*}
\nu(\Rm) = \nu(E) \leq (\ssfL^m \hel |v|_2)(E) = \|v\|_{L_1} \leq (1+\veps) \cdot \lambda^{-1} \cdot \|F\|_{\calF} = (1+\veps) \cdot \lambda^{-1} \cdot |b-a|_2,
\end{equation*}
by {\bf (ii)}.
Accordingly, condition \ref{4.2.4}(2)(iii) is satisfied.
This completes the proof that $\mu$ is a $(1+\veps) \cdot \lambda^{-1}$-pencil of curves in $\Omega$ from $a$ to $b$.
Since $a$ and $b$ are arbitrary, conclusion (A) ensues.
\par 
{\bf (viii)}
We have $\Lambda_* \leq (1+\veps) \cdot \lambda^{-1} \leq (1+ \veps) \cdot (\lambda^* - \veps)^{-1}$.
Since $\veps$ is arbitrary, it follows that $\Lambda_* \leq (\lambda^*)^{-1}$.
This completes the proof of conclusion (C).
\end{proof}

\section{The spaces \texorpdfstring{$\SCH_{0,\calM,p}(\Omega)$}{SCH-0-M-p} and \texorpdfstring{$\SCH_{0,\calW,p}(\Omega)$}{SCH-0-W-p}}
\label{sec.SCH}

\subsection{The localized topology \texorpdfstring{$\calM_{p,TV}$}{M-p-TV}}
\label{subsec.BVEXT}
In this subsection, unless stated otherwise, $\Omega$ is an arbitrary non-empty open subset of $\Rm$

\begin{Empty}[The localized topology $\calM_{p,TV}$]
\label{5.1.3}
As usual, $1 \leq p \leq 1^*$.
We let $BV_{p,\rmcst}(\Omega)$ denote the quotient of $BV_p(\Omega)$ by its subspace $BV_p(\Omega) \cap \{ u : u \text{ is constant a.e.}\}$.
Notice that this subspace is non-trivial if and only if $\ssfL^m(\Omega) < \infty$ in which case it is isomorphic to $\R$.
Thus, if $\ssfL^m(\Omega) = \infty$ then $BV_{p,\rmcst}(\Omega) = BV_p(\Omega)$.
\par 
Members of $BV_{p,\rmcst}(\Omega)$ will be denoted as $[u]$, where $u \in BV_p(\Omega)$.
Note that $BV_{p,\rmcst}(\Omega) \subset L_{p,\rmcst}(\Omega)$ and recall the notation $(u)_\Omega$ and $\|[u]\|_{p}$ from \ref{2.5}.
Note also that for $[u] \in BV_{p,\rmcst}(\Omega)$ the real number $\|u - (u)_\Omega \cdot \ind_\Omega\|_{L_p(\Omega)}$ is independent upon the choice of a representative of $[u]$ and this defines on $BV_{p,\rmcst}(\Omega)$ a norm equivalent to $\|[u]\|_p$, according to \ref{2.5}(A).
%
%
The corresponding locally convex topology on $BV_{\rmcst}(p,\Omega)$ is denoted $\calM_p$.
\par 
Moreover, note that $\lno [u] \rno = \|Du\|(\Omega)$ well-defines a seminorm on $BV_{p,\rmcst}(\Omega)$ and that this seminorm is a norm in case $\Omega$ is connected or satisfies the $(p,1)$-Poincar\'e inequality, according to \ref{2.6}(D).
Based on \ref{2.5}(A) and the lower semicontinuity of $u \mapsto \|Du\|(\Omega)$ with respect to distributional convergence, it is not hard to check that $\lno \cdot \rno$ is lower semicontinuous with respect to $\calM_p$.
Thus, abbreviating $C_{p,k} = BV_{p,\rmcst}(\Omega) \cap \{ [u] : \lno [u] \rno \leq k \}$, $k \in \N$, we infer that $\calC_p = \{ C_{p,k} : k \in \N \}$ is a localizing family on $BV_{p,\rmcst}(\Omega)$ consisting of an increasing sequence of $\calM_p$-closed sets.
We denote by $\calM_{p,TV}$ the localized locally convex topology on $BV_{p,\rmcst}(\Omega)$ associated with $\calM_p$ and $\calC_p$.
We gather below the properties of $\calM_{p,TV}$ that ensue from \cite{DEP.26c}.
\begin{enumerate}
\item[(A)] {\it Let $\la [u_j]\ra_j$ be a sequence in, and $[u]$ a member of, $BV_{p,\rmcst}(\Omega)$. The following are equivalent.
\begin{enumerate}
\item[(a)] $\lim_j [u_j] = [u]$ with respect to $\calM_{p,TV}$.
\item[(b)] $\lim_j \| u_j - (u_j)_\Omega \|_{L_p(\Omega)} = 0$ and $\sup_j \|Du_j\|(\Omega) < \infty$.
\end{enumerate}
}
\end{enumerate}
\par 
{\it Proof.}
This is a direct consequence of \cite[3.1(A)]{DEP.26c} and \ref{2.5}(A).\cqfd
\begin{enumerate}
\item[(B)] {\it A set $B \subset BV_{p,\rmcst}(\Omega)$ is $\calM_{p,TV}$-bounded if and only if it satisfies both conditions (a) and (b) below:
\begin{enumerate}
\item[(a)] $\sup \{ \|u - (u)_\Omega\|_{L_p(\Omega)} : [u] \in B \} < \infty$;
\item[(b)] $\sup \{ \|Du\|(\Omega) : [u] \in B \} < \infty$.
\end{enumerate}
Moreover, if $\Omega$ satisfies the $(p,1)$-Poincar\'e inequality then $B$ is $\calM_{p,TV}$-bounded if and only if it satisfies condition (b) above and in that case $\la C_{p,k} \ra_k$ is a fundamental system of $\calM_{p,TV}$-bounded sets in $BV_{p,\rmcst}(\Omega)$.}
\end{enumerate}
\par 
{\it Proof.}
In the general case, this is a direct consequence of \cite[3.1(B)]{DEP.26c}.
If $\Omega$ satisfies the $(p,1)$-Poincar\'e inequality then condition (a) holds whenever (b) does and in that case the last conclusion trivially follows.\cqfd
\begin{enumerate}
\item[(C)] {\it The following hold for every $k \in \N$.
\begin{enumerate}
\item[(a)] $C_{p,k}[\calM_p \hel C_{p,k}]$ is metrizable. 
\item[(b)] If $\Omega$ is bounded, connected, and a $BV$-extension set then $C_{1,k}[\calM_1 \hel C_{1,k}]$ is compact.
\end{enumerate}
}
\end{enumerate}
\par 
{\it Proof.}
Conclusion (a) is trivial and conclusion (b) is a consequence of \ref{2.6}(A) and \ref{bv.ext}(A)(a) in the following way.
Let $\la [u_j] \ra_j$ be a sequence in $C_{1,k}$.
Since $\Omega$ satisfies the $(1,1)$-Poincar\'e inequality, by \ref{bv.ext}(A)(b), there are $y_j \in \R$ such that, upon defining $\tilde{u}_j = u_j - y_j \cdot \ind_\Omega$, we have $\| \tilde{u}_j \|_{L_1(\Omega)} \leq \bc_{1}(\Omega) \cdot \|D\tilde{u}_j \|(\Omega)$, whence, $\|\tilde{u}_j\|_{BV(\Omega)} \leq k \cdot \left( 1 + \bc_1(\Omega) \right)$.
By compactness \ref{bv.ext}(A)(a), a subsequence still denoted $\la \tilde{u}_j \ra_j$ converges with respect to $\|\cdot\|_{L_1(\Omega)}$ to $\tilde{u} \in BV(\Omega)$. 
Accordingly, $\| [ \tilde{u}_j - \tilde{u}]\|_1 \leq \| \tilde{u}_j - \tilde{u}\|_{L_1(\Omega)} \to 0$ as $j \to \infty$.\cqfd
\begin{enumerate}
\item[(D)] {\it If $\Omega$ is bounded, connected, and a $BV$-extension set then $BV_{1,\rmcst}(\Omega)[\calM_{1,TV}]$ is sequential.}
\end{enumerate}
\par 
{\it Proof.}
In view of (C) above, this is a consequence of \cite[4.4(A)]{DEP.26c}.\cqfd
\begin{enumerate}
\item[(E)] {\it If $1 \leq p < 1^*$ then $BV_{p,\rmcst}(\Omega)[\calM_{p,TV}]$ is none of Fr\'echet-Urysohn, barrelled, and bornological.}
\end{enumerate}
\par 
{\it Proof.}
Note that \cite[3.8]{DEP.26c} applies, according to (C)(a) above.
Observe that $\lno \cdot \rno$ is not sequentially $\calM_p$-continuous (consider a decreasing null-sequence $\la r_j \ra_j$ of positive real numbers, $a \in \Omega$ such that $B(a,r_0) \subset \Omega$, and $u_j = r_j^{-(m-1)} \cdot \ind_{B(a,r_j)}$, and observe that $\|Du_j\|(\Omega)=1$ for all $j$ and $\|[u_j]\|_p \leq \|u_j\|_{L_p(\Omega)} = \balpha(m)^\frac{1}{p} \cdot r_j^{-(m-1) + \frac{m}{p}} \to 0$ as $j \to \infty$, since the exponent of $r_j$ is positive when $p < 1^*$).
It now ensues from \cite[3.8(A)]{DEP.26c} that $BV_{p,\rmcst}(\Omega)[\calM_{p,TV}]$ is neither Fr\'echet-Urysohn nor bornological and from \cite[3.8(B)]{DEP.26c} that it is not barrelled.\cqfd
\begin{enumerate}
\item[(F)] {\it Let $F : BV_{p,\rmcst}(\Omega) \to \R$ be linear. The following are equivalent.
\begin{enumerate}
\item[(a)] $F$ is $\calM_{p,TV}$-continuous.
\item[(b)] $(\forall \veps > 0)(\exists \theta > 0)(\forall [u] \in BV_{p,\rmcst}(\Omega)): | \la [u] , F \ra | \leq \theta \cdot \|u-(u)_\Omega\|_{L_p(\Omega)} + \veps \cdot \|Du\|(\Omega)$.
\end{enumerate}
}
\end{enumerate}
\par 
{\it Proof.}
It suffices to recall that $\|u - (u)_\Omega\|_{L_p(\Omega)}$ is equivalent to $\|[u]\|_p$ and to observe that \cite[3.6]{DEP.26c} applies.\cqfd
\begin{enumerate}
\item[(G)] {\it $BV_{p,\rmcst}(\Omega)[\calM_{p,TV}]^*$ equipped with its strong topology is a Banach space normed by
\begin{enumerate}
\item[(a)] $\vvvert F \vvvert = \sup \big\{ |\la [u] , F \ra| : \|u-(u)_\Omega\|_{L_p(\Omega)} + \|Du\|(\Omega) \leq 1 \big\}$
\item[(b)] If we moreover assume that $\Omega$ satisfies the $(p,1)$-Poincar\'e inequality then we may take $\vvvert F \vvvert = \sup \big\{ | \la [u] , F \ra | : \|Du\|(\Omega) \leq 1 \big\}$.
\end{enumerate}
}
\end{enumerate}
{\it Proof.}
{\bf (i)} 
Upon letting $D_{p,k} = BV_{p,\rmcst}(\Omega) \cap \{ [u] : \|u-(u)_\Omega\|_{L_p(\Omega)} + \|Du\|(\Omega) \leq k \}$, we observe that $\la D_{p,k} \ra_k$ is a fundamental system of $\calM_{p,TV}$-bounded sets, according to (B) above.
Since $\la D_{p,k} \ra_k$ is increasing and $D_{p,k} \subset k \cdot D_{p,1}$ for all $k \geq 1$, we infer that the strong topology of $BV_{p,\rmcst}(\Omega)[\calM_{p,TV}]^*$ is normable by $\vvvert F \vvvert = \sup \{ | \la [u] , F \ra | : [u] \in D_{p,1} \}$.
\par 
{\bf (ii)}
If $\la F_j \ra_j$ is $\vvvert \cdot \vvvert$-Cauchy in the strong dual of $BV_{p,\rmcst}(\Omega)[\calM_{p,TV}]$ then it readily converges pointwise to some linear functional $F : BV_{p,\rmcst}(\Omega) \to \R$.
One classically shows that $\la F_j \ra_j$ converges to $F$ uniformly on $D_{p,1}$.
We now establish the $\calM_{p,TV}$-continuity of $F$ by means of (F) above.
Let $\veps > 0$, choose $j$ such that $\vvvert F - F_j \vvvert < \frac{\veps}{2}$, and let $\theta_j$ be associated with $F_j$ and $\frac{\veps}{2}$ in (F)(b).
If $[u] \in BV_{p,\rmcst}(\Omega)$ then $| \la [u] , F \ra | \leq |\la [u] , F - F_j \ra| + | \la [u] , F_j \ra| \leq \frac{\veps}{2} \cdot \left( \|u-(u)_\Omega\|_{L_p(\Omega)} + \|Du\|(\Omega) \right) + \theta_j \cdot \|u-(u)_\Omega\|_{L_p(\Omega)} + \frac{\veps}{2} \cdot \|Du\|(\Omega)$.
This completes the proof that the strong dual of $BV_{p,\rmcst}(\Omega)[\calM_{p,TV}]$ is a Banach space.
\par 
{\bf (iii)}
Conclusion (b) follows from (B) above.\cqfd
\begin{enumerate}
\item[(H)] {\it If $\Omega$ is a bounded, connected $BV$-extension set then $BV_{1,\rmcst}(\Omega)[\calM_{1,TV}]$ is semireflexive.}
\end{enumerate}
\par 
{\it Proof.}
According to (B) and (C) above, this follows from \cite[7.4]{DEP.26c}.\cqfd
\end{Empty}

\begin{Empty}[$\SCH_{0,\calM,p}(\Omega)$]
\label{def.sch.strong}
We abbreviate $\SCH_{0,\calM,p}(\Omega) = BV_p(\Omega)[\calM_{p,TV}]^*$.
This is a Banach space whose norm $\vvvert \cdot \vvvert$ is described in \ref{5.1.3}(G) above.
\end{Empty}

\begin{Empty}[A divergence operator]
\label{5.1.4}
Let $v \in C_0(\Omega;\Rm)$.
For each $[u] \in BV_{p,\rmcst}(\Omega)$ and each $i=1,\ldots,m$ the function $v_i$, being Borel-measurable and bounded, is summable with respect to the signed measure $D_iu$.
Furthermore, the measures $D_iu$ do not depend upon the choice of a representative of $[u]$.
Accordingly, the following is a well-defined linear form:
\begin{equation*}
\bdiv v : BV_{p,\rmcst}(\Omega) \to \R: [u] \mapsto - \sum_{i=1}^m \int_\Omega v_i \, d(D_iu).
\end{equation*}
\begin{enumerate}
\item[(A)] {\it $| \la [u] , \bdiv v \ra| \leq \lno [u] \rno \cdot \|v\|_\infty$ for all $[u]$.}
\end{enumerate}
\par 
{\it Proof.}
Trivial.\cqfd
\begin{enumerate}
\item[(B)] {\it $\bdiv v$ is $\calM_{p,TV}$-continuous and $\vvvert \bdiv v \vvvert \leq \|v\|_\infty$.}
\end{enumerate}
\par 
{\it Proof.}
Let $\veps > 0$ and choose a compact $K \subset \Omega$ such that $|v(x)|_2 < \veps$ for all $x \in \Omega \setminus K$.
Select $r > 0$ such that $\rmspt w \subset \Omega$ and $\|w - v \cdot \ind_K \|_\infty < \veps$, where $w = \Phi_r *(v \cdot \ind_K)$.
Let $[u] \in BV_{p,\rmcst}(\Omega)$, abbreviate $\tilde{u} = u - (u)_\Omega$, and observe that
\begin{multline*}
\la [u] , \bdiv v \ra = - \sum_{i=1}^m \int_\Omega w_i \,d(D_i\tilde{u}) - \sum_{i=1}^m \int_\Omega (v_i-w_i) \, d(D_iu) \\ = \int_\Omega \tilde{u} \cdot \rmdiv w \, d\ssfL^m - \sum_{i=1}^m  \int_\Omega (v_i-w_i) \, d(D_iu)
\leq \|\rmdiv w \|_{L_q(\Omega)} \cdot \|u - (u)_\Omega\|_{L_p(\Omega)} + \veps \cdot \|Du\|(\Omega).
\end{multline*}
The $\calM_{p,TV}$-continuity of $\bdiv v$ follows from \ref{5.1.3}(F).
The second conclusion is a consequence of (A) and \ref{5.1.3}(G).\cqfd
\end{Empty}

\begin{Empty}[$\bLambda_q$]
\label{5.1.5}
Let $f \in L_{q,\#}(\Omega)$ and recall \ref{2.5} that $\int_\Omega u \cdot f \,d\ssfL^m$ is defined and independent of the choice of a representative of $[u] \in BV_{p,\rmcst}(\Omega)$.
This defines a linear form $\bLambda_q(f) : BV_{p,\rmcst}(\Omega) \to \R : [u] \mapsto \int_\Omega u \cdot f \, d\ssfL^m$.
\begin{enumerate}
\item[(A)] {\it $\bLambda_q(f)$ is $\calM_{p,TV}$-continuous.}
\end{enumerate}
\par 
{\it Proof.}
This immediately follows from \ref{5.1.3}(F), since 
\begin{equation*}
|\la [u] , \bLambda_q(f) \ra | = \left| \int_\Omega (u-(u)_\Omega\cdot \ind_\Omega) \cdot f \, d\ssfL^m \right| \leq \|u - (u)_\Omega \cdot \ind_\Omega \|_{L_p(\Omega)} \cdot \|f\|_{L_q(\Omega)}.\cqfd
\end{equation*}
\begin{enumerate}
\item[(B)] {\it $\vvvert \bLambda_q(f) \vvvert \leq \|f\|_{L_q(\Omega)}$.}
\end{enumerate}
\par 
{\it Proof.}
This is a direct consequence of \ref{5.1.3}(G)(a) and the inequality in the proof of (A) above.\cqfd
\end{Empty}

\begin{Theorem}
\label{5.1.6}
If the bounded linear operator
\begin{equation*}
\bdiv : C_0(\Omega;\Rm) \to BV_{p,\rmcst}(\Omega)[\calM_{p,TV}]^*
\end{equation*}
is surjective then $\Omega$ satisfies the $(p,1)$-Poincar\'e inequality.
\end{Theorem}

\begin{proof}
{\bf (i)}
Recall that the range $Y = BV_{p,\rmcst}(\Omega)[\calM_{p,TV}]^*[\vvvert \cdot \vvvert]$ is Banach \ref{5.1.3}(G).
Therefore, by \ref{2.2}(A), there exists $\bc > 0$ such that $\bc \cdot \|y^*\|_Y^* \leq \|\bdiv^*(y^*)\|_{M^m}$ for all $y^* \in Y^*$.
Notice that there is no restriction to assume that $0 \leq \bc < 1$.
\par 
{\bf (ii)}
We consider the evaluation map $\rmev : BV_{p,\rmcst}(\Omega) \to Y^*$ and we shall apply {\bf (i)} to $y^* = \rmev([u])$ for all $[u] \in BV_{p,\rmcst}(\Omega)$.
To this end, we recall \ref{5.1.5} and we note that 
\begin{equation*}
\begin{split}
\| \rmev([u]) \|_Y^* & = \sup \left\{ | \la [u] , F \ra| : F \in BV_{p,\rmcst}(\Omega)[\calM_{p,TV}]^* \text{ and }\vvvert F \vvvert \leq 1 \right\} \\
& \geq \sup \left\{ | \la [u] , \bLambda_q(f) \ra| : f \in L_{q,\#}(\Omega) \text{ and } \|f\|_{L_q(\Omega)} \leq 1 \right\} \\
& = \|[u]\|_p,
\end{split}
\end{equation*}
where the last inequality follows from \ref{2.5}(B).
\par 
Moreover, we claim that $\bdiv^*(\rmev([u])) = - (D_1u,\ldots,D_mu)$.
Indeed, if $\bdiv^*(\rmev([u])) = (\mu_1,\ldots,\mu_m)$ then for every $v \in C_0(\Omega;\Rm)$ we have
\begin{equation*}
\sum_{i=1}^m \int_\Omega v_i \, d\mu_i = \la v , \bdiv^*(\rmev([u])) \ra = \la \bdiv v , \rmev([u]) \ra = \la [u] , \bdiv v \ra = - \sum_{i=1}^m \int_\Omega v_i \, d(D_iu)
\end{equation*}
and the claim readily ensues.
\par 
{\bf (iii)} 
For each $u \in BV_p(\Omega)$ we infer from {\bf (i)} and {\bf (ii)} that
\begin{equation*}
\bc \cdot  \|[u]\|_p \leq \bc \cdot \| \rmev([u]) \|_Y^* \leq \|\bdiv^*(\rmev([u]))\|_{M^m} = \|(D_1u,\ldots,D_mu)\|_{M^m} = \|Du\|(\Omega).
\end{equation*}
Since $u$ is arbitrary, $\Omega$ satisfies the $(p,1)$-Poincar\'e inequality with $\bc_p(\Omega) = \bc^{-1}$.
\end{proof}

\begin{Theorem}
\label{5.1.7}
Assume that $\Omega$ is a bounded, connected $BV$-extension set.
Then the bounded linear operator 
\begin{equation*}
\bdiv : C_0(\Omega;\Rm) \to \SCH_{0,\calM,1}(\Omega)
\end{equation*}
is surjective.
Furthermore, for every $\veps > 0$ there exists a (non-linear) map
\begin{equation*}
\bI : \SCH_{0,\calM,1}(\Omega) \to C_0(\Omega;\Rm)
\end{equation*}
such that
\begin{enumerate}
\item[(a)] $(\bdiv \circ\, \bI)(F) = F$ for all $F \in \SCH_{0,\calM,1}(\Omega)$.
\item[(b)] $\bI$ is positively homogeneous of degree 1.
\item[(c)] $\bI$ is continuous.
\item[(d)] For all $F \in \SCH_{0,\calM,1}(\Omega)$ we have $\|\bI(F)\|_\infty \leq (1+\veps) \cdot \vvvert F \vvvert$.
\end{enumerate}
\end{Theorem}

\begin{proof}
This is an application of \cite[8.1]{DEP.26c}.
We refer to the notation there in order to check that all hypotheses are satisfied.
\begin{enumerate}
\item[(A)] Hypothesis (A) of \cite[8.1]{DEP.26c} is with the Banach space $\bE = C_0(\Omega;\Rm)$ equipped with its norm $\bq_1 = \|\cdot\|_\infty$.
\item[(B)] Hypothesis (B) is with $X[\calT] = BV_{1,\rmcst}(\Omega)[\calM_1]$ whose topology $\calM_1$ is generated by the norm $\|[\cdot]\|_1$.
Moreover $X_k = X$ for all $k$.
\item[(C)] Hypothesis (C) is with $\lno \cdot \rno$ defined in \ref{5.1.3}, \ie $\lno [u] \rno = \|Du\|(\Omega)$.
\item[(D)] Hypothesis (D) is satisfied, by \ref{5.1.3}(C)(b), thanks to our assumptions regarding $\Omega$.
\item[(E)] Hypothesis (E) is with $\bD = \bdiv$.
\item[(F)] Hypothesis (F) is satisfied, by \ref{5.1.4}(B).
\item[(G)] Hypothesis (G) is satisfied with $\bc_{(G)}(k)=1$ for all $k$, by \ref{5.1.4}(A).
\item[(H)] Hypothesis (H) is satisfied with $\bc_{(H)}(k)=1$ for all $k$, by definition of total variation:
\begin{equation*}
\begin{split}
\lno [u] \rno & = \|Du\|(\Omega) \\
& = \sup \left\{ \left| \int_\Omega u \cdot \rmdiv v \, d\ssfL^m \right| : v \in C_c^1(\Omega;\Rm) \text{ and } \|v\|_\infty \leq 1 \right\} \\
& \leq \sup \left\{ | \la [u] , \bdiv v \ra | : v \in C_0(\Omega,\Rm) \text{ and } \|v\|_\infty \leq 1 \right\}.
\end{split}
\end{equation*}
\item[(I)] Hypothesis (I) is void, since $X_k = X$ for all $k$.
\end{enumerate}
\par 
The surjectivity of $\bdiv$ is \cite[8.1(L)]{DEP.26c} and the remaining conclusions are an application of Michael's selection theorem and of \cite[8.1(M)]{DEP.26c} as in the proof of \cite[8.3(N)]{DEP.26c}.
\end{proof}

\subsection{The localized topology \texorpdfstring{$\calW_{p,TV}$}{W-p-TV}} 
\label{subsec.PCI}

Throughout this subsection we make the following two assumptions:
\begin{enumerate}
\item[(1)] $1 < p \leq 1^*$.
\item[(2)] $\Omega$ satisfies the $(p,1)$-Poincar\'e inequality.
\end{enumerate}
Notice carefully that $p \neq 1$.

\begin{Empty}[The localized topology $\calW_{p,TV}$]
\label{5.2.1}
Recall the definition of $BV_{p,\rmcst}(\Omega)$ in \ref{5.1.3}.
We let $\calW_p$ be the restriction to $BV_{p,\rmcst}(\Omega)$ of the weak topology of $L_{p,\rmcst}(\Omega)$, \ie $\sigma(L_{p,\rmcst},L_{q,\#})$.
It follows from \ref{2.5}(B) that $\calW_p$ is generated by the family of seminorms $\la p_f \ra_{f \in \wh{\Omega}_q}$, where $\wh{\Omega}_q = B_{L_{q,\#}(\Omega)} = L_{q,\#}(\Omega) \cap \{ f : \|f\|_{L_q(\Omega)} = 1 \}$ and $p_f : BV_{p,\rmcst}(\Omega) \to \R : u \mapsto \left| \int_\Omega u \cdot f \,d\ssfL^m \right|$.
Notice that $\la p_f \ra_{f \in \hat{\Omega}_q}$ is not filtering and that this can be overcome as in \ref{4.1.2}: For $E \subset \rmFin(\hat{\Omega}_q)$ letting $p_E(u) = \max \{ p_f(u) : f \in E \}$, the family of seminorms $\la p_E \ra_{E \in \rmFin(\wh{\Omega}_q)}$ is filtering and generates the topology $\calW_p$.
As before, members of $BV_{p,\rmcst}(\Omega)$ are denoted $[u]$, corresponding to $u \in BV_p(\Omega)$.
\par 
Exactly as in \ref{5.1.3} we let $\lno \cdot \rno$ be the seminorm defined on $BV_{p,\rmcst}(\Omega)$ by the formula $\lno [u] \rno = \|Du\|(\Omega)$.
It is a norm, by \ref{2.6}(D).
It is also lower semicontinuous with respect to $\calW_p$. 
Indeed, this is equivalent to the set $BV_{p,\rmcst}(\Omega) \cap \{ [u] : \lno [u] \rno > 1 \}$ being $\calW_p$-open.
If $[u]$ belongs to this set then there are $\veps > 0$ such that $\|Du\|(\Omega) = \lno [u] \rno > 1 + \veps$ and $v \in C^1_c(\Omega;\Rm)$ such that $\|v\|_\infty \leq 1$ and $\int_\Omega u \cdot \rmdiv v \, d\ssfL^m > 1 + \veps$.
Observing that $f = \rmdiv v - (\rmdiv v)_\Omega \in L_{q,\#}(\Omega) \setminus \{0\}$ and letting $\tilde{f} = \|f\|_{L_q(\Omega)}^{-1} \cdot f \in \wh{\Omega}_q$ we infer that $\|f\|_{L_q(\Omega)} \cdot p_{\tilde{f}}([u]) = \left| \int_\Omega u \cdot \rmdiv v \, d\ssfL^m \right| > 1 + \veps$.
Accordingly, if $p_{\tilde{f}}([u] - [u']) < \frac{\veps}{2} \cdot \|f\|_{L_q(\Omega)}^{-1}$ then $\lno [u'] \rno = \|Du'\|(\Omega) \geq \left| \int_\Omega u' \cdot \rmdiv v \,d\ssfL^m \right| = \|f\|_{L_q(\Omega)} \cdot p_{\tilde{f}}([u']) > - \frac{\veps}{2} + \|f\|_{L_q(\Omega)} \cdot p_{\tilde{f}}([u]) > 1 + \frac{\veps}{2}$.
\par
As in \ref{5.1.3}, $C_{p,k} = BV_{p,\rmcst}(\Omega) \cap \{ [u] : \lno [u] \rno \leq k \}$, $k \in \N$.
Thus, $\calC_p = \{ C_{p,k} : k \in \N\}$ is a localizing family on $BV_{p,\rmcst}(\Omega)$ consisting of an increasing sequence of $\calW_p$-closed sets.
We denote by $\calW_{p,TV}$ the localized locally convex topology on $BV_{p,\rmcst}(\Omega)$ associated with $\calW_p$ and $\calC_p$.
We gather below the properties of $\calW_{p,TV}$ that ensue from \cite{DEP.26c}.
\begin{enumerate}
\item[(A)] {\it Let $\la [u_j]\ra_j$ be a sequence in, and $[u]$ a member of, $BV_{p,\rmcst}(\Omega)$. The following are equivalent.
\begin{enumerate}
\item[(a)] $\lim_j [u_j] = [u]$ with respect to $\calW_{p,TV}$.
\item[(b)] $\lim_j \left| \int_\Omega (u_j - (u_j)_\Omega) \cdot f \, d\ssfL^m \right|=0$ for every $f \in \wh{\Omega}_q$ and $\sup_j \|Du_j\|(\Omega) < \infty$.
\end{enumerate}
}
\end{enumerate}
\par 
{\it Proof.}
This is a direct consequence of \cite[3.1(A)]{DEP.26c}.\cqfd
\par 
Notice that the first condition in (b) can equivalently rephrased for every $f \in L_{q,\#}(\Omega)$ and that the term $(u_j)_\Omega$ in the integral can be replaced by $y_j \cdot \ind_\Omega$ for any $y_j \in \R$.
\begin{enumerate}
\item[(B)] {\it A set $B \subset BV_{p,\rmcst}(\Omega)$ is $\calW_{p,TV}$-bounded if and only if $\sup \{ \|Du\|(\Omega) : [u] \in B \} < \infty$.}
\end{enumerate}
\par 
{\it Proof.}
In view of \cite[3.1(B)]{DEP.26c}, it suffices to show that if $\sup \{ \|Du\|(\Omega) : [u] \in B \} < \infty$ then $B$ is $\calW_p$-bounded.
Call $\Gamma$ the supremum above.
Given $[u] \in B$, there exists $y \in \R$ such that $\|u - y \cdot \ind_\Omega \|_{L_p(\Omega)} \leq \bc(\Omega) \cdot \Gamma$, since $\Omega$ satisfies the $(p,1)$-Poincar\'e inequality.
If $f \in \wh{\Omega}_q$ then, by H\"older's inequality, $p_f([u]) = \left| \int_\Omega (u - y \cdot \ind_\Omega) \cdot f \, d\ssfL^m \right| \leq \|u - y \cdot \ind_\Omega \|_{L_p(\Omega)} \cdot \|f\|_{L_q(\Omega)} \leq \bc_p(\Omega) \cdot \Gamma$.
Accordingly, $\sup \{ p_f([u]) : [u] \in B \} < \infty$ for all $f \in \wh{\Omega}$ and the proof is complete.\cqfd
\begin{enumerate}
\item[(C)] {\it The topological spaces $C_{p,k}[\calW_p \hel C_{p,k}]$ are metrizable and compact, $k \in \N$.}
\end{enumerate}
\par 
{\it Proof.}
Recall that $\calW_p$ is the restriction to $BV_{p,\rmcst}(\Omega)$ of the weak topology $\sigma(L_{p,\rmcst},L_{q,\#})$.
Since $L_{p,\rmcst}(\Omega)$ is reflexive, by \ref{2.5}(C), this is also the weak* topology of $L_{p,\rmcst}(\Omega)$ inasmuch as $L_{p,\rmcst}(\Omega) \cong L_{q,\#}(\Omega)^*$, recall \ref{2.5}(B).
Arguing as in the proof of (B) above, one easily checks that $C_{p,k}$ is (strongly) bounded in $L_{p,\rmcst}(\Omega)$, in fact, $\sup \{ \|[u]\|_p : [u] \in C_k \} \leq k \cdot \bc_p(\Omega)$.
Therefore, $C_{p,k}[\sigma(L_{p,\rmcst},L_{q,\#}) \hel C_{p,k}]$ is relatively compact, by the Banach-Alaoglu theorem \cite[4.3(c)]{RUDIN}, and metrizable \cite[3.16]{RUDIN}, since $L_{q,\#}(\Omega)$ is separable.
It remains to observe that $C_{p,k}$ is $\sigma(L_{p,\rmcst},L_{q,\#})$-closed in $L_{p,\rmcst}(\Omega)$, according to the lower semicontinuity of $\lno \cdot \rno$ with respect to $\sigma(L_{p,\rmcst},L_{q,\#})$ (recall the paragraph before (A) above).\cqfd
\begin{enumerate}
\item[(D)] {\it $BV_{p,\rmcst}(\Omega)[\calW_{p,TV}]$ is sequential.}
\end{enumerate}
\par 
{\it Proof.}
In view of (C) above, this is a consequence of \cite[4.4(A)]{DEP.26c}.\cqfd
\begin{enumerate}
\item[(E)] {\it If $1 \leq p < 1^*$ then $BV_{p,\rmcst}(\Omega)[\calW_{p,TV}]$ is none of Fr\'echet-Urysohn, barrelled, and bornological.}
\end{enumerate}
\par 
{\it Proof.}
Note that \cite[3.8]{DEP.26c} applies, according to (C) above.
Observe that $\lno \cdot \rno$ is not sequentially $\calW_p$-continuous, by the same example given in \ref{5.1.3}(E).
It now ensues from \cite[3.8(A)]{DEP.26c} that $BV_{p,\rmcst}(\Omega)[\calW_{p,TV}]$ is neither Fr\'echet-Urysohn nor bornological and from \cite[3.8(B)]{DEP.26c} that it is not barrelled.\cqfd
\begin{enumerate}
\item[(F)] {\it Let $F : BV_{p,\rmcst}(\Omega) \to \R$ be linear. The following are equivalent.
\begin{enumerate}
\item[(a)] $F$ is $\calW_{p,TV}$-continuous.
\item[(b)] $(\forall \veps > 0)(\exists E \in \rmFin(\wh{\Omega}_q))(\exists \theta > 0)(\forall [u] \in BV_{p,\rmcst}(\Omega)):$
\begin{equation*}
 | \la [u] , F \ra | \leq \theta \cdot \max \left\{ \left| \int_\Omega u \cdot f \, d\ssfL^m \right| : f \in E \right\} + \veps \cdot \|Du\|(\Omega).
\end{equation*}
\end{enumerate}
}
\end{enumerate}
\par 
{\it Proof.}
It suffices to observe that \cite[3.6]{DEP.26c} applies and that $\la p_E \ra_{E \in \rmFin(\wh{\Omega}_q)}$ is filtering.\cqfd
\begin{enumerate}
\item[(G)] {\it $BV_{p,\rmcst}(\Omega)[\calW_{p,TV}]^*$ equipped with its strong topology is a Banach  space normed by $\vvvert F \vvvert = \sup \big\{ | \la [u] , F \ra | : \|Du\|(\Omega) \leq 1 \big\}$.}
\end{enumerate}
{\it Proof.}
The strong topology on $BV_{p,\rmcst}(\Omega)[\calW_{p,TV}]^*$ being that of uniform convergence of $\calW_{p,TV}$-bounded sets it follows from (B) above and the inclusions $C_k \subset k \cdot C_1$ that the strong topology is normed as stated. 
Completeness follows as in \cite[5.4]{DEP.26c}.\cqfd
\begin{enumerate}
\item[(H)] {\it $BV_{p,\rmcst}(\Omega)[\calW_{p,TV}]$ is semireflexive.}
\end{enumerate}
\par 
{\it Proof.}
According to (B) and (C) above, this is a consequence of \cite[7.4]{DEP.26c}.\cqfd
\begin{enumerate}
\item[(I)] {\it If $\Omega$ is connected, bounded, and a $BV$-extension set then $BV_{1,\rmcst}(\Omega) = BV_{p,\rmcst}(\Omega)$ and the topologies $\calM_{1,TV}$ and $\calW_{p,TV}$ coincide.}
\end{enumerate}
\par 
{\it Proof.}
{\bf (i)} 
It follows from \ref{bv.ext}(A)(c) that $BV(\Omega) = BV_p(\Omega)$, hence, $BV_{1,\rmcst}(\Omega) = BV_{p,\rmcst}(\Omega)$ which we shall simply denote as $BV_\rmcst(\Omega)$ for the remainder of this proof.
We shall also write $C_k$ instead of $C_{1,k}=C_{p,k}$.
Note also that the topologies $\calW_p$ and $\calW_{p,TV}$ are defined.
It remains to prove that $\calW_{p,TV} = \calM_{1,TV}$.
\par 
{\bf (ii)}
It suffices to establish the the identity map $\rmid : BV_{\rmcst}(\Omega)[\calM_{1,TV}] \to BV_{\rmcst}(\Omega)[\calW_{p,TV}]$ is a homeomorphism.
To this end, it is enough to show that $\rmid$ and $\rmid^{-1}$ are sequentially continuous (recall \cite[A.1(E)]{DEP.26c}, since both $\calM_{1,TV}$ and $\calW_{p,TV}$ are sequential (by \ref{5.1.3}(D) and (D) above).
In other words we ought to show that $(\calM_{1,TV}) \lim_j [u_j] = 0$ if and only if $(\calW_{p,TV}) \lim_j [u_j] = 0$ for every sequence $\la [u_j] \ra_j$ in $BV_{\rmcst}(\Omega)$.
\par 
{\bf (iii)}
Assume that $(\calM_{1,TV}) \lim_j [u_j] = 0$.
By \ref{5.1.3}(A), $\la [u_j] \ra_j$ is contained in $C_k$ for some $k \in \N$.
Since $C_k[\calW_p \hel C_k]$ is metrizable (recall (C) above), it suffices to show that each subsequence of $\la [u_j] \ra_j$ contains a subsequence converging to 0 with respect to $\calW_{p}$.
We do not change the notation when passing to a subsequence.
There are $y_j \in \R$ such that $\| u_j - y_j \cdot \ind_\Omega \|_{L_{p}(\Omega)} \leq k \cdot \bc_p(\Omega)$.
By Banach-Alaoglu's theorem applied in $L_{p}(\Omega) \cong L_q(\Omega)^*$, there exists $u \in L_{p}(\Omega)$ and a subsequence such that $u_j - y_j \cdot \ind_\Omega -u \to 0$ as $j \to \infty$ with respect to $\sigma(L_{p},L_q)$.
As $\ssfL^m(\Omega) < \infty$, we have $L_\infty(\Omega) \subset L_q(\Omega)$, thus, $u_j - y_j \cdot \ind_\Omega -u \to 0$ as $j \to \infty$ with respect to $\sigma(L_1,L_\infty)$ as well.
Moreover, $\lim_j \|[u_j]\|_1 = 0$, by assumption, \ie there exist $y'_j \in \R$ such that $\lim_j \|u_j - y'_j \cdot \ind_\Omega\|_{L_1(\Omega)}=0$.
Therefore, $\lim_j (u_j - y'_j \cdot \ind_\Omega) = 0$ with respect to $\sigma(L_1,L_\infty)$ as well and, in turn, $\lim_j (y_j-y'_j) \cdot \ind_\Omega = u$ with respect to $\sigma(L_1,L_\infty)$.
This readily implies that the sequence $\la y_j-y'_j \ra_j$ is convergent in $\R$, whence, $[u] = 0$.
\par 
{\bf (iv)}
Finally, it remains to show that $\rmid : C_k[\calM_{1,TV} \hel C_k] \to C_k[\calW_{p,TV} \hel C_k]$ is a homeomorphism.
By {\bf (iii)} and metrizability, this map is continuous. 
Since both domain and range are compact (by \ref{5.1.3}(C)(b) and (C) above), the conclusion follows.\cqfd
\end{Empty}

\begin{Empty}[$\SCH_{0,\calW,p}(\Omega)$]
\label{def.sch.weak}
We abbreviate $\SCH_{0,\calW,p}(\Omega) = BV_p(\Omega)[\calW_{p,TV}]^*$.
This is a Banach space whose norm $\vvvert \cdot \vvvert$ is described in \ref{5.2.1}(G) above.
\end{Empty}

\begin{Empty}[A divergence operator]
\label{5.2.2}
The operator $\bdiv : BV_{p,\rmcst}(\Omega) \to \R$ is defined as in \ref{5.1.4} and we notice that \ref{5.1.4}(A) is unchanged.
\begin{enumerate}
\item[(B)] {\it $\bdiv v$ is $\calW_{p,TV}$-continuous and $\vvvert \bdiv v \vvvert \leq \|v\|_\infty$.}
\end{enumerate}
\par 
{\it Proof.}
We modify slightly the proof of \ref{5.1.4}(B) of which we retain all notation.
Without loss of generality we assume that $\rmdiv w \neq 0$.
We associate with the vector field $w$ the function $f = \rmdiv w - (\rmdiv w)_\Omega$.
Notice that, as $w$ has compact support, $f \in L_{q,\#}(\Omega) \setminus \{0\}$ and define $\tilde{f} = \|f\|_{L_q(\Omega)}^{-1} \cdot f \in \wh{\Omega}_q$.
Furthermore, $\int_\Omega \tilde{u} \cdot \rmdiv w \, d\ssfL^m = \int_\Omega \tilde{u} \cdot f \, d\ssfL^m$, since $\int_\Omega \tilde{u} \,d\ssfL^m = 0$.
Accordingly,
\begin{multline*}
\la [u] , \bdiv v \ra = - \sum_{i=1}^m \int_\Omega w_i \,d(D_i\tilde{u}) - \sum_{i=1}^m \int_\Omega (v_i-w_i) \, d(D_iu) \\ = \int_\Omega \tilde{u} \cdot f \, d\ssfL^m - \sum_{i=1}^m  \int_\Omega (v_i-w_i) \, d(D_iu)
\leq \|f\|_{L_q(\Omega)} \cdot p_{\tilde{f}}([u]) + \veps \cdot \|Du\|(\Omega).
\end{multline*}
The $\calW_{p,TV}$-continuity of $\bdiv v$ follows from \ref{5.2.1}(F).\cqfd
\end{Empty}

\begin{Theorem}
\label{5.2.3}
Assume that $\Omega$ satisfies the $(p,1)$-Poincar\'e inequality for some $1 < p \leq 1^*$.
Then the bounded linear operator 
\begin{equation*}
\bdiv : C_0(\Omega;\Rm) \to \SCH_{0,\calW,p}(\Omega)
\end{equation*}
is surjective.
Furthermore, for every $\veps > 0$ there exists a (non-linear) map
\begin{equation*}
\bI : \SCH_{0,\calW,p}(\Omega) \to C_0(\Omega;\Rm)
\end{equation*}
such that
\begin{enumerate}
\item[(a)] $(\bdiv \circ\, \bI)(F) = F$ for all $F \in \SCH_{0,\calW,p}(\Omega)$.
\item[(b)] $\bI$ is positively homogeneous of degree 1.
\item[(c)] $\bI$ is continuous.
\item[(d)] For all $F \in \SCH_{0,\calW,p}(\Omega)$ we have $\|\bI(F)\|_\infty \leq (1+\veps) \cdot \vvvert F \vvvert$.
\end{enumerate}
\end{Theorem}

\begin{proof}
This is yet another application of \cite[8.1]{DEP.26c}.
We refer to the notation there in order to check that all hypotheses are satisfied.
\begin{enumerate}
\item[(A)] Hypothesis (A) of \cite[8.1]{DEP.26c} is with the Banach space $\bE = C_0(\Omega;\Rm)$ equipped with its norm $\bq_1 = \|\cdot\|_\infty$.
\item[(B)] Hypothesis (B) is with $X[\calT] = BV_{p,\rmcst}(\Omega)[\calW_p]$ whose topology $\calW_p$ is generated by filtering family of seminorms $\la p_E \ra_{E \in \rmFin(\wh{\Omega}_q)}$.
Moreover $X_k = X$ for all $k$.
\item[(C)] Hypothesis (C) is with $\lno \cdot \rno$ defined in \ref{5.2.1}, \ie $\lno [u] \rno = \|Du\|(\Omega)$.
\item[(D)] Hypothesis (D) is satisfied, by \ref{5.2.1}(C).
\item[(E)] Hypothesis (E) is with $\bD = \bdiv$.
\item[(F)] Hypothesis (F) is satisfied, by \ref{5.2.1}(B).
\item[(G)] Hypothesis (G) is satisfied with $\bc_{(G)}(k)=1$ for all $k$, by \ref{5.1.4}(A).
\item[(H)] Hypothesis (H) is satisfied with $\bc_{(H)}(k)=1$ for all $k$, by definition of total variation:
\begin{equation*}
\begin{split}
\lno [u] \rno & = \|Du\|(\Omega) \\
& = \sup \left\{ \left| \int_\Omega u \cdot \rmdiv v \, d\ssfL^m \right| : v \in C_c^1(\Omega;\Rm) \text{ and } \|v\|_\infty \leq 1 \right\} \\
& \leq \sup \left\{ | \la [u] , \bdiv v \ra | : v \in C_0(\Omega,\Rm) \text{ and } \|v\|_\infty \leq 1 \right\}.
\end{split}
\end{equation*}
\item[(I)] Hypothesis (I) is void, since $X_k = X$ for all $k$.
\end{enumerate}
\par 
The surjectivity of $\bdiv$ is \cite[8.1(L)]{DEP.26c} and the remaining conclusions are an application of Michael's selection theorem and of \cite[8.1(M)]{DEP.26c} as in the proof of \cite[8.3(N)]{DEP.26c}.
\end{proof}

The following is the analog of theorem \ref{4.1.3}.

\begin{Theorem}
\label{5.2.4}
$\SCH_0(\Rm) = \SCH_{0,\calW,1^*}(\Rm)$.
\end{Theorem}

\begin{proof}
{\bf (i)}
Recall that members of $\SCH_0(\Rm)$ are linear forms $F : \bvs(\Rm) \to \R$ that are $\|\cdot\|_{\bvs}$-continuous and lie in the $\|\cdot\|_{\bvs}^*$-closure of the range of $\bLambda_m$, see \ref{3.2.2}.
Here, we show that such $F$ is $\calW_{1^*,TV}$-continuous, by means of \ref{5.2.1}(F).
Let $\veps > 0$.
There exists $g \in L_m(\Rm)$ such that $\|F-\bLambda_m(g)\|_{\bvs}^* < \veps$. 
Thus, for all $u \in \bvs(\Rm)$ we have
\begin{multline*}
| \la u , F \ra| \leq |\la u , \bLambda_m(g) \ra| + |\la u , \bLambda_m(g) -F \ra| \leq 
\left| \int_{\Rm} u \cdot g \,d\ssfL^m \right| + \|\bLambda_m(g)-F\|_{\bvs}^* \cdot \|u\|_{\bvs}
\\
\leq \theta \cdot \left| \int_{\Rm} u \cdot f \,d\ssfL^m \right| + \veps \cdot \|Du\|(\Omega),
\end{multline*}
where\footnote{If $g = 0$ then choose $f \in \wh{\Omega}_m$ and $\theta > 0$ arbitrarily.} $\theta = \|g\|_{L_m}$ and $f = \theta^{-1} \cdot g \in \wh{\Omega}_m$.
\par 
{\bf (ii)}
According to {\bf (i)}, we may consider the identity $\SCH_0(\Rm) \to \SCH_{0,\calW,1^*}(\Rm)$.
Moreover, the identity is an isometry; a direct consequence of the definition of the norms $\|\cdot\|_{\bvs}^*$ in $\SCH_0(\Rm)$ and $\vvvert\cdot\vvvert$ on $\SCH_{0,\calW,1^*}(\Rm)$ (recall \ref{5.2.1}(G)).
In particular, the range $W$ of the identity in $\SCH_{0,\calW,1^*}(\Rm)$ is closed and in order to complete the proof it remains only to show that $W$ is dense.
To this end, we apply Hahn-Banach's theorem: we ought to show that if $\alpha \in \SCH_{0,\calW,1^*}(\Rm)[\vvvert\cdot\vvvert]^*$ vanishes on $W$ then it vanishes identically.
By \ref{5.2.1}(H), there exists $u \in \bvs(\Rm)$ such that $\alpha = \rmev_u$.
For each $f \in L_m(\Rm)$, $\bLambda_m(f) \in SCH_0(\Rm) = W$, therefore, $\int_{\Rm} u \cdot f \, d\ssfL^m = \la \bLambda_m(f) , \rmev_u \ra = 0$.
Since $f$ is arbitrary, $u= 0$ and, in turn, $\alpha = 0$.
\end{proof}

\section{Examples}
\label{sec.EXAMPLES}

\subsection{Examples of members of \texorpdfstring{$\calF(\Omega)$}{F(Omega)}}
\label{subsec.EX.F.RM}

Here, we transfer the results of subsection \ref{subsec.LSVF} to the language of distributions and 1-dimensional currents.

\begin{Empty}[$\calG(\Rm)$]
\label{EX.1.1}
In this subsection, we sometimes use the notation $\calD^0(\Rm)$ for the space of test functions (also denoted $\calD(\Rm)$) and we let $\calD_0(\Rm)$ be the space of distributions (seen as 0-dimensional currents).
Similarly, $\calD^1(\Rm)$ is the space of $C^\infty$ differential forms of degree 1, $\Rm \to \bigwedge^1\Rm$ with compact support and $\calD_1(\Rm)$ is its topological dual, \ie the space of 1-dimensional currents.
Recall that the boundary operator $\partial : \calD_1(\Rm) \to \calD_0(\Rm)$ is defined to be the adjoint of the exterior derivative $d : \calD^0(\Rm) \to \calD^1(\Rm)$.
The {\em mass} of a current $T \in \calD_i(\Rm)$, $i=0,1$, is defined by
\begin{equation*}
\bM(T) = \sup \{ \la \omega , T \ra : \omega \in \calD^i(\Rm) \text{ and } \|\omega\|_\infty \leq 1 \} \in [0,\infty].
\end{equation*}
Given $a \in \Rm$, we let $\lseg a \rseg \in \calD_0(\Rm)$ be defined by $\la \vphi , \lseg a \rseg \ra = \vphi(a)$.
Given $a,b \in \Rm$, we let $\lseg a,b \rseg \in \calD_1(\Rm)$ be defined by $\la \omega , \lseg a ,b \rseg \ra = |b-a|_2^{-1} \cdot  \int_0^1 \la b-a , \omega(a+ t\cdot (b-a)) \ra \, dt$ if $a \neq b$ and $\lseg a,b \rseg = 0$ otherwise.
It is easy to check that $\bM(\lseg a \rseg) = 1$ and $\bM(\lseg a,b \rseg) = |b-a|_2$.
\par 
With a distribution $T \in \calD_0(\Rm)$ we associate 
\begin{equation*}
\bG(T) = \inf \{ \bM(S) : S \in \calD_1(\Rm) \text{ and } T = \partial S \} \in [0,\infty].
\end{equation*}
In particular, $\bG(T) = \infty$ if $T$ is not the boundary of any 1-dimensional current.
\begin{enumerate}
\item[(A)] {\it $\bG(T) = \sup \{ \la \vphi , T \ra : \vphi \in \calD^0(\Rm) \text{ and } \|d\vphi\|_\infty \leq 1 \}$.}
\item[(B)] {\it If $\bG(T) < \infty$ then there exists $S \in \calD_1(\Rm)$ such that $\bG(T) = \bM(S)$.}
\end{enumerate}
\par 
Both claims are based on an application of Hahn-Banach's theorem, see \cite[3.2]{DEP.26b}.
Notice that $\|d\vphi\|_\infty = \|\nabla \vphi\|_\infty = \rmLip \vphi$.
It follows from (A) that the restriction of $\bG$ to $\calD_0(\Rm) \cap \{ T : \bG(T) < \infty \}$ is a norm and it is not hard to show that it gives a Banach space.
Let $E \subset \Rm$ be a finite set and let $\theta : E \to \R$ be such that $\sum_{x \in E} \theta_x = 0$.
In that case, the current $T_{E,\theta} = \sum_{x \in E} \theta(x) \cdot \lseg x \rseg \in \calD_0(\Rm)$ is a boundary, for instance $T_{E,\theta} = \partial S$, where $S = \sum_{x \in E} \theta(x) \cdot \lseg 0, x \rseg$.
In particular, $\bG(T_{E,\theta}) \leq \sum_{x \in E} |\theta(x)| \cdot |x|_2 < \infty$.
We let $\calP_{0,\#}(\Rm)$ consist of the 0-dimensional currents $T_{E,\theta}$ corresponding to all choices of $E$ and $\theta$ as above.
Next, we define $\calG(\Rm)$ to be the $\bG$-completion of $\calP_{0,\#}(\Rm)$ in $\calD_0(\Rm)$.
\par
We will need the following lemma.
\begin{enumerate}
\item[(C)] {\it If $u \in \rmLip_0(\Rm)$, $R > 0$, and $\veps > 0$ then there exists $\vphi_u \in \calD(\Rm)$ such that $\|d\vphi_u\|_\infty \leq \veps + \|u\|_L$ and $|u(x) - [ \vphi_u(x) - \vphi_u(0)]| < \veps$ for all $x \in B(0,R)$.} 
\end{enumerate}
\par 
{\it Proof.}
Define $\tilde{u} \in \rmLip_0(\Rm)$ by $\tilde{u}(x) = \max \{ R \cdot \|u\|_L , \min \{ -R \cdot \|u\|_L , u(x) \} \}$ so that $\|\tilde{u}\|_\infty \leq R \cdot \|u\|_L$ and $\tilde{u}|_{B(0,R)} = u|_{B(0,R)}$, since $u(0)=0$.
Choose $r > 0$ such that $\|\tilde{u}_r - u \|_\infty < \frac{\veps}{2}$, where $\tilde{u}_r = \Phi_r * u$, and $\chi_k \in \calD(\Rm)$ such that $\ind_{B(0,R)} \leq \chi_k \leq \ind_{B(0,R+k)}$ and $\| \nabla \chi_k \|_\infty \leq \frac{2}{k}$, where $k \in \N$ is large enough for $\veps \cdot k > 2 \cdot R \cdot \|u\|_L$.
Put $\vphi_u = \chi_k \cdot \tilde{u}_r$.
It follows from the chain rule that $\|d\vphi_u\|_\infty = \| \nabla (\chi_k \cdot \tilde{u}_r) \|_\infty \leq \frac{2 \cdot R \cdot \|u\|_L}{k} + \|u\|_L < \veps + \|u\|_L$.
If $x \in B(0,R)$ then $\chi_k(x)=1$, hence, $\vphi_u(x) = \tilde{u}_r(x)$, thus, $|u(x) - [ \vphi_u(x) - \vphi_u(0)]| \leq |u(x) - \tilde{u}_r(x)| + |\tilde{u}_r(0) - u(0)| < \veps$.\cqfd
\begin{enumerate}
\item[(D)] {\it There exists an isometric isomorphism $\Upsilon : \calF(\Rm) \to \calG(\Rm)$ such that $\la \vphi , \Upsilon(F) \ra = \la \vphi - \vphi(0) \cdot \ind_{\Rm} , F \ra$ for all $(\vphi,F) \in \calD(\Rm) \times \calF(\Rm)$.}
\end{enumerate}
\par
{\it Proof.}
{\bf (i)}
Recalling the notation of \ref{3.1.1} we first define a linear map
\begin{equation*}
\Upsilon : \rmspan (\rmim \bdelta) \to \calP_{0,\#}(\Rm) : \sum_{x \in E} \theta(x) \cdot \lseg x \rseg \mapsto \sum_{x \in E} \theta(x) \cdot \lseg x \rseg - \left( \sum_{x \in E} \theta(x)\right) \cdot \lseg 0 \rseg.
\end{equation*}
If $\vphi \in \calD(\Rm)$ then we abbreviate $u_\vphi = \vphi - \vphi(0) \cdot \ind_{\Rm} \in \rmLip_0(\Rm)$ and we trivially check that $\la \vphi , \Upsilon(F) \ra = \la u_\vphi , F \ra$ for all $(\vphi,F) \in \calD(\Rm) \times \rmspan(\rmim \bdelta)$
Moreover, as $\|u_\vphi\|_L = \|d\vphi\|_\infty$, we have $\bG(\Upsilon(F)) \leq \|F\|_L^*$ for every $F \in \rmspan(\rmim \bdelta)$.
\par 
{\bf (ii)}
In order to prove the reverse inequality, given $F \in \rmspan(\rmim \bdelta)$ and $\veps > 0$ we select $u \in \rmLip_0(\Rm)$ such that $\|u\|_L \leq 1$ and $\|F\|_L^* < \veps + \la u , F \ra$.
Write $F = \sum_{x \in E} \theta(x) \cdot \lseg x \rseg$.
Let $R = \rmdiam(\{0\} \cup E)$.
Consider $\vphi_u$ associated with $u$, $R$, and $\veps$ in (C) above.
Observe that $|\la u,F \ra - \la \vphi_u , \Upsilon(F) \ra| = | \la u,F \ra - \la \vphi_u - \vphi_u(0) \cdot \ind_{\Rm} , F \ra | = \left| \sum_{x \in E} \theta(x) \cdot ( u(x) - [ \vphi_u(x) - \vphi_u(0]) \right| < \veps \cdot \sum_{x \in E} |\theta(x)|$.
Moreover, $| \la \vphi_u , \Upsilon(F) \ra| \leq \bG(\Upsilon(F)) \cdot \|d\vphi_u\|_\infty < \bG(\Upsilon(F)) \cdot (1+\veps)$.
Accordingly, $\|F\|_L^* < \veps \cdot \left( 1 + \sum_{x \in E} |\theta(x)| \right) + \bG(\Upsilon(F)) \cdot (1+\veps)$.
Since $\veps$ is arbitrary, we conclude that $\|F\|_L^* \leq \bG(\Upsilon(F))$.
\par 
{\bf (iii)}
We infer from {\bf (i)} and {\bf (ii)} that $\Upsilon$ defined in {\bf (i)} is an isometry.
Note also that that it is surjective, \ie $\Upsilon(\rmspan(\rmim \bdelta)) = \calP_{0,\#}(\Rm)$.
It then follows from the definition of $\calF(\Rm)$ and $\calG(\Rm)$ that $\Upsilon$ extends to a bijective linear isometry still denoted $\Upsilon : \calF(\Rm) \to \calG(\Rm)$. 
\par 
{\bf (iv)}
It remains to establish that $\la \vphi , \Upsilon(F) \ra = \la \vphi - \vphi(0) \cdot \ind_{\Rm} , F \ra$ for all $(\vphi,F) \in \calD(\Rm) \times \calF(\Rm)$.
By {\bf (i)}, this holds for all $(\vphi,F) \in \calD(\Rm) \times \rmspan(\rmim \bdelta)$.
Fix $(\vphi,F) \in \calD(\Rm) \times \calF(\Rm)$ and choose $\la F_j \ra_j$ in $\rmspan(\rmim \bdelta)$ such that $\lim_j \|F-F_j\|_L^* = 0$, whence, also $\lim_j \bG(\Upsilon(F) - \Upsilon(F_j))=0$.
Thus, $\la F_j \ra_j$ and $\la \Upsilon(F_j) \ra_j$ converge weakly* to, respectively, $F$ and $\Upsilon(F)$ and the conclusion follows.\cqfd
\par 
The following requires \cite[5.6 and 5.7]{DEP.26c}.
\begin{enumerate}
\item[(E)] {\it Let $T : \calD(\Rm) \to \R$ be a linear functional. The following are equivalent.
\begin{enumerate}
\item[(a)] $T \in \calG(\Rm)$.
\item[(b)] $(\forall \veps > 0)(\exists E \in \rmFin(\Rm))(\exists \theta > 0)(\forall \vphi \in \calD(\Rm)):$
\begin{equation*}
| \la \vphi , T \ra| \leq \theta \cdot \max \{ |\vphi(x) - \vphi(0)| : x \in E \} + \veps \cdot \|d\vphi\|_\infty.
\end{equation*}
\end{enumerate}
}
\end{enumerate}
\par 
{\it Proof.}
{\bf (i)}
Assume $T \in \calG(\Rm)$ and let $F = \Upsilon^{-1}(T)$.
It follows from \ref{4.1.3} that $F$ is $\calP_{\rmLip}$-continuous.
Given $\veps > 0$, apply \ref{4.1.2}(F) to $F$ to obtain $E \in \rmFin(\Rm)$ and $\theta > 0$ such that for every $\vphi \in \calD(\Rm)$ we have $|\la \vphi , T \ra | = | \la \vphi , \Upsilon(F) \ra | = | \la u_\vphi , F \ra | \leq \theta \cdot p_E(u_\vphi) + \veps \cdot \|u_\vphi\|_L = \theta \cdot \max \{ |\vphi(x)-\vphi(0)| : x\in E\} + \veps \cdot \|d\vphi\|_\infty$, where $u_\vphi = \vphi - \vphi(0) \cdot \ind_{\Rm}$ as in (D).
This completes the proof that $(a) \Rightarrow (b)$.
\par 
{\bf (ii)}
Assume that $T$ satisfies (b).
Consider the vector space $Y = \rmLip_0(\Rm) \cap \{ \vphi - \vphi(0) \cdot \ind_{\Rm} : \vphi \in \calD(\Rm) \}$.
We need to observe that $Y$ is uniformly sequentially $\calP_{\rmLip}$-dense in $\rmLip_0(\Rm)$ in the sense of \cite[5.6]{DEP.26c}.
Let $u \in \rmLip_0(\Rm) \setminus \{0\}$.
For each $j \in \N$, apply (C) with $R=j$ and $\veps = \min \{ j^{-1},\|u\|_L\}$ to obtain $\vphi_j \in \calD(\Rm)$ such that $\|d\vphi_j\|_\infty < \veps + \|u\|_L \leq 2 \cdot \|u\|_L$ and $\|u - (\vphi_j - \vphi_j(0) \cdot \ind_{\Rm})\|_{\infty,B(0,j)} < j^{-1}$.
Define $u_j = \vphi_j - \vphi_j(0) \cdot \ind_{\Rm} \in Y$ and observe that $u_j \to u$ with respect to $\calS$ (\ie pointwise) and $\|u_j\|_L \leq 2 \cdot \|u\|_L$ for all $j$.
This shows that $Y$ is uniformly sequentially $\calP_{\rmLip}$-dense in $\rmLip_0(\Rm)$.
\par 
Next, we define $F : Y \to \R$ by the formula $\la \vphi - \vphi(0) \cdot \ind_{\Rm} , F \ra = \la \vphi , T \ra$.
Notice that is well-defined, since to each member $u \in Y$ corresponds a unique $\vphi$ such that $u=u_\vphi$.
It follows from hypothesis (b) that for every $\veps > 0$ there are a finite set $E \subset \Rm$ and $\theta > 0$ such that for all $u_\vphi \in Y$ we have $|\la u_\vphi, F \ra| = | \la \vphi , T \ra | \leq \theta \cdot \max \{ |\vphi(x)-\vphi(0)| : x \in E \} + \veps \cdot \|d\vphi\|_\infty = \theta \cdot \max \{ |u_\vphi(x)| : x \in E \} + \veps \cdot \|u_\vphi\|_\infty$.
It follows from \cite[5.7]{DEP.26c} that $F$ extends to a $\calP_{\rmLip}$-continuous linear functional $\hat{F} : \rmLip_0(\Rm) \to \R$.
We infer from \ref{4.1.3} that $\hat{F} \in \calF(\Rm)$.
Since readily $T = \Upsilon(\hat{F})$, we conclude that $T \in \calG(\Rm)$.\cqfd
\end{Empty}

\begin{Empty}[Distributional divergence]
\label{EX.1.2}
We consider the composition of $\Upsilon$ defined \ref{EX.1.1}(D) and $\bdiv$ defined in \ref{4.2.6} (with $\Omega = \Rm$):
\begin{equation*}
\begin{CD}
L_1(\Rm;\Rm) @>{\bdiv}>> \calF(\Rm) @>{\Upsilon}>> \calG(\Rm) \subset \calD_0(\Rm).
\end{CD}
\end{equation*}
We claim that $\Upsilon \circ \bdiv$ is the distributional divergence defined in \ref{2.3}.
Indeed, if $v \in L_1(\Rm;\Rm)$, $\vphi \in \calD(\Rm)$, and $u_\vphi = \vphi - \vphi(0) \cdot \ind_{\Rm}$ then
\begin{equation*}
\la \vphi , (\Upsilon \circ \bdiv)(v) \ra = \la u_\vphi , \bdiv v \ra = - \int_{\Rm} (\nabla u_\vphi) \ip v \, d\ssfL^m = - \int_{\Rm} (\nabla \vphi) \ip v \,d\ssfL^m.
\end{equation*}
\end{Empty}

\begin{Theorem}
\label{EX.1.3}
Let $T : \calD(\Rm) \to \R$ be linear.
The following are equivalent.
\begin{enumerate}
\item[(A)] There exists $v \in L_1(\Rm;\Rm)$ such that $\bdiv v = T$ in the sense of distributions.
\item[(B)] For every $\veps > 0$ there are a finite set $E \subset \Rm$ and $\theta > 0$ such that for every $\vphi \in \calD(\Rm)$ we have
\begin{equation*}
|\la \vphi , T \ra| \leq \theta \cdot \max \{ | \vphi(x) - \vphi(0) | : x \in E \} + \veps \cdot \|d\vphi\|_\infty.
\end{equation*}
\item[(C)] For every sequence $\la \vphi_j \ra_j$ in $\calD(\Rm)$ that converges to $0$ pointwise and satisfies $\sup_j \|d\vphi_j\|_\infty < \infty$ we have $\lim_j \la \vphi_j , T \ra = 0$.
\item[(D)] $T \in \calG(\Rm)$.
\end{enumerate}
Furthermore, for every $\veps > 0$ there exists $\bI : \calG(\Rm) \to L_1(\Rm;\Rm)$ such that
\begin{enumerate}
\item[(a)] $(\bdiv \circ \bI)(T) = T$ for all $T \in \calG(\Rm)$.
\item[(b)] $\bI$ is positively homogeneous of degree 1.
\item[(c)] $\bI$ is continuous.
\item[(d)] For all $T \in \calG(\Rm)$ we have $\|\bI(T)\|_{L_1} \leq (1+\veps) \cdot \bG(T)$. 
\end{enumerate}
\end{Theorem}

\begin{proof}
{\bf (i)}
The equivalence of (B) and (D) has been established in \ref{EX.1.1}(E).
The equivalence of (D) with (A) follows from the fact that $\Upsilon$ is an isomorphism, \ref{EX.1.2}, and \ref{3.1.4}(B).
It is obvious that (B) implies (C).
\par 
{\bf (ii)}
Here, we show that (C) implies (B) which seems to necessitate some argument.
We shall consider a localized locally convex topology $\calT_\calC$ on $\calD(\Rm)$.
The topology $\calT$ on $\calD(\Rm)$ is that of pointwise convergence.
Furthermore, $\calC = \{ C_k : k \in \N \}$, where $C_k =  \calD(\Rm) \cap \{ \vphi : \|d\vphi\|_\infty \leq k \}$.
Clearly, each $C_k$ is convex and $\calT$-closed and $\calC$ is a localizing family.
Observe also that each $C_k$ equipped with the topology $\calT \hel C_k$ of pointwise convergence is metrizable, hence sequential, by Ascoli's theorem, as in the proof of \ref{4.1.2}(C).
Therefore, \cite[3.2(B)]{DEP.26b} applies, \ie a linear functional $T : \calD(\Rm) \to \R$ is $\calT_\calC$-continuous if and only if it is sequentially $\calT_\calC$-continuous\footnote{Note that the argument at this step is somewhat delicate. In particular, I do not know whether $\calD(\Rm)[\calT_\calC]$ itself is sequential, as \cite[4.4(A)]{DEP.26b} does not apply, for a lack of $\calT$-compactness of the $C_k$. In other words, the linearity of $T$ is essential for \cite[3.2(B)]{DEP.26b} to apply.}.
On account of \cite[3.6]{DEP.26b}, condition (B) is equivalent to the $\calT_\calC$-continuity of $T$.
Owing to \cite[3.1(A)]{DEP.26b}, condition (C) is equivalent to the sequential $\calT_\calC$-continuity of $T$.
\par 
{\bf (iii)}
The existence of $\bI$ is a consequence of \ref{4.2.7} (applied with $\Omega = \Rm$).
\end{proof}

We now proceed to giving three classes of distributions to which theorem \ref{EX.1.3} applies.

\begin{Empty}[Members of $\calP_{0,\#}(\Rm)$]
\label{EX.1.4}
By definition of $\calG(\Rm)$, all members of $\calP_{0,\#}(\Rm)$ belong to $\calG(\Rm)$.
The easiest non-trivial example being $T = \theta \cdot (\lseg b \rseg - \lseg a \rseg) = \theta \cdot \partial \lseg a,b \rseg$ (which is used extensively in the proof of theorem \ref{4.2.7}).
Note that $\bG(\theta \cdot \partial \lseg a,b \rseg) = |\theta| \cdot |b-a|_2$.
More generally, if $a_1,\ldots,a_Q,b_1,\ldots,b_Q \in \Rm$ and $P_a = \sum_{i=1}^Q \lseg a_i \rseg$, $P_b = \sum_{i=1} \lseg b_i \rseg$, then $\bG(P_a-P_b)$ is equivalent to the $\calG$-distance between the $Q$-multi-points $P_a$ and $P_b$, see \eg \cite[p.767]{BOU.DEP.GOB}
\par 
It is useful to observe that members of $\calP_{0,\#}(\Rm)$ satisfy the continuity condition (B) of theorem \ref{EX.1.3} without ``$\veps$-term'' and (therefore) with a unique $\theta$ and a unique $E$.
\end{Empty}

\begin{Empty}[Signed Borel measures on $\Rm$]
\label{EX.1.5}
Let $\mu$ is a signed Borel measure on $\Rm$.
Note that a linear functional $T_\mu : \calD(\Rm) \to \R$ is well-defined by the formula $\la \vphi , T_\mu \ra = \int_{\Rm} \vphi \, d\mu$.
We claim that if
\begin{equation*}
\mu(\Rm) = 0 \quad \text{ and } \quad \int_{\Rm} |x|_2 \,d|\mu|(x) < \infty
\end{equation*}
then $T_\mu$ satisfies condition (B) in \ref{EX.1.3}.
\par 
We may, of course, assume that $\mu \neq 0$.
Let $\veps > 0$.
Choose a compact $K \subset \Rm$ such that $\int_{\Rm \setminus K} |x|_2 \,d|\mu|(x) < \veps$.
Next, choose a finite Borel partition $\calB$ of $K$ whose members have diameter less than $\veps$.
For each $B \in \calB$ pick $x_B \in B$.
Fix $\vphi \in \calD(\Rm)$.
Since $\int_{\Rm} \vphi(0) \,d\mu = 0$ we have 
\begin{multline*}
| \la \vphi , T_\mu \ra| = \left| \int_{\Rm} \vphi \, d\mu - \int_{\Rm} \vphi(0) \,d\mu \right| \\
= \left|\sum_{B \in \calB} \int_B (\vphi(x) - \vphi(0) \, d\mu(x)  \right|
+ \left| \int_{\Rm \setminus K} (\vphi(x) - \vphi(0) \,d\mu(x) \right|
\\ \leq \left|\sum_{B \in \calB} \int_B (\vphi(x) - \vphi(x_B) \, d\mu(x)  \right|
+
\left|\sum_{B \in \calB} \int_B (\vphi(x_B) - \vphi(0) \, d\mu(x)  \right| 
+  \|d\vphi\|_\infty \cdot \int_{\Rm \setminus K} |x|_2 \,d|\mu|(x) \\
\leq \sum_{B \in \calB} (\rmLip \vphi) \cdot (\rmdiam B) \cdot |\mu|(B) + \sum_{B \in \calB} |\vphi(x_B) - \vphi(0)| \cdot |\mu|(B) \\
< \veps \cdot \left( 1 + |\mu|(\Rm) \right) \cdot \|d\vphi\|_\infty + |\mu|(\Rm) \cdot \max \{ | \vphi(x_B) - \vphi(0)| : B \in \calB \}.
\end{multline*}
\par 
It is interesting to note that $\theta$ can be chosen uniformly (independent of $\veps$) in this case, in fact $\theta = |\mu|(\Rm)$.
Moreover, given $\veps > 0$ the choice of $E = \{x_B : B \in \calB\}$ depends on $\mu$ only inasmuch as it depends on $K$. 
In view of \cite[9.1]{DEP.26c}, these observations imply that for every $\kappa > 0$ and every null-sequence of positive real numbers $\la \veps_j \ra_j$ the set
\begin{equation*}
\calG(\Rm) \cap \left\{ T_\mu : \mu(\Rm) = 0 ,\, |\mu|(\Rm) \leq \kappa \text{ and } \int_{\Rm \setminus B(0,j)} |x|_2 \,d|\mu|(x) \leq \veps_j \text{ for all } j \in \N \right\}
\end{equation*}
is $\bG$-compact.
\par 
If $\mu$ is a signed Borel measure on $\Rm$ and $T_\mu \in \calG(\Rm)$ then it is easy to see (based on \ref{EX.1.3}(B)) that necessarily $\mu(\Rm)=0$.
However, I do not know whether necessarily $\int_{\Rm} |x|_2 \,d|\mu|(x) < \infty$.
\end{Empty}

\begin{Empty}[Flat chains]
\label{EX.1.6}
Here, we give examples of $T \in \calG(\Rm)$ so that, in condition \ref{EX.1.3}(B), $\theta(\veps) \to \infty$ as $\veps \to 0$ contrary to the case of signed measure presented in \ref{EX.1.5}.
We consider two null-sequence $\la a_j \ra_j$ and $\la b_j \ra_j$ in $\Rm$ as well as a sequence $\la \theta_j \ra_j$ of real numbers such that $\sum_{j=1}^\infty |\theta_j| \cdot |b_j-a_j|_2 < \infty$.
Define $P_j = \theta_j \cdot \partial \lseg a_j , b_j \rseg \in \calG(\Rm)$ and recall that $\bG(P_j) = |\theta_j| \cdot |b_j-a_j|_2$, $j \in \N$.
Abbreviate $S_k = \sum_{j=1}^k P_j$ and observe that $\la S_k \ra_k$ is $\bG$-Cauchy.
Let $T = (\bG) \lim_k S_k \in \calG(\Rm)$.
\par 
We shall also assume that no $a_j$ coincides with any $a_k$ ($j\neq k$) nor with any $b_k$ (all $k$).
Under this assumption, $\bM(T) = 2 \cdot \sum_{j=1}^\infty |\theta_j|$.
Thus, we shall assume that $\sum_{j=1}^\infty |\theta_j| = \infty$ so that $T$ is not associated with a signed measure as in \ref{EX.1.5}.
\par 
We now try to understand condition \ref{EX.1.3}(B) for this $T$.
Let $\veps > 0$.
There is an integer $j_\veps$ such that $\sum_{j > j_\veps} |\theta_j| \cdot |b_j-a_j|_2 < \veps$.
Upon letting $E = \{ a_j : j=1,\ldots,j_\veps\} \cup \{ b_j : j=1,\ldots,j_\veps\}$ and $\theta = 2 \cdot \sum_{j \leq j_\veps} |\theta_j|$ we obtain
\begin{multline*}
| \la \vphi , T \ra| = \left| \sum_{j \leq j_\veps} \theta_j \cdot (\vphi(b_j) - \vphi(0) + \vphi(0) - \vphi(a_j)) \right| + \left| \sum_{j > j_\veps} \theta_j \cdot (\vphi(b_j) - \vphi(a_j)) \right| \\
\leq \sum_{j \leq j_\veps} |\theta_j| \cdot \big( |\vphi(b_j) - \vphi(0)| + |\vphi(a_j) - \vphi(0)| \big) + \|d\vphi\|_\infty \cdot \sum_{j > j_\veps} |\theta_j| \cdot |b_j-a_j|_2 \\
\leq \theta \cdot \max \{ | \vphi(x)-\vphi(0)| : x \in E \} + \veps \cdot \|d\vphi\|_\infty.
\end{multline*}
for all $\vphi \in \calD(\Rm)$.
\par 
In the above example, $\rmspt T$ is countable.
Of course, this is not necessary and there are many variants of this construction.
For instance, one could replace each $P_j$ by a mollified version $\Phi_{r_j} * P_j$.
A different way of modifying $P_j$ is as follows.
Let $A_j$ and $B_j$ be compact sets in $\Rm$ and $f_j : A_j \to B_j$ a homeomorphism such that $0 < | f_j(x)-x|_2 \leq \delta_j$ for all $x \in A_j$. 
Let also $\nu_j$ be a Borel probability measure on $A_j$.
One can then define a 1-dimensional current $P_j = \theta_j \int_{A_j} \lseg x,f_j(x) \rseg \,d\nu_j(x)$ and one can check that $\bM(P_j) \leq |\theta_j| \cdot \delta_j$ and $\partial P_j = \theta_j \int_{A_j} (\lseg f_j(x) \rseg - \lseg x \rseg ) \,d\nu_j(x)$.
\par 
All these examples are flat chains.
In fact, $\calG(\Rm)$ exactly consists in the 0-dimensional flat chains $T$ in $\Rm$ that satisfy the condition ``$\la \ind_{\Rm}, T \ra = 0$'' (it is possible to make sense of this for flat chains).
See \cite{DEP.26b} for more on this topic.
\end{Empty}

\begin{Empty}[The case $\Omega \neq \Rm$]
\label{EX.1.7}
When $\Omega \neq \Rm$, distributions, \ie members of $\calD_0(\Omega)$, seem impractical to describe members of $\calF(\Omega)$.
The problem is that \ref{EX.1.1}(C) does not extend to this case.
If $a \in \Omega$, $u = |\cdot-a|_2 \in \rmLip_0(\Omega)$, $b \in \rmbdry \Omega$, and $\la \vphi_j \ra_j$ is a sequence in $\calD(\Omega)$ that converges pointwise to $u$ and such that $\sup_j \|d\vphi_j\|_\infty  < \infty$ then $\la \vphi_j \ra_j$ converges to $u$ locally uniformly, by Ascoli's theorem, therefore, $d\vphi_j$ is unbounded (with respect to $j$) in any neighborhood of $b$, since $u(b) \neq 0$, a contradiction.
Notwithstanding, the three classes of examples presented above survive in this context provided they are interpreted appropriately (transferred to $\calF(\Rm)$ via $\Upsilon^{-1}$ first and then repeated verbatim in the context of $\Omega$ for domains that satisfy the hypotheses of theorem \ref{4.2.7}).
\par 
Another possible option is to consider distributions in $\Rm$ that are supported in $\rmclos \Omega$.
We have not traveled this route here, for inspiration in this respect see \eg \cite{DEP.26b}.
\end{Empty}

\subsection{Examples of members of \texorpdfstring{$\SCH_{0,\calW,p}(\Omega)$}{SCH-0-W-p(Omega)}}
\label{subsec.EX.SCH.RM}

As in subsection \ref{subsec.PCI}, throughout this subsection, $\Omega \subset \Rm$ is a non-empty connected open set and we make the following two assumptions:
\begin{enumerate}
\item[(1)] $1 < p \leq 1^*$.
\item[(2)] $\Omega$ satisfies the $(p,1)$-Poincar\'e inequality.
\end{enumerate}

%
%
%

\begin{Empty}[Members of $\SCH_0(\Rm)$ as distributions]
\label{EX.2.2}
As was done in \ref{EX.1.1}(E) for members of $\calF(\Rm)$ we now interpret the members of $\SCH_0(\Rm)$ as distributions.
The situation is slightly simpler, since there is no need to introduce an analog of $\Upsilon$.
Upon noticing that $\calD(\Rm) \subset \bvs(\Rm)$ we may consider the restriction map
\begin{equation*}
\rmrestr : \SCH_0(\Rm) \to \calD_0(\Rm) : F \mapsto F|_{\calD(\Rm)}
\end{equation*}
($\rmrestr F$ is, indeed, a distribution, owing the to the $\|\cdot\|_{\bvs}$ continuity of $F$).
\par
We will need the following lemma, analog of \ref{EX.1.1}(C).
\begin{enumerate}
\item[(A)] {\it If $u \in \bvs(\Rm)$ and $\veps > 0$ then there exists $\vphi_u \in \calD(\Rm)$ such that $\|D\vphi_u\|(\Rm) < \veps + \|Du\|(\Rm)$ and $\left| \int_{\Rm} (\vphi_u - u) \cdot f \,d\ssfL^m \right| < \veps$ for all $f \in L_m(\Rm)$ such that $\|f\|_{L_m}=1$.} 
\end{enumerate}
\par 
{\it Proof.}
This is very similar to \ref{3.2.2}(D).
Choose $r > 0$ small enough for $\|u_r - u \|_{L_{1^*}} < \veps$, where $u_r = \Phi_r * u$.
Note that $\|Du_r\|(\Rm) \leq \|Du\|(\Rm)$.
Next, choose $R > 0$ large enough for $\| u \cdot \ind_{\Rm \setminus B(0,R)}\|_{L_{1^*}} < \veps$.
Let $\chi \in \calD(\Rm)$ be such that $\ind_{B(0,R)} \leq \chi \leq \ind_{B(0,2 \cdot R)}$ and $\|\nabla \chi \|_{\infty} \leq 2$.
Put $\vphi_u = \chi \cdot u_r$.
By the triangle inequality, $\left| \int_{\Rm} (\vphi_u - u) \cdot f \,d\ssfL^m \right| < 2 \cdot \veps \cdot \|f\|_{L_m}$ for all $f \in L_m(\Rm)$.
Moreover, $\|D\vphi_u\|(\Rm) \leq \|u_r \cdot \nabla \chi\|_{L_1} + \|\chi \cdot \nabla u_r \|_{L_1} \leq \| u_r \cdot \ind_{\Rm \setminus B(0,R)} \|_{L^{1^*}} \cdot \|\nabla \chi\|_{L_m} + \|\nabla \chi \|_\infty \cdot \|\nabla u_r\|_{L_1} \leq \veps \cdot 4 \cdot \balpha(m)^\frac{1}{m} + \|Du\|(\Rm)$.\cqfd
\begin{enumerate}
\item[(B)] {\it Let $T : \calD(\Rm) \to \R$ be a linear functional. The following are equivalent.
\begin{enumerate}
\item[(a)] $T = \rmrestr F$ for some $F \in \SCH_0(\Rm)$.
\item[(b)] $(\forall \veps > 0)(\exists f \in L_m(\Rm))(\forall \vphi \in \calD(\Rm))$:
\begin{equation*}
| \la \vphi , T \ra | \leq  \left| \int_{\Rm} \vphi \cdot f \, d\ssfL^m \right| + \veps \cdot  \| \nabla \vphi \|_{L_1(\Rm)}.
\end{equation*}
\item[(c)] $(\forall \veps > 0)(\exists E \in \rmFin[(\wh{\Rm})_m])\footnote{Recall that $(\wh{\Rm})_m$ is the unit sphere of $L_m(\Rm)$, \ie consists of those $f \in L_m(\Rm)$ such that $\|f\|_{L_m}=1$.}(\exists \theta > 0)(\forall \vphi \in \calD(\Rm))$:
\begin{equation*}
| \la \vphi , T \ra | \leq  \theta \cdot \max \left\{ \left| \int_{\Rm} \vphi \cdot f \, d\ssfL^m \right| : f \in E \right\} + \veps \cdot  \| \nabla \vphi \|_{L_1(\Rm)}.
\end{equation*}
\item[(d)] For every sequence $\la \vphi_j \ra_j$ in $\calD(\Rm)$ such that $\lim_j \int_{\Rm} \vphi_j \cdot f \,d\ssfL^m = 0$ for all $f \in L_m(\Rm)$ and $\sup_j \|\nabla \vphi_j\|_{L_1(\Rm)} < \infty$ we have $\lim_j \la \vphi_j , T \ra = 0$.
\end{enumerate}
}
\end{enumerate}
\par 
{\it proof.}
{\bf (i)}
Assume that $T = \rmrestr F$, $F \in \SCH_0(\Rm)$.
Recalling the definition of $\SCH_0(\Rm)$ \ref{3.2.2}, given $\veps > 0$ there exists $f \in L_m(\Rm)$ such that $\|\bLambda(f) - F\|_{\bvs}^* < \veps$.
As in the proof of theorem \ref{5.2.4}{\bf (i)}, we infer that condition (b) holds, since $\|D\vphi\|(\Rm) = \|\nabla \vphi\|_{L_1(\Rm)}$.
Thus, $(a) \Rightarrow (b)$.
\par 
{\bf (ii)}
One trivially checks that $(b) \Rightarrow (c)$, by taking $\theta = \|f\|_{L_m}$ and $E = \{ \theta^{-1} \cdot f\}$ in case $f \neq 0$ and choosing $E$ and $\theta$ arbitrarily otherwise.
\par 
{\bf (iii)}
The proof that $(c) \Rightarrow (a)$ is application of \cite[5.7]{DEP.26b}.
In order for this to apply we ought to check that $\calD(\Rm)$ is uniformly sequentially $\calW_{1^*,TV}$-dense in $\bvs(\Rm)$ (see \cite[5.6]{DEP.26b} for the definition).
This is a direct consequence of (A) above.\cqfd
\par 
{\bf (iv)}
It is obvious that $(c) \Rightarrow (d)$.
The proof that $(d) \Rightarrow (c)$ is similar to that of $(C) \Rightarrow (D)$ in theorem \ref{EX.1.3}.
We equip $\calD(\Rm)$ with the localized topology $\calT_\calC$ associated with the topology $\calT$ being the restriction to $\calD(\Rm)$ of the weak* topology $\sigma(L_{1^*},L_m)$ of $L_{1^*}(\Rm)$ and $\calC = \{ C_k : k \in \N \}$, where $C_k = \calD(\Rm) \cap \{ \vphi : \|\nabla \vphi\|_{L_1(\Rm)} \leq k \}$.
Clearly, each $C_k$ is convex and $\calT$-closed and $\calC$ is a localizing family.
Observe also that each $C_k$ equipped with the topology $\calT \hel C_k$ of pointwise convergence is metrizable, hence sequential, by separability of $L_m(\Rm)$, as in the proof of \ref{5.2.1}(C).
Therefore, \cite[3.2(B)]{DEP.26b} applies, \ie a linear functional $T : \calD(\Rm) \to \R$ is $\calT_\calC$-continuous if and only if it is sequentially $\calT_\calC$-continuous.
In view of \cite[3.6]{DEP.26b}, condition (B) is equivalent to the $\calT_\calC$-continuity of $T$.
In the light of \cite[3.1(A)]{DEP.26b}, condition (C) is equivalent to the sequential $\calT_\calC$-continuity of $T$.\cqfd
\begin{enumerate}
\item[(C)] {\it Let $T : \calD(\Rm) \to \R$ be a linear functional. There exists $v \in C_0(\Rm;\Rm)$ such that $T = \bdiv v$ in the sense of distributions if and only if $T$ satisfies the equivalent condition (b), (c), and (d) in (B) above.}
\end{enumerate}
\par 
{\it Proof.}
This is an immediate consequence of (B) above and theorems \ref{5.2.3} and \ref{5.2.4}.\cqfd
\end{Empty}

\begin{Empty}[The case $\Omega \neq \Rm$]
\label{EX.2.3}
For a similar reason as explained in \ref{EX.1.7}, when $\Omega \neq \Rm$ it is impractical to attempt to characterize members of $\SCH_{0,\calW,p}(\Omega)$ as distributions.
The problem is that \ref{EX.2.2}(A) does not extend\footnote{Specifically, there is not enough room near $\rmbdry \Omega$ to define $\chi$, as in the proof of \ref{EX.2.2}(A), with a uniform bound on $\|\nabla \chi\|_\infty$.}
Nevertheless, below we give a version of theorem \ref{5.2.3} by explicitly stating the continuity condition.
\par 
It is probably worth noting that, since $\calD(\Rm)$ naturally embeds in $BV_{p,\rmcst}(\Omega)$, the solutions of the equation $\bdiv v = F$ considered below are, in particular, solutions in the sense of distributions.
\end{Empty}

\begin{Theorem}
\label{EX.2.4}
Assume that $1 < p \leq 1^*$ and $\Omega$ satisfies the $(p,1)$-Poincar\'e inequality.
Let $F : BV_{p,\rmcst}(\Omega) \to \R$ be linear.
The following are equivalent.
\begin{enumerate}
\item[(A)] There exists $v \in C_0(\Omega;\Rm)$ such that $\bdiv v = F$.
\item[(B)] For every $\veps > 0$ there is a finite set $E$ in the unit ball of $L_{q,\#}(\Omega)$ and there is $\theta > 0$ such that for every $u \in BV_{p,\rmcst}(\Omega)$ we have
\begin{equation*}
|\la [u] , F \ra| \leq \theta \cdot \max \left\{ \left| \int_\Omega u \cdot f \, d\ssfL^m \right| : f \in E \right\} + \veps \cdot \|Du\|(\Omega).
\end{equation*}
\item[(C)] For every sequence $\la [u_j] \ra_j$ in $BV_{p,\rmcst}(\Omega)$ such that $\lim_j \int_\Omega u_j \cdot f \,d\ssfL^m = 0$ for all $f \in L_{q,\#}(\Omega)$ and $\sup_j \|Du_j\|(\Omega) < \infty$ we have $\lim_j \la u_j , F \ra = 0$.
\end{enumerate}
Furthermore, in this case, given $\veps > 0$, $v$ can be chosen so that (A) holds and $\|v\|_\infty \leq (1+\veps) \cdot \vvvert F \vvvert$, where $\vvvert F \vvvert = \sup \{ |\la u,F \ra| : u \in BV_{p,\rmcst}(\Omega) \text{ and }\|Du\|(\Omega) \leq 1 \}$.
\end{Theorem}

\begin{proof}
{\bf (i)}
According to \ref{5.2.2}(B) and theorem \ref{5.2.3}, the linear forms $F$ as in (A) are characterized as being those that are $\calW_{p,TV}$-continuous.
By \ref{5.2.1}(F), condition (B) characterizes $\calW_{p,TV}$-continuity as well.
Thus, $(A) \iff (B)$.
\par 
{\bf (ii)}
In view of \ref{5.2.1}(D), $F$ is $\calW_{p,TV}$-continuous if and only if it sequentially $\calW_{p,TV}$-continuous.
The equivalence of (A), (B), (C) now follows from the characterization of $\calW_{p,TV}$-convergent sequences as in \cite[3.1(A)]{DEP.26b}.
\par 
{\bf (iii)}
The ``Furthermore'' conclusion here follows from like-named conclusion of theorem \ref{5.2.3}.
\end{proof}

We now proceed to giving two classes of examples to which theorem \ref{EX.2.4} applies.

\begin{Empty}[Members of $L_{q,\#}(\Omega)$]
\label{EX.2.5}
Let $0 \neq f \in L_{q,\#}(\Omega)$ and observe that the formula $\bLambda_q(f) : BV_{p,\rmcst}(\Omega) \to \R : u \mapsto \int_{\Omega} u \cdot f \, d\ssfL^m$ well-defines a linear form.
It most obviously satisfies condition (B) of theorem \ref{EX.2.4} with $\theta = \|f\|_{L_q(\Omega)}$ and $E = \{ \theta^{-1} \cdot f \}$.
In particular, $\theta$ is independent of $\veps$ in this case.
Note the analogy with \ref{EX.1.4}.
In case $p=1^*$ (hence, $q=m$) and $\Omega = \ts ]0,2\cdot\pi[^m$, this was obtained in \cite[\S3 proof of proposition 1]{BOU.BRE.03}.
\end{Empty}

\begin{Empty}[Members of the closure of $L_{q,\#}(\Omega)$ in $L_{q,\infty}(\Omega)$]
\label{EX.2.6}
We recall the definition of the weak Lebesgue space $L_{q,\infty}(\Omega)$.
To this end, given a Borel-measurable function $f : \Omega \to \R$ and $y > 0$ we let $A(f,y) = \Omega \cap \{ x : |f(x)| > y \}$ and $d(f,y) = \ssfL^m(A(f,y)) \in [0,\infty]$.
Thus, $d(f,\cdot)$ is the distribution function of $f$.
Then $L_{q,\infty}(\Omega)$ consists of all those $f$ such that for all $y > 0$ we have $d(f,y) \leq \frac{\bc^q}{y^q}$, for some $\bc > 0$.
In that case, the smallest such $\bc$ is denoted $\|f\|_{L_{q,\infty}}$ and this defines a quasi-norm on $L_{q,\infty}(\Omega)$ such that $\|f\|_{L_{q,\infty}} \leq \|f\|_{L_q}$, see \eg \cite[p.5-6]{GRAFAKOS.1}.
The corresponding structure of a topological vector space is, in fact, completely normable, see \eg \cite[exercise 1.1.12 and theorem 1.4.11]{GRAFAKOS.1}.
Additionally, it is easy to observe that if $\ssfL^m(\Omega) < \infty$ then $L_{q,\infty}(\Omega) \subset L_1(\Omega)$, see \eg \cite[exercise 1.1.11(b)]{GRAFAKOS.1}.
\par 
Furthermore, we associate with any $f \in L_0(\Omega,\calB(\Omega))$ a function $$\bveps_q(f,\cdot) : \R_{> 0} \to [0,\infty] : y \mapsto d(f,y) \cdot y^q.$$
It is easy to characterize membership to either $L_q(\Omega)$ or $L_{q,\infty}(\Omega)$ by means of $\bveps_q$, namely: $f \in L_q(\Omega)$ iff $\int_0^\infty \frac{\bveps_q(f,y)}{y}\,dy < \infty$ and $f \in L_{q,\infty}(\Omega)$ iff $\|\bveps_q(f,\cdot)\|_\infty < \infty$.
We introduce the intermediate space
\begin{equation}
\label{eq.new.weak}
L_{q,0}(\Omega) = L_0(\Omega,\calB(\Omega)) \cap \left\{ f : \lim_{y \to 0^+} \bveps_q(f,y) = 0 = \lim_{y \to \infty} \bveps_q(f,y) \right\}.
\end{equation}
It is useful to observe that $d(f,y) < \infty$ for all $y > 0$ whenever $f \in L_{q,0}(\Omega)$ and, in fact, that $L_{q,0}(\Omega) \subset L_{q,\infty}(\Omega)$.
\begin{enumerate}
\item[(A)] {\it $L_{q,0}(\Omega) = \rmclos_{L_{q,\infty}}[L_q(\Omega)]$.}
\end{enumerate}
\par 
{\it Proof.}
{\bf (i)}
Let $f \in \rmclos_{L_{q,\infty}}[L_q(\Omega)]$.
There exists a sequence $\la f_j \ra_j$ in $L_q(\Omega)$ such that $\lim_j\|f-f_j\|_{L_{q,\infty}}=0$.
For each $j$ there exists $t_j > 0$ such that $\int_{\Omega \cap \{ |f_j| > t_j \}} |f_j|^q \,d\ssfL^m < j^{-q}$ and, upon letting $g_j = f_j \cdot \ind_{\Omega \cap \{ |f_j| \leq t_j \}}$, there exists $r_j > 0$ such that $\int_{\Omega \setminus B(0,r_j)} |g_j|^q \,d\ssfL^m < j^{-q}$.
Thus, if $h_j = g_j \cdot \ind_{\Omega \cap B(0,r_j)}$ then $\|f_j - h_j\|_{L_{q,\infty}} \leq \|f_j - h_j\|_{L_q} < 2 \cdot j^{-1}$ and $h_j \in L_{q,0}(\Omega)$, since $d(h_j,y) = 0$ for all $y > t_j$ and $d(h_j,y) \leq \ssfL^m(\Omega \cap B(0,r_j))$ for all $y > 0$.
\par 
For all $j$ and $y$ we have $d(f,y) \leq d(h_j,y/2) + d(f-h_j,y/2) \leq 2^{-q} [ y^{-q} \cdot \veps_q(h_j,y/2) + y^{-q} \cdot \|f-h_j\|^q_{L_{q,\infty}}]$ and, multiplying by $y^q$ both sides, we obtain 
\begin{equation*}
\bveps_q(f,y) \leq 2^q \left( \bveps_q \left( h_j , \frac{y}{2}\right) + \|f-h_j\|^q_{L_{q,\infty}} \right).
\end{equation*}
Since $y$ is arbitrary we infer that, for every $j$,
\begin{equation*}
\max \left\{ \limsup_{y \to 0^+} \bveps_q(f,y) , \limsup_{y \to \infty} \bveps_q(f,y) \right\} \leq 2^q \cdot \|f-h_j\|^q_{L_{q,\infty}}.
\end{equation*}
Upon letting $j \to \infty$ we conclude that $f \in L_{q,0}(\Omega)$.
\par 
{\bf (ii)}
Let $f \in L_{q,0}(\Omega)$.
For each $j$ define a Borel-measurable $\tilde{f}_j : \Omega \to \R$ by $\tilde{f}_j(x) = f(x)$ if $|f(x)| \leq j$ and $\tilde{f}_j(x) = \rmsign(f(x)) \cdot j$ otherwise.
Next, define $f_j = \tilde{f}_j \cdot \ind_{\Omega \cap B(0,j)}$.
Given $x \in \Omega$ observe that if $x \in B(0,j)$ then $|(f-f_j)(x)| = |f(x)|-j$ if $x \in A(f,j)$ and $|(f-f_j)(x)| = 0$ if $x \not \in A(f,j)$ whereas, if $x \not \in B(0,j)$ then $|(f-f_j)(x)| = |f(x)|$.
Accordingly,
\begin{multline*}
A(f-f_j,y) = [A(f-f_j,y) \cap B(0,j)] \cup [A(f-f_j,y) \setminus B(0,j)] \\
\subset [A(f,y+j) \cap B(0,j)] \cup [A(f,y) \setminus B(0,j)] 
\end{multline*}
and, in turn,
\begin{equation}
\label{eq.EX.1}
d(f-f_j,y) \leq d(f,y+j) + \ssfL^m[ A(f,y) \setminus B(0,j)] =: \rmI_j(y) + \rmII_j(y).
\end{equation}
Let $\eta > 0$.
Since $f \in L_{q,0}(\Omega)$ there exists $y_* > 0$ such that $\bveps_q(f,y) < \eta$ for all $0 < y < y_*$ and there exists $y^* > y_*$ such that $\bveps_q(f,y) < \eta$ for all $y > y^*$.
Choose an integer $j_0 > y^*$.
Choose also a positive integer $j_1$ such that $\ssfL^m [ A(f,y_*) \setminus B(0,j)] < \frac{\eta}{(y^*)^q}$ whenever $j \geq j_1$ (recall that $d(f,y_*) < \infty$).
Define $j_* = \max \{j_0,j_1 \}$ and, for the remaining part of this proof, assume that $j \geq j_*$.
\par 
Here, we find an upper bound for $\rmI_j(y)$.
Notice that for all $y > 0$ we have $y + j \geq y + j_0 > y^*$, so that
\begin{equation*}
\rmI_j(y) = \frac{\bveps_q(f,y+j)}{(y+j)^q} < \frac{\eta}{(y+j)^q} < \frac{\eta}{y^q}.
\end{equation*}
\par 
Next, we find an upper bound for $\rmII_j(y)$ by distinguishing between the following cases.
If $0 < y < y_*$ or $y^* < y$ then 
\begin{equation*}
\rmII_j(y)  \leq \ssfL^m[A(f,y)] = d(f,y) = \frac{\bveps_q(f,y)}{y^q} < \frac{\eta}{y^q}.
\end{equation*}
If $y_* \leq y \leq y^*$ then 
\begin{equation*}
\rmII_j(y) = \ssfL^m[ A(f,y) \setminus B(0,j) ] \leq \ssfL^m[ A(f,y_*) \setminus B(0,j) ] < \frac{\eta}{(y^*)^q} < \frac{\eta}{y^q}
\end{equation*}
by the choice of $j_1$.
In any case, we have $\rmII_j(t) < \frac{\eta}{y^q}$ for all $y > 0$.
\par 
It now follows from \eqref{eq.EX.1} that $d(f-f_j,y) \leq \frac{2 \cdot \eta}{y^q}$ for all $y > 0$, hence, $\|f-f_j\|_{L_{q,\infty}} \leq \sqrt[q]{2 \cdot \eta}$ for all $j \geq j_*$.
As $\eta$ is arbitrary, we conclude that $f \in \rmclos_{L_{q,\infty}}[L_q(\Omega)]$.\cqfd
\begin{enumerate}
\item[(B)] {\it $L_q(\Rm) \subsetneq L_{q,0}(\Rm) \subsetneq L_{q,\infty}(\Rm)$}.
\end{enumerate}
\par 
{\it Proof.}
It is not hard to check that $f(x) = \frac{1}{|x|_2^\frac{m}{q} \cdot \sqrt[q]{1 + |\log |x|_2|}}$, $x \neq 0$, defines a member of $L_{q,0}(\Rm) \setminus L_q(\Rm)$ and that $g(x) = \frac{1}{|x|_2^\frac{m}{q}}$, $x \neq 0$, defines a member of $L_{q,\infty}(\Rm) \setminus L_{q,0}(\Rm)$.\cqfd
\par 
Of course, one can replace in the definition of $f$ (in the previous proof) the logarithm by iterated logarithms or functions whose decay is even slower.
We now introduce the following spaces of functions:
\begin{equation*}
\begin{split}
L_{q,\infty,\#}(\Omega) & = \begin{cases}
L_{q,\infty}(\Omega) \cap \left\{ f : \int_\Omega f \, d\ssfL^m = 0 \right\} & \text{if } \ssfL^m(\Omega) < \infty \\
L_{q,\infty}(\Omega) & \text{if } \ssfL^m(\Omega) = \infty
\end{cases} \\
L_{q,0,\#}(\Omega) & = L_{q,0}(\Omega) \cap L_{q,\infty,\#}(\Omega).
\end{split}
\end{equation*}
These definitions make sense, since $L_{q,\infty}(\Omega) \subset L_1(\Omega)$ whenever $\ssfL^m(\Omega) < \infty$.
The following is a useful observation.
\begin{enumerate}
\item[(C)] {\it $L_{q,0,\#}(\Omega) = \rmclos_{L_{q,\infty}}[L_{q,\#}(\Omega)]$.}
\end{enumerate}
\par 
{\it Proof.}
If $\ssfL^m(\Omega)=\infty$ there is nothing to prove, since the present claim is a rephrasing of (A) in that case.
We henceforth assume that $\ssfL^m(\Omega) < \infty$ and observe that $\left| \int_\Omega h \,d\ssfL^m \right| \leq \int_\Omega |h| \,d\ssfL^m \leq \bc \cdot \|h\|_{L_{q,\infty}}$ for all $h \in L_{q,\infty}(\Omega)$, where $\bc = \frac{q}{q-1} \cdot \ssfL^m(\Omega)^\frac{q-1}{q}$, see \eg \cite[exercise 1.1.11(a)]{GRAFAKOS.1}.
If $f \in \rmclos_{L_{q,\infty}}[L_{q,\#}(\Omega)]$, given $\la f_j \ra_j$ in $L_{q,\#}(\Omega)$ that $\|\cdot\|_{L_{q,\infty}}$-converges to $f$, we infer that $\left| \int_\Omega f \,d\ssfL^m \right| = \left| \int_\Omega (f-f_j)\,d\ssfL^m \right| \leq \bc \cdot \|f-f_j\|_{L_{q,\infty}} \to 0$ as $j \to \infty$, hence, $\int_\Omega f \, d\ssfL^m = 0$.
Conversely, if $f \in L_{q,0,\#}(\Omega)$ then there exists $\la f_j \ra_j$ in $L_q(\Omega)$ such that $\lim_j \|f-f_j\|_{L_{q,\infty}}=0$, by (A).
Letting $y_j = (f_j)_\Omega$ we note that $\ssfL^m(\Omega) \cdot |y_j| = \left| \int_\Omega f_j \,d\ssfL^m \right| = \left| \int_\Omega (f_j-f) \,d\ssfL^m \right| \leq \bc \cdot \|f-f_j\|_{L_{q,\infty}}$.
Then $g_j = f_j - (f_j)_{\Omega} \cdot \ind_\Omega \in L_{q,\#}(\Omega)$ and $\|f - g_j\|_{L_{q,\infty}} \leq 2 \cdot \left( \|f-f_j\|_{L_{q,\infty}} + |y_j| \cdot \|\ind_\Omega\|_{L_{q,\infty}}\right) \to 0$ as $j \to \infty$.\cqfd
\par 
From now on we specialize to $p=1^*$ and $q=m$.
The next result is inspired by the argument in \cite[theorem 2.2]{COH.DEV.TAD.24} where the case $\Omega = \Rm$ is covered. 
\begin{enumerate}
\item[(D)] {\it Assume that either  $\Omega = \Rm$ or $\Omega$ is a bounded, connected $BV$-extension set. 
If $[u] \in BV_{1^*,\rmcst}(\Omega)$ and $f \in L_{m,\infty,\#}(\Omega)$ then $u \cdot f \in L_1(\Omega)$. 
Moreover, the linear form
$
BV_{1^*,\rmcst}(\Omega) \to \R : [u] \mapsto \int_\Omega u \cdot f \, d\ssfL^m
$
is well-defined and continuous. 
In fact, 
\begin{equation*}
\left| \int_\Omega u \cdot f \, d\ssfL^m \right| \leq \bc(\Omega) \cdot \kappa_m \cdot \frac{m }{m-1} \cdot \|f\|_{L_{m,\infty}} \cdot \|Du\|(\Omega)
\end{equation*}
where
\begin{equation*}
\bc(\Omega) = \begin{cases}
1 & \text{if } \Omega = \Rm \\
\bc_{\rmext}(\Omega) \cdot \left( 1 + \bc_{1^*}(\Omega) \cdot \ssfL^m(\Omega)^\frac{1}{m}\right) & \text{if } \Omega \neq \Rm.
\end{cases}
\end{equation*}
}
\end{enumerate}
\par 
{\it Proof.}
{\bf (i)}
Note that the summability of $u \cdot f$ will be proved upon showing that the claimed estimate holds with $\int_\Omega |u| \cdot |f| \,d\ssfL^m$ in place of $\left| \int_\Omega u \cdot f \, d\ssfL^m \right|$.
This is achieved in the following way:
\begin{equation*}
\int_\Omega |u| \cdot |f| \,d\ssfL^m = \int_\Omega |f| \left( \int_0^\infty \ind_{A(u,y)}\, d\ssfL^1(y)\right) d\ssfL^m = \int_0^\infty \left( \int_{A(u,y)} |f| \,d\ssfL^m \right)d\ssfL^1(y),
\end{equation*}
and observing that
\begin{multline*}
\int_{A(u,y)} |f| \,d\ssfL^m = \int_0^\infty \ssfL^m(A(u,y) \cap A(f,t))\,d\ssfL^1(t) \leq \int_0^\infty \min \left\{ \ssfL^m(A(u,y)) , d(f,t) \right\} \,d\ssfL^1(t) \\
\leq \int_0^\infty \min \left\{ \ssfL^m(A(u,y)) , \frac{\|f\|_{L_{m,\infty}}^m}{t^m} \right\} \,d\ssfL^1(t)
= \frac{m}{m-1} \cdot \|f\|_{L_{m,\infty}} \cdot \ssfL^m(A(u,y))^\frac{m-1}{m} \\
\leq \frac{m}{m-1} \cdot \|f\|_{L_{m,\infty}} \cdot \ssfL^m(A(\hat{u},y))^\frac{m-1}{m}
\leq \kappa_m \cdot \frac{m}{m-1} \cdot \|f\|_{L_{m,\infty}} \cdot \|D \ind_{A(\hat{u},y)}\|(\Rm),
\end{multline*}
by the isoperimetric inequality \cite[4.5.9(31)]{GMT}, where $\hat{u} = u$ in case $\Omega = \Rm$ and $\hat{u}$ is a $BV$ extension of $u$ to $\Rm$ as in \ref{bv.ext} otherwise. 
We then apply the coarea formula \cite[theorem 5.9(i)]{EVANS.GARIEPY.2} to infer that
\begin{multline*}
\int_\Omega |u| \cdot |f| \,d\ssfL^m \leq \kappa_m \cdot \frac{m}{m-1} \cdot \|f\|_{L_{m,\infty}} \cdot \int_0^\infty \|D \ind_{A(\hat{u},y)}\|(\Rm)\,d\ssfL^1(y)
\\ = \kappa_m \cdot \frac{m}{m-1} \cdot \|f\|_{L_{m,\infty}} \cdot \|D\hat{u}\|(\Rm).
\end{multline*}
This establishes the integrability of $u \cdot f$ and in case $\Omega = \Rm$ the proof is complete.
\par 
{\bf (ii)}
Assume now that $\Omega$ is a bounded, connected $BV$-extension set.
The integral $\int_\Omega u \cdot f \, d\ssfL^m$ does not depend upon the choice of a representative of $[u]$, by definition of $L_{m,\infty,\#}(\Omega)$, as $\ssfL^m(\Omega) < \infty$.
Replacing $u$ with $u - (u)\Omega \cdot \ind_\Omega$ we infer from the definition of $BV$-extension set and from \ref{bv.ext}(A)(b) that
\begin{multline*}
\|D\hat{u}\|(\Rm) \leq \bc_\rmext(\Omega) \cdot \left[ \|u-(u)_\Omega \cdot \ind_\Omega\|_{L_1(\Omega)} + \|Du\|(\Omega) \right]\\
\leq \bc_\rmext(\Omega) \cdot \left[ \ssfL^m(\Omega)^\frac{1}{m} \cdot \bc_{1^*}(\Omega) \cdot \|Du\|(\Omega) + \|Du\|(\Omega)\right].\cqfd
\end{multline*}
\par 
The above result yields a bounded linear operator
\begin{equation*}
\bLambda_{m,\infty,\#} : L_{m,\infty,\#}(\Omega) \to BV_{1^*,\rmcst}(\Omega)^* : f \mapsto \left( [u] \mapsto \int_\Omega u \cdot f \,d\ssfL^m \right),
\end{equation*}
where $BV_{1^*,\rmcst}(\Omega)$ is equipped with the norm $\|[u|\| = \|Du\|(\Omega)$.
We are now ready to provide a new class of examples of right members $F$ is the equation $\bdiv F = v$ to which theorem \ref{EX.2.5} applies.
\begin{enumerate}
\item[(E)] {\it Assume that either $\Omega = \Rm$ or $\Omega$ is a bounded, connected $BV$-extension set. 
If $f \in L_{m,0,\#}(\Omega)$ then $\bLambda_{m,\infty,\#}(f)$ is $\calW_{1^*,TV}$-continuous and $\vvvert \bLambda_{m,\infty,\#}(f) \vvvert \leq \bc(\Omega) \cdot \|f\|_{L_{m,\infty}}$, where $\bc(\Omega)$ is the constant in (D) above.
In particular, for every $\veps > 0$ there exists $v \in C_0(\Omega;\Rm)$ such that $\bdiv v = \bLambda_{m,\infty,\#}(f)$ and $\|v\|_\infty \leq (1+\veps) \cdot \bc(\Omega) \cdot \|f\|_{L_{m,\infty}}$. 
}
\end{enumerate}
\par 
Notice that $v$ above is, in particular, a solution of $\rmdiv v = f$ in the sense of distributions on $\Omega$, since $\calD(\Omega) \subset BV_{1^*}(\Omega)$.
\par 
{\it Proof.}
It suffices to show that $F = \bLambda_{m,\infty,\#}(f)$ satisfies condition (B) in theorem \ref{EX.2.4}.
Let $\veps > 0$ and abbreviate $\hat{\veps} = \veps \cdot ( \bc(\Omega) \cdot \kappa_m \cdot \frac{m}{m-1})^{-1}$, where $\bc(\Omega)$ is as in (D).
According to (C) above, there exists $g \in L_{m,\#}(\Omega)$ such that $\|f-g\|_{L_{m,\infty}} < \hat{\veps}$.
It immediately follows from (D) that
\begin{multline*}
| \la [u] , \bLambda_{m,\infty,\#}(f) \ra | \leq | \la [u] , \bLambda_{m,\infty,\#}(g) \ra | + | \la [u] , \bLambda_{m,\infty,\#}(f-g) \ra | \\
\leq \left| \int_\Omega u \cdot g \, d\ssfL^m \right| + \bc(\Omega) \cdot \kappa_m \cdot \frac{m}{m-1} \cdot \|f-g\|_{L_{m,\infty}} \cdot \|Du\|(\Omega 
\leq \theta \left| \int_\Omega u \cdot \tilde{g} \, d\ssfL^m \right| + \veps \cdot \|Du\|(\Omega),
\end{multline*}
where $\theta = \|g\|_{L_m}$ and $\tilde{g} = \theta^{-1} \cdot g$\cqfd
\end{Empty}

\subsection{Examples of members of \texorpdfstring{$\SCH_{0,\calM,p}(\Omega)$}{SCH-0-M-p(Omega)}}
\label{subsec.EX.SCH.OMEGA.1}

\begin{Empty}[A variant of Whitney partitions of unity]
\label{WHITNEY} 
Here, we simply observe that Whitney's coverings of $\Omega$ can be chosen to consist of balls whose radius is as small as we wish to stipulate with respect to the distance from their center to the boundary of $\Omega$.
There will be a dimensional constant $\bc_{\theTheorem}(m)$ appearing in upper bounds in (D), (F), and (G) below whose exact value is irrelevant and can be taken equal to $18 \cdot (392)^m$.
\par 
In this whole number, $\emptyset \neq \Omega \subsetneq \Rm$ is open and $0 < \tau \leq 1$.
We abbreviate $C = \Rm \setminus \Omega$ and we define $\delta : \Omega \to \R_{>0}$ by $\delta(x) = \rmdist(x,C)$.
\par 
With each $a \in \Omega$ we associate the open balls
$
\check{B}_a  = U \left( a , \frac{\tau}{8} \cdot \delta(a)\right)
$,
$
B_a  = U \left( a , \frac{\tau}{2} \cdot \delta(a)\right)
$, and
$
\hat{B}_a  = U \left( a , \frac{3 \tau}{4} \cdot \delta(a)\right)
$.
Notice that 
$
\check{B}_a \subset B_a \subset \hat{B}_a  \subset \Omega
$.
\par 
We say that $D \subset \Omega$ is scattered if the following holds:
$
(\forall a \in D)(\forall b \in D) : a \neq b \Rightarrow \|a-b\| \geq \frac{\tau}{4} \cdot \max \left\{ \delta(a) ,  \delta(b) \right\}
$.
We say that a scattered set $D$ is maximal if every scattered set containing $D$ coincides with $D$. 
It follows from Zorn's theorem that there exists a non-empty maximal scattered set $D \subset \Omega$ and we fix such $D$ for the remainder of this number.
\begin{enumerate}
\item[(A)] {\it The family $\la \check{B}_a \ra_{a \in D}$ is disjointed.}
\end{enumerate}
\par 
{\it Proof.}
For otherwise there are distinct $a,b \in D$ and $x \in \check{B}_a \cap \check{B}_b$.
Hence, $|a-b|_2 \leq |a-x|_2 + |x-b|_2 < \frac{\tau}{8} \cdot [\delta(a) +  \delta(b) ] \leq \frac{\tau}{4} \cdot \max \{ \delta(a) , \delta(b) \}$, a contradiction.\cqfd
\begin{enumerate}
\item[(B)] {\it $\ind_\Omega \leq \sum_{a \in D} \ind_{B_a}$.}
\end{enumerate}
\par 
{\it Proof.}
Assume if possible that there exists $x \in \Omega \setminus \cup_{a \in D} B_a$.
Notice that $D^* = D \cup \{x\}$ properly contains $D$, whence, the proof will be complete upon showing that $D^*$ is scattered.
\par 
For all $a \in D$ we have $|x - a |_2 \geq \frac{\tau}{2} \cdot \delta(a) \geq \frac{\tau}{4} \cdot \delta(a)$, since $x \not \in B_a$.
It remains to establish that for all $a \in D$ we have $|x-a|_2 \geq \frac{\tau}{4} \cdot \delta(x)$ as well.
Assume instead that there exists $a \in D$ such that $|x-a|_2 < \frac{\tau}{4} \cdot \delta(x)$.
Then $\frac{\tau}{2} \cdot \delta(a) \leq |x-a|_2 \leq \frac{\tau}{4} \cdot \delta(x)$, thus, $\delta(a) \leq \frac{1}{2} \cdot \delta(x)$.
Moreover, $\delta(x) \leq \delta(a) + |x-a|_2 \leq \delta(a) + \frac{\tau}{4} \cdot \delta(x) \leq \left( \frac{1}{2} + \frac{\tau}{4}\right) \cdot \delta(x)$, a contradiction.\cqfd
\begin{enumerate}
\item[(C)] {\it If $a,b \in D$ and $\hat{B}_a \cap \hat{B}_b \neq \emptyset$ then $\delta(a) > \frac{1}{7} \cdot \delta(b)$.}
\end{enumerate}
\par 
{\it Proof.}
Let $\xi \in C$ be such that $|a - \xi|_2=\delta(a)$.
Choose $x \in \hat{B}_a \cap \hat{B}_b$.
Assuming that $\delta(a) \leq \frac{1}{7} \cdot \delta(b)$ we infer that 
$
\delta(b) \leq |\xi - b |_2 \leq |\xi - a |_2 + |a-x|_2 + |x-b|_2 < \delta(a) + \frac{3\tau}{4} \cdot \delta(a) + \frac{3\tau}{4} \cdot \delta(b) 
\leq \left( \frac{7}{4} \cdot \frac{1}{7} + \frac{3}{4} \right) \cdot \delta(b),
$
a contradiction.\cqfd
\begin{enumerate}
\item[(D)] {\it If $a \in D$ and $D_a = D \cap \{ b : \hat{B}_a \cap \hat{B}_b \neq \emptyset\}$ then $\rmcard D_a \leq \bc_{\theTheorem}(m)$.}
\end{enumerate}
\par 
{\it Proof.}
{\bf (i)}
If $b \in D_a$ then $|a-b|_2 \leq 6  \tau \cdot \delta(a)$.
Indeed, picking $x \in \hat{B}_a \cap \hat{B}_b$ and referring to (C) we see that 
$
|a-b|_2 \leq |a-x|_2 + |x-b|_2 \leq \frac{3\tau}{4} \cdot \delta(a) + \frac{3\tau}{4} \cdot \delta(b) \leq \frac{3\tau}{4} \cdot \delta(a) + \frac{3\tau}{4} \cdot 7 \cdot \delta(a)
$.
\par 
{\bf (ii)}
If $b \in D_a$ then $\check{B}_b \subset B(a,7\tau \cdot \delta(a))$.
This, indeed, follows from {\bf (i)} and (C):
$
B \left(b , \frac{\tau}{8} \cdot \delta(b) \right)  \subset B \left( a, |a-b|_2 +  \frac{\tau}{8} \cdot \delta(b) \right) 
 \subset B \left(a , 6 \tau\cdot \delta(a) + \frac{7\tau}{8} \cdot \delta(a) \right)
$.
\par 
{\bf (iii)}
We infer from (C) (applied with $a$ and $b$ swapped), from (A), and from {\bf (ii)} that
$
\balpha(m) \cdot \left( \frac{\tau}{8}\right)^m (\rmcard D_a) \cdot \left( \frac{1}{7}\right)^m \cdot \delta(a)^m  \leq
\balpha(m) \cdot \left( \frac{\tau}{8}\right)^m \sum_{b \in D_a} \delta(b)^m 
=\sum_{b \in D_a} \ssfL^m \left( \check{B}_b \right) 
 = \ssfL^m \left( \bigcup_{b \in D_a} \check{B}_b \right) 
 \leq \ssfL^m \left(  B \left(a , 7 \tau \cdot \delta(a) \right) \right)
 = \balpha(m) \cdot (7\tau)^m\cdot \delta(a)^m
$.\cqfd
%
%
\par 
From now on, we choose and fix $\vphi \in \calD(\R)$ such that $\ind_{\left[ - \frac{1}{2} , \frac{1}{2} \right]} \leq \vphi \leq \ind_{\left[ - \frac{3}{4} , \frac{3}{4} \right]}$, $\rmspt \vphi \subset \left] - \frac{3}{4} , \frac{3}{4} \right[$, and $\|\vphi'\|_\infty \leq 5$.
With the maximal scattered set $D$ we associate a family $\la \phi_a \ra_{a \in D}$ of members of $\calD(\Rm)$ as follows: $\phi_a(x) = \vphi \left( \frac{|x-a|_2}{\tau \cdot \delta(a)}\right)$, $x \in \Rm$.
\begin{enumerate}
\item[(E)] {\it For all $a \in D$ we have $\ind_{B_a} \leq \phi_a$, $\rmspt \phi_a \subset \hat{B_a}$, and $|\nabla \phi_a(x)|_2 \leq \frac{9}{\tau \cdot \delta(x)} \cdot \ind_{\hat{B}_a}(x)$ for all $x \in \Omega$.}
\end{enumerate}
\par 
{\it Proof.}
The first two conclusions follow at once from the properties of $\vphi$ and the definitions of $B_a$ and $\hat{B}_a$.
With respect to the third conclusion we first notice that
$
\nabla \phi_a(x) = \frac{x-a}{\tau \cdot \delta(a) \cdot |x-a|_2} \cdot \vphi' \left( \frac{|x-a|_2}{\tau \cdot \delta(a)} \right)
$,
so that 
$
| \nabla \phi_a(x) |_2 \leq \frac{\|\vphi'\|_\infty}{\tau \cdot \delta(a)} \leq \frac{5}{\tau \cdot \delta(a)}
$.
Moreover, if $x \in \hat{B}_a$ then $\delta(x) \leq |x-a|_2 + \delta(a) \leq \frac{3\tau}{4} \cdot \delta(a) + \delta(a)$ so that $\frac{1}{\delta(a)} \leq \frac{7}{4} \cdot \frac{1}{\delta(x)}$ and the proof is complete in that case. 
If $x \not \in \hat{B}_a$ then $x \not \in \rmspt \phi_a$, therefore, $\nabla \phi_a(x) = 0$.\cqfd
\begin{enumerate}
\item[(F)] {\it The formula $\phi(x) = \sum_{a \in D} \phi_a(x)$, $x \in \Omega$, defines a function $\phi : \Omega \to \R$ with the following properties:
\begin{enumerate}
\item[(a)] $\phi \in C^\infty(\Omega)$;
\item[(b)] $\ind_\Omega \leq \phi \leq \bc_1(m) \cdot \ind_\Omega$;
\item[(c)] $|\nabla \phi(x)|_2 \leq \frac{\bc_{\theTheorem}(m)}{\tau \cdot \delta(x)}$ for all $x \in \Omega$.
\end{enumerate}
}
\end{enumerate}
\par 
{\it Proof.}
{\bf (i)}
Fix $x \in \Omega$.
There exists $a(x) \in D$ such that $x \in B_{a(x)} \subset \hat{B}_{a(x)}$, by (B).
If $y \in \hat{B}_{a(x)}$ and $a \in D \setminus D_{a(x)}$ (recall the definition of $D_{a(x)}$ from (D)) then $\rmspt(\phi_a) \cap \hat{B}_{a(x)} \subset \hat{B}_a \cap \hat{B}_{a(x)} = \emptyset$, by (E), thus, $\phi_a(y) = 0$.
Accordingly, 
\begin{equation}
\label{eq.1}
(\forall y \in \hat{B}_{a(x)}) : \phi(y) = \sum_{a \in D_{a(x)}} \phi_a(y).
\end{equation}
As $\hat{B}_{a(x)}$ is a neighborhood of $x$, $D_{a(x)}$ is finite (by (D)), and each $\phi_a \in C^\infty(\Rm)$, we conclude that $\phi|_{\hat{B}_{a(x)}} \in C^\infty(\hat{B}_{a(x)})$.
As $x \in \Omega$ is arbitrary, the proof of (a) is complete.
\par 
{\bf (ii)}
It follows from (B) and (E) that $\phi = \sum_{a \in D} \phi_a \geq \sum_{a \in D} \ind_{B_a} \geq \ind_\Omega$.
This proves the first inequality in (b).
With respect to the second inequality, we fix $x \in \Omega$, apply \eqref{eq.1} to $y=x$, and we refer to (D) to conclude that
$
\phi(x) = \sum_{a \in D_{a(x)}} \phi_a(x) \leq \rmcard D_{a(x)} \leq \bc(m)
$,
since $\phi_a(x) \leq 1$ for all $a \in D$ and all $x \in \Omega$.
\par 
{\bf (iii)}
For $x \in \Omega$ it follows from \eqref{eq.1} that 
$
\nabla \phi(x) = \sum_{a \in D_{a(x)}} \nabla \phi_a(x)
$.
Thus,
$
|\nabla \phi(x)|_2 \leq \sum_{a \in D_{a(x)}} |\nabla \phi_a(x)|_2 \leq (\rmcard D_{a(x)}) \cdot \frac{9}{\tau \cdot \delta(x)} \leq \frac{9 \cdot \bc(m)}{\tau \cdot \delta(x)}
$,
by (D) and (E).
This proves (c).\cqfd
\begin{enumerate}
\item[(G)] {\it Given $a \in D$ the formula
\begin{equation*}
\chi_a(x) = \begin{cases}
\frac{\phi_a(x)}{\phi(x)} & \text{if } x \in \Omega \\
0 & \text{if } x \not \in \Omega
\end{cases}
\end{equation*}
defines a function $\chi_a : \Rm \to \R$.
The family $\la \chi_a \ra_{a \in D}$ satisfies the following properties:
\begin{enumerate}
\item[(a)] $(\forall a \in D): \phi_a \in C^\infty(\Rm)$, $0 \leq \chi_a \leq 1$, and $\rmspt \chi_a \subset \hat{B}_a$;
\item[(b)] $\sum_{a \in D} \chi_a = \ind_\Omega$;
\item[(c)] $(\forall a \in D)(\forall x \in \Omega): |\nabla \chi_a(x)|_2 \leq \frac{\bc_{\theTheorem}(m)}{\tau \cdot \delta(x)}$.
\end{enumerate}
}
\end{enumerate}
\par 
{\it Proof.}
{\bf (i)}
Conclusion (a) readily follows from (F).
\par 
{\bf (ii)}
Conclusion (b) follows from \eqref{eq.1}:
$
\sum_{a \in D} \chi_a(x) = \frac{\sum_{a \in D} \phi_a(x)}{\phi(x)} = 1
$
for all $x \in \Omega$.
\par 
{\bf (iii)}
Fix $a \in D$ and $x \in \Omega$.
We have
$
\nabla \chi_a(x) = \frac{\phi(x) \cdot \nabla \phi_a(x) - \phi_a(x) \cdot \nabla \phi(x)}{\phi(x)^2}
$.
We then infer from (E) and (F) that
$
| \nabla \chi_a(x) |_2 \leq |\phi(x)| \cdot |\nabla \phi_a(x)|_2 + |\phi_a(x)| \cdot |\nabla \phi(x)|_2 \leq \bc_1(m) \cdot \frac{9}{\tau \cdot \delta(x)} + 1 \cdot \frac{9 \cdot \bc(m)}{\tau \cdot \delta(x)}
$.
This completes the proof of (c).\cqfd
\end{Empty}

\begin{Empty}[Members of $\MZ_{0,\#}(\Omega)$]
\label{EX.2.7}
We start by introducing a space of signed Borel measures $\MZ(\Rm) \subset M(\Rm)$ named after \cite{MEY.ZIE.77}\footnote{Meyers and Ziemer did not study signed measures but rather Radon measures.}.
With each $\mu \in M(\Rm)$ we associate 
\begin{equation*}
\|\mu\|_{\MZ} = \sup \left\{ \frac{|\mu|(B(x,r))}{r^{m-1}} : x \in \Rm \text{ and } r > 0 \right\} \in [0,\infty].
\end{equation*}
We also define
\begin{equation*}
\MZ(\Rm) = M(\Rm) \cap \{ \mu : \|\mu\|_{\MZ} < \infty \}.
\end{equation*}
Clearly, $\MZ(\Rm)$ is a vector space and $\|\cdot\|_{\MZ}$ is a norm.
It is useful to notice that if $\mu \in \MZ(\Rm)$ then $\mu^-,\mu^+ \in \MZ(\Rm)$.
\par 
For $\mu \in M(\Omega)$ and $0 < \tau \leq 1$ we define
\begin{equation*}
\boldeta_\Omega(\mu,\tau) = \sup \left\{ \frac{1}{\delta(x)} \cdot \frac{|\mu|(B(x, \tau \cdot \delta(x))}{(\tau \cdot \delta(x))^{m-1}} : x \in \Omega \right\}\in [0,\infty],
\end{equation*}
where, as in \ref{WHITNEY}, $\delta(x) = \rmdist(x,\Rm \setminus \Omega)$.
We also define the extension of $\mu$, denoted $\hat{\mu} \in M(\Rm)$, by the formula $\hat{\mu}(B) = \mu(\Omega \cap B)$, $B \in \calB(\Rm)$.
Finally we introduce the following space
\begin{equation*}
\MZ_0(\Omega) = M(\Omega) \cap \left\{ \mu : \hat{\mu} \in \MZ(\Rm) \text{ and } \lim_{\tau \to 0^+} \boldeta_\Omega(\mu,\tau) = 0 \right\}.
\end{equation*}
The following criterion for membership to $\MZ_0(\Omega)$ -- akin a vanishing Ahlfors upper regularity condition -- is useful to provide examples.
\begin{enumerate}
\item[(A)] {\it Let $\mu \in \MZ(\Rm)$ and $\veps : \ts]0,\rmdiam \Omega] \to \R_{>0}$ be nondecreasing and such that $\lim_{r \to 0^+} \veps(r) = 0$. Assume that
\begin{equation*}
(\forall x \in \rmspt \mu)(\forall 0 < r \leq \rmdiam \Omega): \frac{|\mu|(B(x,r))}{r^{m-1}} \leq \veps(r).
\end{equation*}
Then the following hold.
\begin{enumerate}
\item[(a)] If $\rmspt \mu$ is a compact subset of $\Omega$ then $\mu|_{\calB(\Omega)} \in \MZ_0(\Omega)$.
\item[(b)] If $h : \Omega \to \R$ is Borel-measurable, $\bc > 0$, and $|h| \leq \bc \cdot \rmdist(\cdot,\Rm \setminus \Omega)$ then $(\mu|_{\calB(\Omega)}) \hel h \in \MZ_0(\Omega)$.
\end{enumerate}
 }
\end{enumerate}
\par 
{\it Proof.}
Abbreviate $\nu = \mu|_{\calB(\Omega)}$.
Let $x \in \Omega$ and $0 < \tau \leq \frac{1}{2}$ and abbreviate $\rho = \tau \cdot \delta(x) \leq \frac{1}{2} \cdot \rmdiam \Omega$.
If $B(x,\rho) \cap (\rmspt \nu) = \emptyset$ then $|\nu|(B(x,\rho))=0$.
Otherwise, there exists $\xi \in B(x,\rho) \cap (\rmspt \nu)$, thus, $B(x,\rho) \subset B(\xi,\rho+|x-\xi|_2)$ so that 
$
\frac{|\nu|(B(x,\rho))}{\rho^{m-1}} \leq \left( 1 + \frac{|x-\xi|_2}{\rho}\right)^{m-1} \cdot \frac{|\mu|(B(\xi,\rho+|x-\xi|_2))}{(\rho+|x-\xi|_2)^{m-1}} \leq 2^{m-1} \cdot \veps(2\rho)
$.
Therefore, if $h$ is as in (b) then 
$
\frac{|\nu \hel h|(B(x,\rho))}{\rho^{m-1}} \leq \|h|_{B(x,\rho)}\|_\infty \cdot \frac{|\nu|(B(x,\rho))}{\rho^{m-1}} \leq \bc \cdot 2^m \cdot \delta(x) \cdot \veps(2\rho)
$.
Thus, $\boldeta_\Omega(\mu,\tau) \leq \bc \cdot 2^m \cdot \veps(\tau \cdot \rmdiam \Omega)$.
As $\tau$ is arbitrary, this completes the proof of conclusion (b).
Conclusion (a) follows from (b) applied with $h = \delta_0^{-1} \cdot \ind_{\rmspt \mu}$, where $\delta_0 = \min \{ \rmdist(x,\Rm\setminus \Omega) : x \in \rmspt \mu \}$.\cqfd
\par 
We adopt the following abbreviations for a Borel-measurable set $A \subset \Rm$ and $x \in \Rm$: considering the ratios $\frac{\ssfL^m(A \cap B(x,r))}{\balpha(m) r^m}$ wet let $\uTheta^m(A,x)$ and $\Theta^m(A,x)$ be, respectively, their superior limit and limit if exists as $r \to 0^+$.
\begin{enumerate}
\item[(B)] {\it There exists $\bc_{\theTheorem}(m) > 0$ with the following property. Let $A \subset \Rm$ be Borel-measurable and assume that for every $x \in A$ we have $\uTheta^m(A,x) > \frac{1}{4}$. Then there exists a sequence $\la B(x_j,r_j) \ra_j$ of closed balls such that $A \subset \cup_j B(x_j,r_j)$, $x_j \in A$, $\sum_j r_j^{m-1} \leq \bc_{\theTheorem}(m) \cdot \|D\ind_A\|(\Rm)$, and $\frac{\ssfL^m(A \cap B(x_j,r_j))}{\balpha(m)r_j} = \frac{1}{4}$ for all $j$.}
\end{enumerate}
\par 
{\it Proof.}
This is the boxing inequality initially proved in a different context by Gustin \cite{GUS.60}, see \cite[5.9.3]{ZIEMER} and \cite[4.5.4]{GMT}.
That the radii $r_j$ can be chosen to satisfy the last stated property follows \eg from inspecting the proof of \cite[5.9.3]{ZIEMER}; I learned this from \cite[remark 2.12]{PHU.TOR.08}.\cqfd
\par 
The proof of the following is inspired by \cite[4.7]{MEY.ZIE.77}, see also \cite[5.2.14]{ZIEMER}.
We let $\rmclos_e A = \Rm \cap \{ x : \uTheta^m(A,x) > 0\}$, $\rmint_e A = \Rm \cap \{ x : \lTheta^m(A,x)=1\}$, as well as $\rmbdry_e A = (\rmclos_e A) \setminus (\rmint_a A)$.
We also let $\rmbdry_r A$ be the reduced boundary of $A$, see \cite[3.54]{AMBROSIO.FUSCO.PALLARA} or \cite[5.5.1]{ZIEMER}.
Finally, for $\mu \in M(\Rm)$ we introduce 
\begin{equation*}
\boldeta^A(\mu,\rho) = \sup \left\{ \frac{|\mu|(B(x,r))}{r^{m-1}} : x \in A \text{ and } 0 < r \leq \rho \right\} ,
\end{equation*}
where we allow $\rho = \infty$.
Note that $\boldeta^A(\mu,\rho) \leq \|\mu\|_{\MZ}$.
\begin{enumerate}
\item[(C)] {\it Assume $A \subset \Rm$ is Borel-measurable and contained in a closed ball $B$ of radius $0 < \rho \leq \infty$ and that $\mu \in \MZ(\Rm)$. Then $|\mu|(A) \leq \bc_{\theTheorem}(m) \cdot \boldeta^B(\mu,2\rho) \cdot \|D\ind_A\|(\Rm)$.}
\end{enumerate}
\par 
{\it Proof.}
{\bf (i)}
We may as well assume that $\|D\ind_A\|(\Rm) < \infty$ for otherwise there is nothing to prove.
Define $C = A \cap \{ x : \uTheta^m(A,x) \geq \frac{1}{2} \}$.
Since $\rmbdry_r A \subset C$, by \cite[3.61]{AMBROSIO.FUSCO.PALLARA}, we infer that $A \setminus C \subset (\rmbdry_e A) \setminus (\rmbdry_r A)$.
Therefore, $\ssfH^{m-1}(A \setminus C) = 0$, according to \cite[3.61]{AMBROSIO.FUSCO.PALLARA}.
In turn, $|\mu|(A) = |\mu|(C)$.
\par 
{\bf (ii)}
Since $\ssfL^m(A \setminus C)=0$, (B) applies to $C$ (in place of $A$).
Letting $\la B(x_j,r_j) \ra_j$ be the corresponding sequence of closed balls we observe that for every $j$ we have $\balpha(m) \cdot r_j^m \leq 4 \cdot \ssfL^m(A) \leq 4 \cdot \balpha(m) \cdot \rho^m$.
Thus, $r_j \leq 2 \cdot \rho$ so that $|\mu|(C) \leq \sum_j |\mu|(B(x_j,r_j)) \leq \boldeta^B(\mu,2\rho) \cdot \sum_j r_j^{m-1}$ and the conclusion ensues.\cqfd
\par 
Below $BV_\#(\Omega)$ consists of all those $u \in BV(\Omega)$ such that $\int_\Omega u \,d\ssfL^m = 0$.
\begin{enumerate}
\item[(D)] {\it Let $\Omega$ be a connected, bounded $BV$-extension set.
\begin{enumerate}
\item[(a)] If $\mu \in M(\Omega)$ and $\hat{\mu} \in \MZ(\Rm)$ then for every $u \in BV_\#(\Omega)$ the precise representative $u^*$ is $\mu$-integrable and the formula 
\begin{equation*}
BV_\#(\Omega) \to \R : u \mapsto \int_\Omega u^* \,d\mu
\end{equation*}
defines a linear form such that
\begin{equation*}
\left| \int_\Omega u^* \,d\mu \right| \leq \bc_{\theTheorem}(m) \cdot \bc_\rmext(\Omega) \cdot (1 + \bc_{1^*}(\Omega)) \cdot \|\hat{\mu}\|_{\MZ} \cdot \|Du\|(\Omega).
\end{equation*}
\item[(b)] If, moreover, $\mu \in \MZ_0(Z)$ then for every $\veps > 0$ there exists $\theta > 0$ such that for all $u \in BV_\#(\Omega)$ we have 
\begin{equation*}
\left| \int_\Omega u^* \,d\mu \right| \leq \theta \cdot \|u\|_{L_1(\Omega)} + \veps \cdot \|Du\|(\Omega).
\end{equation*}
\end{enumerate}
}
\end{enumerate}
\par 
{\it Proof.}
{\bf (i)}
We recall that if $u \in BV(\Omega)$ then for $\ssfH^{m-1}$-a.e. $x \in \Omega$ the limit of averages $\lim_{r \to 0^+} \dashint_{B(x,r)} u \,d\ssfL^m$ exists, see \eg \cite[theorem 5.20]{EVANS.GARIEPY.2}, and we then define $u^*(x)$ to coincide with this limit when it exists and to be equal to 0 otherwise.
\par 
{\bf (ii)}
Fix $u \in BV_\#(\Omega)$, let $\hat{u} \in BV(\Rm)$ be an extension of $u$ as in \ref{bv.ext} and note that $(\hat{u})^*|_\Omega = u^*$ everywhere on $\Omega$.
Our first goal is to establish that $(\hat{u})^*$ is $|\hat{\mu}|$-summable, as this will immediately imply that $u^*$ is $|\mu|$-summable and $\int_\Omega u^* \, d\mu = \int_{\Rm} (\hat{u})^* \, d\hat{\mu}$.
\par 
For each $y > 0$, abbreviate $A_y = \Rm \cap \{ x : |(\hat{u})^*(x)| > y \}$ and observe that $|\hat{\mu}|(A_y) \leq \bc_{\theTheorem}(m) \cdot \|\hat{\mu}\|_{\MZ} \cdot \|D\ind_{A_y}\|(\Rm)$, by (C). 
Therefore, it ensues from the coarea formula that
\begin{multline*}
\int_{\Rm} |(\hat{u})^*| \, d|\hat{\mu}| = \int_0^\infty |\hat{\mu}|(A_y) \, dy \leq
\bc(m) \cdot \|\hat{\mu}\|_{\MZ} \cdot \int_0^\infty  \|D\ind_{A_y}\|(\Rm) \,dy \\
= \bc_{\theTheorem}(m) \cdot \|\hat{\mu}\|_{\MZ} \cdot \|D\hat{u}\|(\Rm)
\leq \bc_{\theTheorem}(m) \cdot \|\hat{\mu}\|_{\MZ} \cdot \bc_\rmext(\Omega) \cdot (1 + \bc_{1^*}(\Omega)) \cdot \|Du\|(\Omega).
\end{multline*}
\par 
{\bf (iii)}
We have shown in {\bf (ii)} that the integral $\int_\Omega u^* \, d\mu$ is convergent for each $u \in BV_\#(\Omega)$.
Note that $(u_1+u_2)^* = u_1^* + u_2^*$ and $(t \cdot u)^* = t \cdot u^*$ both $\ssfH^{m-1}$-a.e. in $\Omega$.
Therefore, the linearity of $u \mapsto \int_\Omega u^* \, d\mu$ is a consequence of the observation that $|\hat{\mu}| \ll \ssfH^{m-1}$, \ie $|\hat{\mu}|(Z)=0$ whenever $\ssfH^{m-1}(Z)=0$ and $Z$ is Borel-measurable.
To see this, let $\veps > 0$ and choose a sequence $\la B_j \ra_j$ of closed balls such that $Z \subset \cup_j B_j$ and $\sum_j (\rmdiam B_j)^{m-1} < \veps$.
Thus $|\hat{\mu}|(Z) \leq \sum_j |\hat{\mu}|(B_j) \leq 2^{m-1} \cdot \|\mu\|_{\MZ} \cdot \veps$.
This completes the proof of conclusion (a).
\par 
{\bf (iv)}
We turn to proving conclusion (b).
For each $0 < \tau \leq 1$ we abbreviate $\bar{\boldeta}_\Omega(\mu,\tau) = \sup \{ \boldeta_\Omega(\mu,t) : 0 < t \leq \tau \}$.
Let $\veps > 0$.
Since $\mu \in \MZ_0(\Omega)$, there exists $0 < \tau \leq \frac{1}{6}$ such that $\bar{\boldeta}_\Omega(\mu,6\tau) < \frac{\veps}{2 \cdot \bc_{\theTheorem}(m) \cdot (\rmdiam \Omega)} =: \hat{\veps}$.
We apply \ref{WHITNEY} to $\Omega$ and this $\tau$.
We choose a numbering $\la a_j \ra_j$ of the ensuing scattered subset $D$ of $\Omega$ and we abbreviate $\chi_j = \chi_{a_j}$ and $\hat{B}_j = \hat{B}_{a_j}$. 
Thus, $\la \chi_j \ra_j$ is a partition of unity of $\Omega$ satisfying the estimates in \ref{WHITNEY}(G), $\rmspt \chi_j \subset \hat{B}_j$, and $\la \hat{B}_j \ra_j$ is a locally finite covering of $\Omega$ satisfying the uniform intersection bound in \ref{WHITNEY}(D). 
\par 
{\bf (v)}
Fix $u \in BV_\#(\Omega)$.
We claim that
\begin{equation*}
\int_\Omega u^* \,d\mu = \sum_{j \in \N} \int_\Omega \chi_j \cdot u^* \,d\mu.
\end{equation*}
This follows from an application of the dominated convergence theorem, since for all $n \in \N$ we have $\left| \sum_{j=1}^n \chi_j \cdot u^* \right| \leq \left( \sum_{j \in \N} \chi_j \right)  \cdot |u^*| = |u^*|$ and $|u^*| \in L_1(\Omega,|\mu|)$, by {\bf (i)}.
\par 
{\bf (vi)}
We now fix $j \in \N$.
Define $A_{j,y} = \Omega \cap \{ x : \chi_j(x) \cdot |u^*(x)| > y \}$, $y > 0$.
Notice that $A_{j,y} \subset \rmspt \chi_j \subset \hat{B}_j \subset \Omega$ and let $\rho_j = \frac{3\tau}{4}\cdot \delta(a_j)$ be the radius of $\hat{B}_j$.
It follows from (C) that $|\hat{\mu}|(A_{j,y}) \leq \bc_{\theTheorem}(m) \cdot \boldeta^{\hat{B}_j}(\hat{\mu},2\rho_j) \cdot \|D\ind_{A_{j,y}}\|(\Rm)$.
In this expression we may replace the two occurrences of $\hat{\mu}$ by $\mu$ and we may replace $\Rm$ by $\Omega$.
The first claim follows from the fact that $A_{y,j} \subset \Omega$ and that, as we shall see below, if $x \in \hat{B}_j$ and $0 < r r2\rho_j$ then $B(x,r) \subset \omega)$.
The second claim follows from the fact that $\rmspt \|D\ind_{A_{j,y}}\| \subset \Omega$.
Thus,
\begin{equation*}
|\mu|(A_{j,y}) \leq \bc_{\theTheorem}(m) \cdot \boldeta^{\hat{B}_j}(\mu,2\rho_j) \cdot \|D\ind_{A_{j,y}}\|(\Omega).
\end{equation*}
\par 
{\bf (vii)}
Here, $j$ is still fixed and we ought to bound $\boldeta^{\hat{B}_j}(\mu,2\rho_j)$ from above.
Let $x \in \hat{B}_j$ and $0 < r \leq 2 \rho_j = \frac{3\tau}{2} \cdot \delta(a_j)$.
Notice that $|\delta(x) - \delta(a_j)| \leq |x-a|_2 < \rho_j \leq \frac{3}{4} \cdot \delta(a_j)$.
In particular, $\frac{1}{4} \cdot \delta(a_j) \leq \delta(x) \leq 2 \cdot \delta(a_j)$.
Let $t > 0$ be such that $t \cdot \delta(x) = r$.
We have $ t = \frac{r}{\delta(x)} \leq \frac{\frac{3\tau}{2} \cdot \delta(a_j)}{\frac{1}{4} \cdot \delta(a_j)} = 6 \tau \leq 1$.
It follows that
\begin{equation*}
\frac{|\mu|(B(x,r))}{r^{m-1}} = \frac{|\mu|(B(x,t\cdot \delta(x)))}{(t \cdot \delta(x))^{m-1}} \leq \boldeta_\Omega(\mu,t) \cdot \delta(x) \leq \bar{\boldeta}_\Omega(\mu,6\tau) \cdot  2\cdot \delta(a_j)
\leq \hat{\veps} \cdot 2 \cdot \delta(a_j).
\end{equation*}
From the arbitrariness of $x$ and $r$ we conclude that
$
\boldeta^{\hat{B}_j}(\mu,2\rho_j) \leq \hat{\veps} \cdot 2 \cdot \delta(a_j)
$.
It subsequently follows from {\bf (vi)} that
\begin{equation*}
|\mu|(A_{j,y}) \leq \hat{\veps} \cdot 2 \cdot \bc_{\theTheorem}(m) \cdot \delta(a_j) \cdot \|D\ind_{A_{j,y}}\|(\Omega).
\end{equation*}
\par 
{\bf (viii)}
Here $j$ is still fixed.
We refer to {\bf (vii)} and the coarea formula to infer that 
\begin{equation*}
\begin{split}
\int_\Omega \chi_j \cdot |u^*| \,d\mu = \int_0^\infty |\mu|(A_{j,y}) \,dy &
\leq \hat{\veps} \cdot 2 \cdot \bc_{\theTheorem}(m) \cdot \delta(a_j) \cdot \int_0^\infty \|D\ind_{A_{j,y}}\|(\Omega)\,dy \\
& = \hat{\veps} \cdot 2 \cdot \bc_{\theTheorem}(m) \cdot \delta(a_j) \cdot \|D(\chi_j \cdot u)\|(\Omega).
\end{split}
\end{equation*}
It follows from \cite[5.3.2]{ZIEMER} that 
\begin{equation*}
\|D(\chi_j \cdot u)\|(\Omega) \leq \int_\Omega |\nabla \chi_j|_2 \cdot |u| \,d\ssfL^m + \int_\Omega \chi_j \, d\|Du\|.
\end{equation*}
We bound from above the first term of the right-hand side by means of \ref{WHITNEY}(G)(c):
\begin{multline*}
\int_\Omega|\nabla \chi_j|_2 \cdot |u| \,d\ssfL^m = \int_{\hat{B}_j} |\nabla \chi_j|_2 \cdot |u| \,d\ssfL^m \leq \int_{\hat{B}_j} \frac{\bc_{\ref{WHITNEY}(m)}}{\tau \cdot \delta(x)} \cdot |u(x)| \, d\ssfL^m(x) \\
\leq \frac{4 \cdot \bc_{\ref{WHITNEY}}(m)}{\tau \cdot \delta(a_j)} \cdot \int_{\hat{B}_j} |u| \,d\ssfL^m.
\end{multline*}
Finally,
\begin{multline*}
\int_\Omega \chi_j \cdot |u^*| \, d\ssfL^m  \\  \leq \left( \frac{\hat{\veps}}{\tau} \right) \cdot 8 \cdot \bc_{\ref{WHITNEY}(m)} \cdot \bc_{\theTheorem}(m) \cdot \int_{\hat{B}_j} |u| \,d\ssfL^m
+ \hat{\veps} \cdot 2 \cdot \bc_{\theTheorem}(m) \cdot (\rmdiam \Omega) \cdot \int_{\Omega} \chi_j  \,d\|Du\|.
\end{multline*}
\par 
{\bf (ix)}
Referring to {\bf (viii)}, summing over $j$ as in {\bf (v)}, applying two more times the dominated convergence theorem, and referring to \ref{WHITNEY}(D) yields
\begin{equation*}
\begin{split}
\left| \int_\Omega u^* \,d\mu \right| & \leq \sum_{j \in \N} \int_\Omega \chi_j \cdot |u^*| \,d|\mu| \\
& \leq
\left( \frac{\hat{\veps}}{\tau} \right) \cdot 8 \cdot \bc_{\ref{WHITNEY}(m)} \cdot \bc_{\theTheorem}(m) \cdot \int_{\Omega} \left( \sum_{j \in \N} \chi_j \right) \cdot |u| \,d\ssfL^m \\
& \qquad \qquad + \hat{\veps} \cdot 2 \cdot \bc_{\theTheorem}(m) \cdot (\rmdiam \Omega) \cdot \int_{\Omega} \left( \sum_{j \in \N}\chi_j \right)  \,d\|Du\| \\
& \leq \theta \cdot \int_\Omega |u| \,d\ssfL^m + \veps \cdot \|Du\|(\Omega),
\end{split}
\end{equation*}
where $\theta = \left( \frac{\veps}{\tau}\right)\cdot \left( \frac{4 \cdot \bc_{\ref{WHITNEY}}(m)^2}{\rmdiam \Omega} \right)$.
The proof is complete.\cqfd
\par 
We are now ready to introduce the class of examples:
\begin{equation*}
\MZ_{0,\#}(\Omega) = \MZ_0(\Omega) \cap \{ \mu : \mu(\Omega) = 0 \}.
\end{equation*}
\begin{enumerate}
\item[(E)] {\it If $\Omega$ is a bounded, connected $BV$-extension set and $\mu \in \MZ_{0,\#}(\Omega)$ then there exists $v \in C_0(\Omega;\Rm)$ such that $\bdiv v = F$ and $\|v\|_\infty \leq 2 \cdot \bc_{\theTheorem}(m) \cdot \bc_\rmext(\Omega) \cdot (1 + \bc_{1^*}(\Omega)) \cdot \|\hat{\mu}\|_{\MZ}$. }
\end{enumerate}
\par 
{\it Proof.}
{\bf (i)}
Let $u \in BV(\Omega)$.
Observe that for all $y \in \R$ the function $u^*-y\cdot \ind_\Omega$ is $\mu$-integrable, since this is the case for $y = (u)_\Omega$, by (D)(a), and, $\mu$ being a signed measure, any constant is $\mu$-integrable.
Furthermore, $\int_\Omega (u - (u)_\Omega \cdot \ind_\Omega) \, d\mu$ is independent of $y$, since $\mu(\Omega) = 0$.
Thus, the formula $\la [u] , F_\mu \ra = \int_\Omega (u - (u)_\Omega \cdot \ind_\Omega) \, d\mu$ well-defines a linear functional $BV_{\rmcst}(\Omega) \to \R$.
\par 
{\bf (ii)}
It follows from (D)(b) and \ref{5.1.3}(F) that $F_\mu$ is $\calM_{TV}$-continuous.
\par 
{\bf (iii)}
The conclusion follows from the $\calM_{TV}$-continuity of $F_\mu$ in conjunction with theorem \ref{5.1.7}.
The upper bound on $\|v\|_\infty$ is a consequence of the bound on $\vvvert F_\mu \vvvert$ that ensues from (D)(a).\cqfd
\par 
We end this number with applications of (E).
Let $m-1 < d < m$ and let $C \subset \Omega$ be a compact set which is upper Ahlfors-regular of dimension $d$, \ie there exists $\bc > 0$ such that $\ssfH^d(C \cap B(x,r)) \leq \bc \cdot r^d$ for all $x \in C$ and all $0 < r \leq \rmdiam C$.
Choose a Borel-measurable set $A \subset C$ such that $\ssfH^d(A) = \frac{1}{2} \cdot \ssfH^d(C)$ and define $\mu = \ssfH^d \hel \left( \ind_A - \ind_{C \setminus A}\right)$.
It follows from (A) that $\mu \in \MZ_{0,\#}(\Omega)$ so that (E) applies to $\mu$.
Of course, we could have let $\mu = \ssfH^d \hel h$ for any $h \in L_1(\ssfH^d \hel C)$ such that $\int_C h \,d\ssfH^d = 0$.
In this case, the existence of $v \in C(\Omega;\Rm)$ such that $\rmdiv v = F$ in the sense of distributions can be obtained using tools of classical harmonic analysis
Namely, let $v = K * \mu$, where $K$ is the Riesz kernel, and show that $v$ is H\"older continuous based on $\boldeta_\Omega(\mu,r) \lesssim r^{d+1-m}$.
\par 
The interest of (E) is precisely in the case when it becomes problematic to establish the continuity of $K * \mu$, namely when $\boldeta_\Omega(\mu,r) \lesssim \veps(r) \cdot r^{1-m}$ for a positive function $\veps$ whose decay can be stipulated as slow as one wishes.
One can construct such $\mu$ supported in sets $C$ that are either Cantor topologically or a variant of the Koch curve so that the Hausdorff dimension of $C$ is $m-1$ but $\ssfH^{m-1} \hel C$ misses (as barely as we wish) to be $\sigma$-finite. 
Specifically, if $m=2$ and $\veps : \R_{>0} \to \R_{>0}$ is increasing and $\lim_{r \to 0^+} \veps(r) = 0$ then exists a compact set $C \subset \R^2$ which is homeomorphic to $\{0,1\}^{\N}$ (resp. homeomorphic to $[0,1]$) and $\bc^{-1} \cdot \veps(r) \cdot r \leq \ssfH^1(C \cap B(x,r)) \leq \bc \cdot \veps(r) \cdot r$ for all $x \in C$ and $0 < r \leq \rmdiam C$, where $\bc \geq 1$.
\end{Empty}


\printbibliography





\end{document}